\documentclass[10pt]{amsart}
\usepackage{a4wide,color}
\usepackage{mathrsfs}
\usepackage{amsbsy}
\usepackage{amstext}
\usepackage{amsthm}
\usepackage{amssymb}
\usepackage[active]{srcltx}
\usepackage{ulem}
\usepackage{enumerate}
\usepackage{cancel}

\makeatletter
\numberwithin{equation}{section}
\numberwithin{figure}{section}
\allowdisplaybreaks
\let \pa \partial

\let \eps \varepsilon

\newtheorem{theorem}{Theorem}
\newtheorem{lemma}[theorem]{Lemma}
\newtheorem{proposition}[theorem]{Proposition}
\newtheorem{remark}[theorem]{Remark}

\let \de=\delta
\let \eps=\varepsilon

\let \pa=\partial

\begin{document}
\title[Boltzmann equation]{Space-time behavior of the solution to the
Boltzmann equation with soft potentials}
\author{Yu-Chu Lin}
\address{Yu-Chu Lin, Department of Mathematics, National Cheng Kung
University, Tainan, Taiwan}
\email{yuchu@mail.ncku.edu.tw }
\author{Ming-Jiea Lyu}
\address{Ming-Jiea Lyu, Department of Mathematics, National Cheng Kung
University, Tainan, Taiwan}
\email{mingjiealyu@gmail.com}
\author{Haitao Wang}
\address{Haitao Wang, Institute of Natural Sciences and School of
Mathematical Sciences, LSC-MOE, Shanghai Jiao Tong University, Shanghai,
China}
\email{haitallica@sjtu.edu.cn}
\author{Kung-Chien Wu}
\address{Kung-Chien Wu, Department of Mathematics, National Cheng Kung
University, Tainan, Taiwan and National Center for Theoretical Sciences,
National Taiwan University, Taipei, Taiwan}
\email{kungchienwu@gmail.com}
\date{\today }

\begin{abstract}
In this paper, we get the quantitative space-time behavior of the full
Boltzmann equation with soft potentials ($-2<\gamma <0$) in the close to
equilibrium setting, under some velocity decay assumption, but without any
Sobolev regularity assumption on the initial data. {We find that both the
large time and spatial behaviors depend on the velocity decay of the initial
data and the exponent $\gamma $.}
The key step in our strategy is to obtain the $L^{\infty }$ bound of a
suitable weighted full Boltzmann equation directly, rather than using  Green's function and Duhamel's principle to construct the pointwise structure of the solution
as in \cite{[LiuYu]}. This provides a new thinking in the related study.
\end{abstract}

\maketitle





\section{Introduction}

\subsection{The models}

Consider the following Boltzmann equation:
\begin{equation}
\left \{
\begin{array}{l}
\displaystyle \partial _{t}F+\xi \cdot \nabla _{x}F=Q(F,F)\text{,} \\[4mm]
\displaystyle F(0,x,\xi )=F_{0}(x,\xi )\text{,}%
\end{array}%
\right.  \label{bot.1.a}
\end{equation}%
where $F(t,x,\xi )$ is the distribution function of the particles at time $%
t>0$, position $x=(x_{1},x_{2},x_{3})\in {\mathbb{R}}^{3}$ and microscopic
velocity $\xi =(\xi _{1},\xi _{2},\xi _{3})\in {\mathbb{R}}^{3}$. The
left-hand side of this equation models the transport of particles and the
operator on the right-hand side models the effect of collisions on the
transport with
\begin{equation*}
Q(F,G)=\frac{1}{2}\int_{{\mathbb{R}}^{3}\times S^{2}}|\xi -\xi _{\ast
}|^{\gamma }B(\vartheta )\left \{ F_{\ast }^{\prime }G^{\prime }+G_{\ast
}^{\prime }F^{\prime }-F_{\ast }G-G_{\ast }F\right \} d\xi _{\ast }d\omega
\text{.}
\end{equation*}%
Here the usual conventions, i.e., $F=F(t,x,\xi )$, $F_{\ast }=F(t,x,\xi
_{\ast })$,$\ F^{\prime }=F(t,x,\xi ^{\prime })$ and $F_{\ast }^{\prime
}=F(t,x,\xi _{\ast }^{\prime })$, are used.

In this paper, we consider the soft potentials ($-2<\gamma <0$); and $%
B(\vartheta )$ satisfies the Grad's angular cutoff assumption
\begin{equation}
0<B(\vartheta )\leq C \left|\cos \vartheta \right|,  \notag
\end{equation}%
for some constant $C>0$. Moreover, the post-collisional velocities satisfy
\begin{equation*}
\xi ^{\prime }=\xi -[(\xi -\xi _{\ast })\cdot \omega ]\omega \text{,}\quad
\xi _{\ast }^{\prime }=\xi +[(\xi -\xi _{\ast })\cdot \omega ]\omega \text{,}
\end{equation*}%
and $\vartheta $ is defined by
\begin{equation*}
\cos \vartheta =\frac{|(\xi -\xi _{\ast })\cdot \omega |}{|\xi -\xi _{\ast }|%
}\text{.}
\end{equation*}

It is well known that the global Maxwellians are steady-state solutions to
the Boltzmann equation (\ref{bot.1.a}). Therefore, it is natural to consider
the Boltzmann equation (\ref{bot.1.a}) around a global Maxwellian
\begin{equation*}
\mathcal{M}(\xi )=\frac{1}{(2\pi )^{3/2}}\exp \left( \frac{-|\xi |^{2}}{2}%
\right) \text{,}
\end{equation*}%
with the standard perturbation $f(t,x,\xi )$ to $\mathcal{M}$ as
\begin{equation*}
F=\mathcal{M}+\mathcal{M}^{1/2}f\text{,}\quad F_{0}=\mathcal{M}+\eta
\mathcal{M}^{1/2}f_{0}\text{,}
\end{equation*}%
where $\eta >0$ is sufficiently small. After substituting $F$ and $F_{0}$
into (\ref{bot.1.a}), the equation for the perturbation $f$ is
\begin{equation}
\left \{
\begin{array}{l}
\displaystyle \partial _{t}f+\xi \cdot \nabla _{x}f=Lf+\Gamma (f,f)\text{,}%
\vspace {3mm}
\\[4mm]
\displaystyle f(0,x,\xi )=\eta f_{0}(x,\xi )=\frac{F_{0}-\mathcal{M}}{\sqrt{%
\mathcal{M}}}\text{,}%
\end{array}%
\right.  \label{bot.1.d}
\end{equation}%
where $L=-\nu \left( \xi \right)+K$ is the linearized collision operator
defined as
\begin{equation*}
Lf=\mathcal{M}^{-1/2}\left[ Q(\mathcal{M},\mathcal{M}^{1/2}f)+Q(\mathcal{M}%
^{1/2}f,\mathcal{M})\right] \, \text{,}
\end{equation*}%
and $\Gamma $ is the nonlinear operator defined as
\begin{equation*}
\Gamma (f,f)=\mathcal{M}^{-1/2}Q(\mathcal{M}^{1/2}f,\mathcal{M}^{1/2}f)\text{%
.}
\end{equation*}%
It is well-known that the null space of $L$ is a five-dimensional vector
space with the orthonormal basis $\{ \chi _{i}\}_{i=0}^{4}$, where
\begin{equation*}
Ker(L)=\left \{ \chi _{0},\chi _{i},\chi _{4}\right \} =\left \{ \mathcal{M}%
^{1/2},\  \xi _{i}\mathcal{M}^{1/2},\  \frac{1}{\sqrt{6}}(|\xi |^{2}-3)%
\mathcal{M}^{1/2}\right \} \text{,}\quad i=1\text{,}\ 2\text{,}\ 3\text{.}
\end{equation*}%
Based on this property, we can introduce the macro-micro decomposition: let $%
\mathrm{P}_{0}$ be the orthogonal projection with respect to the $L_{\xi
}^{2}$ inner product onto $\mathrm{Ker}(L)$, and $\mathrm{P}_{1}\equiv
\mathrm{Id}-\mathrm{P}_{0}$.

\subsection{Notation}

Before the presentation of the main theorem, let us define some notations used in
this paper. We denote $\left \langle \xi \right \rangle ^{s}=(1+|\xi
|^{2})^{s/2}$ and $\left \langle \xi \right \rangle _{D}^{s}=(D^{2}+|\xi
|^{2})^{s/2}$, where $D>0$, $s\in {\mathbb{R}}$. For the microscopic
variable $\xi $, we denote
\begin{equation*}
|g|_{L_{\xi }^{q}}=\Big(\int_{{\mathbb{R}}^{3}}|g|^{q}d\xi \Big)^{1/q}\text{
if }1\leq q<\infty \text{,}\quad \quad |g|_{L_{\xi }^{\infty }}=\sup_{\xi
\in {\mathbb{R}}^{3}}|g(\xi )|\text{,}
\end{equation*}%
and the weighted norms can be defined by
\begin{equation*}
|g|_{L_{\xi ,\beta }^{q}}=\Big(\int_{{\mathbb{R}}^{3}}\left \vert \left
\langle \xi \right \rangle ^{\beta }g\right \vert ^{q}d\xi \Big)^{1/q}\text{
if }1\leq q<\infty \text{,}\quad \quad |g|_{L_{\xi ,\beta }^{\infty
}}=\sup_{\xi \in {\mathbb{R}}^{3}}\left \vert \left \langle \xi \right
\rangle ^{\beta }g(\xi )\right \vert \text{,}
\end{equation*}%
and
\begin{equation*}
|g|_{L_{\xi }^{\infty }(m)}=\sup_{\xi \in {\mathbb{R}}^{3}}\left \{ |g(\xi
)|m(\xi)\right \} \text{,}
\end{equation*}%
where $\beta \in {\mathbb{R}}$ and $m$ is a weight function. The $L_{\xi
}^{2}$ inner product in ${\mathbb{R}}^{3}$ will be denoted by $\big<\cdot
,\cdot \big>_{\xi }$, i.e.,
\begin{equation*}
\left \langle f,g\right \rangle _{\xi }=\int f(\xi )\overline{g(\xi )}d\xi
\text{.}
\end{equation*}%
For the Boltzmann equation, the natural norm in $\xi $ is $|\cdot
|_{L_{\sigma }^{2}}$, which is defined as
\begin{equation*}
|g|_{L_{\sigma }^{2}}^{2}=\left|\left \langle \xi \right \rangle ^{\frac{%
\gamma }{2}}g\right|_{L_{\xi }^{2}}^{2}\text{.}
\end{equation*}%
For the space variable $x$, we have similar notations, namely,
\begin{equation*}
|g|_{L_{x}^{q}}=\left( \int_{{\mathbb{R}^{3}}}|g|^{q}dx\right) ^{1/q}\text{
if }1\leq q<\infty \text{,}\quad \quad |g|_{L_{x}^{\infty }}=\sup_{x\in {%
\mathbb{R}^{3}}}|g(x)|\text{.}
\end{equation*}%
Furthermore, we define the high order Sobolev norm: let $s\in {\mathbb{N}}$
and define
\begin{equation*}
\left \vert g\right \vert _{H_{\xi }^{s}}=\sum_{|\alpha |\leq s}\left \vert
\partial _{\xi }^{\alpha }g\right \vert _{L_{\xi }^{2}}\text{,\  \  \  \  \  \ }%
\left \vert g\right \vert _{H_{x}^{s}}=\sum_{|\alpha |\leq s}\left \vert
\partial _{x}^{\alpha }g\right \vert _{L_{x}^{2}}\text{,}
\end{equation*}%
where $\alpha $ is any multi-index with $|\alpha |\leq s$.

Finally, with $\mathcal{X}$ and $\mathcal{Y}$ being normed spaces, we define
\begin{equation*}
\left \Vert g\right \Vert _{\mathcal{XY}}=\left \vert \left \vert g\right
\vert _{\mathcal{Y}}\right \vert _{\mathcal{X}}\text{,}
\end{equation*}%
and for simplicity, we denote
\begin{equation*}
\Vert g\Vert _{L^{2}}=\Vert g\Vert _{L_{\xi }^{2}L_{x}^{2}}=\left( \int_{{%
\mathbb{R}^{3}}}|g|_{L_{x}^{2}}^{2}d\xi \right) ^{1/2}\text{.}
\end{equation*}

The domain decomposition plays an important role in our analysis, so we
introduce a cut-off function $\chi :{\mathbb{R}}\rightarrow {\mathbb{R}}$,
which is a smooth non-increasing function, $\chi (s)=1$ for $s\leq 1$, $\chi
(s)=0$ for $s\geq 2$ and $0\leq \chi \leq 1$. Moreover, we define $\chi
_{R}(s)=\chi (s/R)$ for positive $R$.

For simplicity of notations, hereafter, we abbreviate \textquotedblleft {\ $%
\leq C$} \textquotedblright \ to \textquotedblleft {\ $\lesssim $ }%
\textquotedblright , where $C$ is a positive constant depending only on
fixed numbers.

\subsection{Review of previous works and main result}

In the literature, there are a lot of works concerning the large time
behavior of the solution for various models of the Boltzmann equation, such
as the hard sphere, hard potentials and soft potentials.

In the literature, there
are several energy methods for the study of the Boltzmann equations near
Maxwellian in the whole space. The direct energy method through the
micro-macro decomposition was initiated by Liu-Yu \cite{[LiuYu3]} and
developed by Liu-Yang-Yu \cite{[LiuYangYu]} and Guo \cite{[Guo]}
independently in two different ways. In between there is another energy
method introduced by Kawashima \cite{[Kawashima]}, which is based on
constructing compensating function for the thirteen moments of Boltzmann
equation. Under some suitable Sobolev regularity assumptions on the initial
condition, combining energy estimate with the spectrum method \cite%
{[Duan-Ukai--],[Ellis],[Ukai]} or compensating function method \cite{[Chai],
[Kawashima],[Yu]}, one can get the time decay rate. For more details, the
reader is referred to the reference therein. In addition, people are aware
that the large time behavior is governed by the long wave part in terms of
the Fourier variables of the linearized equation, no matter for the hard
sphere, hard potential or soft potential.

For Boltzmann equation in a bounded domain, an important $L^{2}-L^{\infty}$ theory was developed in \cite{[Guo1]} to obtain the global existence and the exponential decay rate of the solution around a global Maxwellian for  hard potentials associated with appropriate boundary conditions.
See also \cite{[Liu-Yang]} for its extension to soft potential in a bounded domain, where a sub-exponential decay rate is obtained. One is also referred to \cite{[EGKM],[GKTT],[KimLee]} for the recent advancements of this theory.

On the other hand, it is noted that the inter-molecular potential can
influence the spatially asymptotic behavior for the stationary linearized
Boltzmann half space problem (i.e., the Milne problem). Indeed, \cite{[BCN]}
obtained exponential decay for the hard sphere case, \cite{[ALW],
[GolsePoupand]} obtained arbitrary polynomial decay for the hard potential
upon assuming corresponding velocity weights on boundary data, and \cite%
{[cc]} obtained sub-exponential decay for the hard potential upon assuming
Gaussian weight.
{Thus, it would be interesting to investigate the space-time behaviors of
the solutions for different potentials.
To this end, the pointwise approach has been initiated by \cite{[LiuYu],
[LiuYu2], [LiuYu1]} for the full nonlinear hard sphere case, and then
generalized by \cite{[LeeLiuYu], [LinWangWu-2], [LinWangWu]} to hard and
soft potential cases on the linear level. }

However, the nonlinear problems for hard potential and soft potential have
not been settled. In this paper, the spatially asymptotic behavior and
uniform time decay for fully nonlinear Boltzmann equation with soft
potential are established. The similar result for hard potential is also
stated without proof, which is actually easier. It is worth  mentioning that
our results do not require any Sobolev regularity of the initial data. The
main results are stated as follows.

\begin{theorem}[The large time behavior for $-2<\protect \gamma <0$]
\label{prop: nonlinear}Let $-2<\gamma <0$, $0<p_{1}\leq 2$, $p_{2}>3/2$, $%
\hat{\varepsilon}\geq 0$ sufficiently small, and $j>0$ sufficiently large.
Assume that the initial data $\eta f_{0}$ satisfies $f_{w_{3}0}=w_{3}f_{0}%
\in L_{\xi ,p_{2}+3j}^{\infty }(L_{x}^{1}\cap L_{x}^{\infty })$ where $%
w_{3}=e^{\hat{\varepsilon}\left \langle \xi \right \rangle ^{p_{1}}}\,(\hat{\varepsilon}\geq0)$, and $%
\eta >0$ is sufficiently small. Then there is a unique solution $f$ to $(\ref%
{bot.1.d})$ in $L_{\xi ,p_{2}+2j}^{\infty }(e^{\hat{\varepsilon}\left \langle
\xi \right \rangle ^{p_{1}}})L_{x}^{2}\cap L_{\xi ,p_{2}+2j}^{\infty }(e^{%
\hat{\varepsilon}\left \langle \xi \right \rangle ^{p_{1}}})L_{x}^{\infty }$
with
\begin{eqnarray}
\left \Vert w_{3}f(t)\right \Vert _{L_{\xi ,p_{2}}^{\infty }L_{x}^{2}} &\leq
&\eta C_{1}(1+t)^{-\frac{3}{4}}\left( \left \Vert w_{3}f_{0}\right \Vert
_{L_{\xi ,p_{2}+2j}^{\infty }L_{x}^{1}}+\left \Vert w_{3}f_{0}\right \Vert
_{L_{\xi ,p_{2}+2j}^{\infty }L_{x}^{\infty }}\right) \text{,}%
\vspace {3mm}
\label{eq: N. 1. aa} \\
\left \Vert w_{3}f(t)\right \Vert _{L_{\xi ,p_{2}}^{\infty }L_{x}^{\infty }}
&\leq &\eta C_{2}(1+t)^{-\frac{3}{2}}\left( \left \Vert w_{3}f_{0}\right \Vert
_{L_{\xi ,p_{2}+3j}^{\infty }L_{x}^{1}}+\left \Vert w_{3}f_{0}\right \Vert
_{L_{\xi ,p_{2}+3j}^{\infty }L_{x}^{\infty }}\right) \text{,}
\label{eq: N. 1. a}
\end{eqnarray}%
\begin{eqnarray}
\left \Vert w_{3}f(t)\right \Vert _{L_{\xi ,p_{2}+2j}^{\infty }L_{x}^{2}}
&\leq &\eta \bar{C}_{1}\left( \left \Vert w_{3}f_{0}\right \Vert _{L_{\xi
,p_{2}+2j}^{\infty }L_{x}^{1}}+\left \Vert w_{3}f_{0}\right \Vert _{L_{\xi
,p_{2}+2j}^{\infty }L_{x}^{\infty }}\right) \text{,}%
\vspace {3mm}
\label{eq: N. 1. b} \\
\left \Vert w_{3}f(t)\right \Vert _{L_{\xi ,p_{2}+2j}^{\infty }L_{x}^{\infty
}} &\leq &\eta \bar{C}_{2}\left( \left \Vert w_{3}f_{0}\right \Vert _{L_{\xi
,p_{2}+2j}^{\infty }L_{x}^{1}}+\left \Vert w_{3}f_{0}\right \Vert _{L_{\xi
,p_{2}+2j}^{\infty }L_{x}^{\infty }}\right) \text{,}  \label{eq: N. 1. bb}
\end{eqnarray}%
for some positive constants $C_{1}$,$C_{2}$, $\bar{C}_{1}$, $\bar{C}_{2}$
depending on $\gamma $, $\hat{\varepsilon}$, $p_{1}$, $p_{2}$, and $j$.
\end{theorem}

We here mention that whenever $\hat{\varepsilon}=0$, $f_{w_{3}}=f$ is the
solution to the equation (\ref{bot.1.d}).

\begin{theorem}[The spatially asymptotic behavior for $-2<\protect \gamma <0$]

\label{thm:main}Let $-2<\gamma <0$ and let $f$ be a solution to the
Boltzmann equation $(\ref{bot.1.d})$ with initial data $\eta f_{0}$, where $%
f_{0}$ is compactly supported in the $x$-variable for all $\xi $:%
\begin{equation*}
f_{0}\left( x,\xi \right) \equiv 0\text{ for }\left \vert x\right \vert \geq 1%
\text{, }\xi \in \mathbb{R}^{3}\text{,}
\end{equation*}%
and $\eta >0$ is sufficiently small.$%
\vspace {3mm}%
$\newline
(i) Let $0<\varsigma \ll 1$. Suppose that $\left \vert f_{0}\right \vert
_{L_{x}^{\infty }}\in L_{\xi ,p+\beta +3j}^{\infty }$ for some $p\geq 1$, $%
\beta >3/2$, and $j>0$ large enough. Then:$%
\vspace {3mm}%
$\newline
If $-1<\gamma <0$, there exists $M>0$ such that for $\left \langle
x\right \rangle >2Mt$,
\begin{equation*}
\left \vert f(t,x)\right \vert _{L_{\xi ,\beta }^{\infty }}\lesssim \eta
(1+t)^{2}(\left \langle x\right \rangle +Mt)^{\frac{-p}{1-\gamma }}\left \Vert
\left \langle \xi \right \rangle ^{p+\beta +3j}f_{0}\right \Vert _{L_{\xi
}^{\infty }L_{x}^{\infty }}\text{.}
\end{equation*}%
If $\gamma =-1$, there exists $M>0$ such that for $\left \langle
x\right \rangle >2Mt$,
\begin{equation*}
\left \vert f(t,x)\right \vert _{L_{\xi ,\beta }^{\infty }}\lesssim \eta
(1+t)^{2+\varsigma }(\left \langle x\right \rangle +Mt)^{\frac{-p}{1-\gamma }%
}\left \Vert \left \langle \xi \right \rangle ^{p+\beta +3j}f_{0}\right \Vert
_{L_{\xi }^{\infty }L_{x}^{\infty }}\text{.}
\end{equation*}%
If $-2<\gamma <-1$, there exists $M>0$ such that for $\left \langle
x\right \rangle >2Mt$,%
\begin{equation*}
\left \vert f(t,x)\right \vert _{L_{\xi ,\beta }^{\infty }}\lesssim \eta
(1+t)^{7+\frac{5}{\gamma }}(\left \langle x\right \rangle +Mt)^{\frac{-p}{%
1-\gamma }}\left \Vert \left \langle \xi \right \rangle ^{p+\beta
+3j}f_{0}\right \Vert _{L_{\xi }^{\infty }L_{x}^{\infty }}\text{.}
\end{equation*}%
\newline
\newline
(ii) Let $0<\varsigma \ll 1$. Suppose that $\left \vert f_{0}\right \vert
_{L_{x}^{\infty }}\in L_{\xi }^{\infty }(e^{\hat{\eps}\left \langle \xi
\right \rangle ^{p}}\left \langle \xi \right \rangle ^{p+\beta +3j})$ for some $%
0<p\leq 2$, $\beta >3/2$, $\hat{\eps}>0$ sufficiently small, and $j>0$ large
enough. Then:$%
\vspace {3mm}%
$\newline
If $-1<\gamma <0$, there exist $M>0$ and $0<\eps<\hat{\eps}$ such that for $%
\left \langle x\right \rangle >2Mt$,%
\begin{equation*}
\left \vert f(t,x)\right \vert _{L_{\xi ,\beta }^{\infty }}\lesssim \eta
(1+t)^{2}e^{-\eps(\left \langle x\right \rangle +Mt)^{\frac{p}{p+1-\gamma }%
}}\Vert e^{\hat{\varepsilon}\left \langle \xi \right \rangle ^{p}}\left \langle
\xi \right \rangle ^{p+\beta +3j}f_{0}\Vert _{L_{\xi }^{\infty
}L_{x}^{\infty }}\text{.}
\end{equation*}%
If $\gamma =-1$, there exist $M>0$ and $0<\eps<\hat{\eps}$ such that for $%
\left \langle x\right \rangle >2Mt$,%
\begin{equation*}
\left \vert f(t,x)\right \vert _{L_{\xi ,\beta }^{\infty }}\lesssim \eta
(1+t)^{2+\varsigma }e^{-\eps(\left \langle x\right \rangle +Mt)^{\frac{p}{%
p+1-\gamma }}}\Vert e^{\hat{\varepsilon}\left \langle \xi \right \rangle
^{p}}\left \langle \xi \right \rangle ^{p+\beta +3j}f_{0}\Vert _{L_{\xi
}^{\infty }L_{x}^{\infty }}\text{.}
\end{equation*}%
If $-2<\gamma <-1$, there exist $M>0$ and $0<\eps<\hat{\eps}$ such that for $%
\left \langle x\right \rangle >2Mt$,%
\begin{equation*}
\left \vert f(t,x)\right \vert _{L_{\xi ,\beta }^{\infty }}\lesssim \eta
(1+t)^{7+\frac{5}{\gamma }}e^{-\eps(\left \langle x\right \rangle +Mt)^{\frac{p%
}{p+1-\gamma }}}\Vert e^{\hat{\varepsilon}\left \langle \xi \right \rangle
^{p}}\left \langle \xi \right \rangle ^{p+\beta +3j}f_{0}\Vert _{L_{\xi
}^{\infty }L_{x}^{\infty }}\text{.}
\end{equation*}
\end{theorem}

In fact, we have also established the corresponding results for
the full nonlinear Boltzmann equation with hard potential cases (i.e., $%
0\leq \gamma <1$). The proof in that case is almost the same as in the soft
potential one and most of the parallel lemmas can be obtained more easily.
To avoid a lengthy discussion, we focus on the soft potential case in this
paper and just state the results for the hard potential as below.

\begin{theorem}[The large time behavior for $0\leq \protect \gamma <1$]
Let $0\leq \gamma <1$, $0<p_{1}\leq 2$, $p_{2}>3/2$, and let $\hat{%
\varepsilon}\geq 0$ be sufficiently small. Assume that the initial $f_{0}$
satisfies $f_{w_{3}0}=w_{3}f_{0}\in L_{\xi ,p_{2}+\gamma }^{\infty
}(L_{x}^{1}\cap L_{x}^{\infty })$ where $w_{3}=e^{\hat{\varepsilon}%
\left
\langle \xi \right \rangle ^{p_{1}}}$, and $\eta >0$ is sufficiently
small. Then there exists a unique solution $f$ to $(\ref{bot.1.d})$ in $%
L_{\xi ,p_{2}+\gamma }^{\infty }(e^{\hat{\varepsilon}\left \langle \xi
\right
\rangle ^{p_{1}}})L_{x}^{2}\cap L_{\xi ,p_{2}+\gamma }^{\infty }(e^{%
\hat{\varepsilon}\left \langle \xi \right \rangle ^{p_{1}}})L_{x}^{\infty }$
with
\begin{equation}
\Vert f_{w_{3}}\Vert _{L_{\xi ,p_{2}+\gamma }^{\infty }L_{x}^{\infty
}}\lesssim \eta (1+t)^{-3/2}\left( \Vert f_{w_{3}0}\Vert _{L_{\xi
,p_{2}+\gamma }^{\infty }L_{x}^{\infty }}+\Vert f_{w_{3}0}\Vert _{L_{\xi
,p_{2}}^{\infty }L_{x}^{1}}\right) \text{,}
\end{equation}%
and
\begin{equation}
\Vert f_{w_{3}}\Vert _{L_{\xi ,p_{2}+\gamma }^{\infty }L_{x}^{2}}\lesssim
\eta (1+t)^{-3/4}\left( \Vert f_{w_{3}0}\Vert _{L_{\xi ,p_{2}+\gamma
}^{\infty }L_{x}^{\infty }}+\Vert f_{w_{3}0}\Vert _{L_{\xi ,p_{2}}^{\infty
}L_{x}^{1}}\right) \text{.}\,
\end{equation}
\end{theorem}

We here mention again that $f_{w_{3}}=f$ is the solution to the equation (%
\ref{bot.1.d}) whenever $\hat{\varepsilon}=0$.

\begin{theorem}[The spatially asymptotic behavior for $0\leq \protect \gamma %
<1$]
\label{thm:main copy(1)}Let $0\leq \gamma <1$ and let $f$ be a solution to
the Boltzmann equation $(\ref{bot.1.d})$ with initial data $\eta f_{0}$,
where $f_{0}$ is compactly supported in the $x$-variable for all $\xi $:%
\begin{equation*}
f_{0}\left( x,\xi \right) \equiv 0\text{ for }\left \vert x\right \vert \geq
1\text{, }\xi \in \mathbb{R}^{3}\text{,}
\end{equation*}%
and $\eta >0$ is sufficiently small.$%
\vspace {3mm}%
$\newline
(i) Suppose that $\left \vert f_{0}\right \vert _{L_{x}^{\infty }}\in $ $%
L_{\xi ,p+\beta +\gamma /2}^{\infty }$ for some $p\geq 1$ and $\beta >3/2$.
Then there exists $M>0$ such that for $\left \langle x\right \rangle >2Mt$,
\begin{equation*}
\left \vert f(t,x)\right \vert _{L_{\xi ,\beta }^{\infty }}\lesssim \eta
(1+t)^{1/2}(\left \langle x\right \rangle +Mt)^{\frac{-p}{1-\gamma }}\Vert
f_{0}\Vert _{L_{\xi ,p+\beta +\gamma /2}^{\infty }L_{x}^{\infty }}\text{.}
\end{equation*}%
(ii) Suppose that $\left \vert f_{0}\right \vert _{L_{x}^{\infty }}\in
L_{\xi }^{\infty }(e^{\hat{\varepsilon}\left \langle \xi \right \rangle
^{p}}\left \langle \xi \right \rangle ^{p+\beta +\gamma /2})$ for some $%
0<p\leq 2$, $\beta >3/2$, $\hat{\varepsilon}>0$ sufficiently small. Then
there exist $M>0$ and $0<\varepsilon <\hat{\varepsilon}$ such that for $%
\left \langle x\right \rangle >2Mt$,
\begin{equation*}
\left \vert f(t,x)\right \vert _{L_{\xi ,\beta }^{\infty }}\lesssim \eta
(1+t)^{1/2}e^{-\varepsilon (\left \langle x\right \rangle +Mt)^{\frac{p}{%
p+1-\gamma }}}\Vert e^{\hat{\varepsilon}\left \langle \xi \right \rangle
^{p}}\left \langle \xi \right \rangle ^{p+\beta +\gamma /2}f_{0}\Vert
_{L_{\xi }^{\infty }L_{x}^{\infty }}\text{.}
\end{equation*}
\end{theorem}

\subsection{Method of proof and plan of the paper}

In order to study the spatially asymptotic behavior of the solution $f$ to
the full nonlinear Boltzmann equation (\ref{bot.1.d}), the following weight
functions will be taken into account (which are motivated by the linear
results \cite{[LinWangWu-2],[LinWangWu]}): \newline
\newline
\textbf{Weight function $w_{1}$.} Let $\delta >0$ be sufficiently small, $D$%
, $M\geq 1$ sufficiently large and $p\geq 1$. Define $w_{1}$ as
\begin{equation}
w_{1}\left( t,x,\xi \right) =5\left( \delta \left( \left \langle x\right
\rangle -Mt\right) \right) ^{\frac{p}{1-\gamma }}\left( 1-\chi \left( \frac{%
\delta \left( \left \langle x\right \rangle -Mt\right) }{\left \langle \xi
\right \rangle _{D}^{1-\gamma }}\right) \right) +3\left \langle \xi \right
\rangle _{D}^{p}\chi \left( \frac{\delta \left( \left \langle x\right
\rangle -Mt\right) }{\left \langle \xi \right \rangle _{D}^{1-\gamma }}%
\right) \text{.}
\end{equation}%
\textbf{Weight function $w_{2}$.} Let $\epsilon $, $\delta >0$ be
sufficiently small, $M>0$ sufficiently large and $0<p\leq 2$. Define $w_{2}$
as
\begin{equation}
w_{2}(t,x,\xi )=e^{\epsilon \rho (t,x,\xi )}  \label{eq: w}
\end{equation}%
$\,$ with
\begin{equation*}
\rho \left( t,x,\xi \right) =5\left( \delta \left( \left \langle x\right
\rangle -Mt\right) \right) ^{\frac{p}{p+1-\gamma }}\left( 1-\chi \left(
\frac{\delta \left( \left \langle x\right \rangle -Mt\right) }{\left \langle
\xi \right \rangle ^{p+1-\gamma }}\right) \right) +3\left \langle \xi \right
\rangle ^{p}\chi \left( \frac{\delta \left( \left \langle x\right \rangle
-Mt\right) }{\left \langle \xi \right \rangle ^{p+1-\gamma }}\right) \text{.}
\end{equation*}%
\textbf{Weight function $w_{3}$.} Let $\hat{\varepsilon}\geq 0$ be
sufficiently small and $0<p_{1}\leq 2$. Define $w_{3}$ as
\begin{equation}
w_{3}(\xi )=e^{\hat{\varepsilon}\left \langle \xi \right \rangle ^{p_{1}}}%
\text{.}  \label{w0}
\end{equation}%
Here we mention that the coefficients $5$ and $3$ can be replaced
by other combinations of  positive constants $a$ and $b$ with $a\geq b>0$, meeting the desired
requirement $\partial _{t}w_{i}\leq 0$ ($i=1$, $2$). Now, let $%
f_{w_{i}}=w_{i}f$, $i=1$, $2$. Then $f_{w_{i}}$ ($i=1$, $2$) satisfies the
equation
\begin{equation}
\left \{
\begin{array}{l}
\displaystyle \partial _{t}f_{w_{i}}+\xi \cdot \nabla _{x}f_{w_{i}}-\frac{%
\left( \partial _{t}w_{i}+\xi \cdot \nabla _{x}w_{i}\right) }{w_{i}}%
f_{w_{i}}=L_{w_{i}}f_{w_{i}}+\Gamma _{w_{i}}(f_{w_{i}},f)\, \text{,} \\[4mm]
\displaystyle f_{w_{i}}(0,x,\xi )=\eta w_{i}(0,x,\xi )f_{0}(x,\xi )\equiv
\eta f_{w_{i}0}(x,\xi )\text{.}%
\end{array}%
\right.  \label{eq: add x weight of Boltzmann}
\end{equation}%
Here $L_{w_{i}}f_{w_{i}}=\left( w_{i}Lw_{i}^{-1}\right) f_{w_{i}}=\left(
-\nu \left( \xi \right)+K_{w_{i}}\right) f_{w_{i}}$, $\Gamma
_{w_{i}}(f_{w_{i}},f)=w_{i}\Gamma (w_{i}^{-1}f_{w_{i}},f)$.


Therefore, in order to get the spatially asymptotic behavior of the solution $f$ to \eqref{bot.1.d}, the key step of our strategy is to obtain the $L^{\infty }$ bound
for the solution $u$ to the weighted linearized Boltzmann equation with a
source term as below:%
\begin{equation}
\left \{
\begin{array}{l}
\displaystyle \partial _{t}u+\xi \cdot \nabla _{x}u-\frac{\left( \partial
_{t}w_{i}+\xi \cdot \nabla _{x}w_{i}\right) }{w_{i}}u=L_{w_{i}}u+\Gamma
_{w_{i}}(g_{i},h_{i})\, \text{,} \\[4mm]
\displaystyle u(0,x,\xi )=\eta w_{i}(0,x,\xi )f_{0}(x,\xi )\equiv \eta
f_{w_{i}0}(x,\xi )\text{,}%
\end{array}%
\right.   \label{weighted equation with source term}
\end{equation}%
where $g_{i}$ and $h_{i}$ are prescribed, $i=1$, $2$. With the sharp estimate of $f$, a priori estimate of $f_{w_i}$, and substituting $%
	g_{i}=f_{w_{i}}$, $h_{i}=f$, we can obtain the $L^\infty$ bound of $
	f_{w_{i}}$.

The procedure relies on  large time decay of the solution  $f$ to nonlinear problem for
initial data living in $\xi $-weighted space.
 Using compensating function
methods and the wave-remainder decomposition, we first obtain the large time
behavior of the linearized equation in normed spaces $L_{\xi
}^{2}L_{x}^{2}$ and $L_{\xi }^{2}L_{x}^{\infty }$. By applying Ukai's
bootstrap argument to the integral equation,
we improve
the estimates to the weighted  spaces $L_{\xi }^{\infty }\left( e^{%
\hat{\varepsilon}\left \langle \xi \right \rangle ^{p_{1}}}\left \langle \xi
\right \rangle ^{p_{2}}\right) L_{x}^{2}$, $L_{\xi }^{\infty }\left( e^{\hat{%
\varepsilon}\left \langle \xi \right \rangle ^{p_{1}}}\left \langle \xi
\right \rangle ^{p_{2}}\right) L_{x}^{\infty }$, etc. Furthermore, given a
source term $\Gamma \left( h_{1},h_{2}\right) $ with prescribed time decay
(see \eqref{Cond-hi}), we establish the large time behavior for
inhomogeneous equation, through Duhamel principle in
terms of Green's function and damped transport operator, together with
refined estimates for $\Gamma \left( h_{1},h_{2}\right) $. {\ The estimate
for the nonlinear term $\Gamma $ is more exquisite in the soft potential
case ($-2<\gamma <0\,$).
In particular, in Lemma \ref{Lemma-Gamma-sup}, the extra
decay $\left( -1\right) $ in \eqref{Gamma-sup-2} is important in
studying
the linearized equation with a source term $\Gamma \left( h_{1},h_{2}\right) $. With the
help of an extra interpolation inequality (Lemma \ref{lemma: eps2^n}), it
enables us to get the time decay of $\Gamma \left( h_{1},h_{2}\right) $ from
$h_{1}$ and $h_{2}$ through these refined estimates for $\Gamma$. The large time behavior of the nonlinear problem \eqref{bot.1.d} then follows from an iteration scheme.
{
}
Due to the interpolation argument, we only get the large time in the $%
L_{x}^{\infty }$  at the rate of $\left( 1+t\right) ^{-\frac{3}{4}}$ at first glance, then we
recover the rate of $\left( 1+t\right) ^{-\frac{3}{2}}$
by a bootstrap process (see Section \ref{proof: nonlinear}).

Next we turn to the $L^{\infty }$ bound of the solution $u$ to the
equation (\ref{weighted equation with source term}). We combine the
wave-remainder decomposition, the energy estimate, and the regularization
estimates to conclude the proof.
In the sequel, we explain the idea in more details. The wave-remainder
decomposition is based on a Picard-type iteration, which is manipulated to
construct the increasingly regular particle-like waves.
The pointwise estimate for the wave part is obtained from the property of
the time-dependent damped transport operator (defined in \eqref{eq: def of
S(t)} and \eqref{eq: def of S(t;s)}. It is noted damped transport equation
in weighted equation is not an autonomous differential equation, so one
needs to consider the evolution operator rather than simple semi-group.
The energy estimate is used to analyze the remainder term. In the course of this
procedure, the regularization estimate (see Lemma \ref{lemma: H2 estimate of
h6}) plays a crucial role, which allows us to show the remainder becomes
regular, and in turn do the higher order energy estimate. Also thanks to the
regularization estimate, we obtain the pointwise estimate without regularity
assumption on the initial data.
Finally, we bootstrap the remainder part from $L_{\xi }^{2}$ to $L_{\xi
,\beta }^{\infty }$ $\left( \beta >3/2\right) $ so that the velocity norms
of the remainder part and the wave part become consistent.

Here we would like to remark three points in the proof: (1) due to the weaker damping term (i.e., $-2<\gamma<0$), one needs to trade off velocity decay for time decay either to get the decay of $f$ or to control the growth of $u$, so the delicate velocity-weight-gaining properties of $K_{w_i}, \Gamma, \Gamma_{w_i}$ (see Lemmas \ref{lemma: estimate of weighted-kw}, \ref{Lemma-Gamma-sup} and \ref%
{lemma: inf estimate of gamma_w} ) are fully used in the estimates; (2) although the bootstrap from $L^2_\xi$ to $L^\infty_\xi$ is frequently used in the proof, it is not obvious the integral operator $K$ owns this property if $-2<r\leq -3/2$. Thanks to Riesz-Thorin interpolation theorem, $K$ has $L_{\xi
	,1-\gamma }^{4}\text{-}L_{\xi }^{2}$ estimate in the case $-2<\gamma
\leq -3/2$ (see Lemma \ref{L4-L2}). Associated with $L_{\xi ,7/4-\gamma
}^{\infty }\text{-}L_{\xi }^{4}$ estimate, $K$ eventually has $L_{\xi
	,7/4-\gamma }^{\infty }\text{-}L_{\xi }^{4}\text{-}L_{\xi }^{2}$ estimate in
the case $-2<\gamma \leq -3/2$; (3) In the proof of Lemma \ref{lemma: H2 estimate of h6}, it reveals that the mixture of the
two operators $\mathbb{S}_{w_{i}}^{t}$ and $K_{w_{i}}$ can transport the regularity in the
microscopic velocity $\xi $ induced by $K_{w_{i}}$ to the regularity in the space $x$.
It is worth mentioning that $K_{w_{i}}$ is an integral operator from $L_{\xi }^{2}$ to $H_{\xi }^{1}$  only when $\gamma >-2$ (see Lemma \ref{lemma: estimate of weighted-kw}), this is the reason why we restrict ourselves to the case $\gamma >-2$. The removal of this restriction is left to the future.

{Lastly, we want to compare the method in this paper with those in \cite%
{[LiuYu]}, which studied the nonlinear Boltzmann equation with hard
sphere, and gave the only space-time pointwise structure result of the
nonlinear solution so far. There, it is crucial that the estimates of linear
problem can be obtained in the same weighted space as the initial data,
which allows for the nonlinear iteration, then the authors achieve the
estimate of nonlinear problem. However, for hard potential, as well as soft
potential, this methodology does not work since one needs extra weights for
maintaining the space-time structure even for the linear equation (see \cite%
{[LeeLiuYu],[LinWangWu-2],[LinWangWu]}). As a comparison, to obtain the
spatially asymptotic behavior of the nonlinear equation, we circumvent the
difficulty of nonlinear iteration due to mismatch of velocity weight on the
linear level, and directly study the $L_{\xi ,\beta }^{\infty }L_{x}^{\infty
}$ estimate of the solution $f_{w_{i}}$ to the weighted full Boltzmann
equation (\ref{eq: add x weight of Boltzmann}). This is a new idea in the related
studies.

The rest of this paper is organized as follows: We first present some basic
properties concerning the operators $L$, $\Gamma $ and the corresponding
weighted operators $L_{w_{i}}$ ($i=1$, $2$, $3$) and $\Gamma _{w_{i}}$ ($i=1$%
, $2$, $3$) in Section \ref{pre}. After that, we study the weighted
linearized Boltzmann equation with a source term in Section \ref{WR}. With
these preparations and the large time behavior (Theorem \ref{prop: nonlinear}%
), we demonstrate the spatially asymptotic behavior (Theorem \ref{thm:main}) in
Section \ref{Proof: thm:main}, and postpone the proof of Theorem \ref{prop:
nonlinear} until Section \ref{proof: nonlinear}.

\section{Preliminaries}

\label{pre}

As mentioned in the Introduction section, we will study the weighted
equation (\ref{eq: add x weight of Boltzmann}) first. Before preceding, some
basic properties concerning the operators $L$, $\Gamma $ and the
corresponding weighted operators $L_{w_{i}}$ ($i=1$, $2$, $3$) and $\Gamma
_{w_{i}}$ ($i=1$, $2$, $3$), need to be studied. The linearized collision
operator $L$, which was analyzed extensively by Grad \cite{[Grad]}, consists
of a multiplicative operator $\nu (\xi )$ and an integral operator $K$:
\begin{equation}
Lf=-\nu (\xi )f+Kf\text{,}  \label{linearized-collision}
\end{equation}%
where

\begin{equation*}
\nu (\xi )=\int B(\vartheta )|\xi -\xi _{\ast }|^{\gamma }\mathcal{M}(\xi
_{\ast })d\xi _{\ast }d\omega \text{,}
\end{equation*}%
and
\begin{equation}
Kf=-K_{1}f+K_{2}f  \label{defn:K}
\end{equation}%
is defined as \cite{[Grad]}:
\begin{equation*}
K_{1}f=\int B(\vartheta )|\xi -\xi _{\ast }|^{\gamma }\mathcal{M}^{1/2}(\xi )%
\mathcal{M}^{1/2}(\xi _{\ast })f(\xi _{\ast })d\xi _{\ast }d\omega \text{,}
\end{equation*}%
\begin{align*}
K_{2}f& =\int B(\vartheta )|\xi -\xi _{\ast }|^{\gamma }\mathcal{M}%
^{1/2}(\xi _{\ast })\mathcal{M}^{1/2}(\xi ^{\prime })f(\xi _{\ast }^{\prime
})d\xi _{\ast }d\omega \\
& \quad +\int B(\vartheta )|\xi -\xi _{\ast }|^{\gamma }\mathcal{M}%
^{1/2}(\xi _{\ast })\mathcal{M}^{1/2}(\xi _{\ast }^{\prime })f(\xi ^{\prime
})d\xi _{\ast }d\omega \text{.}
\end{align*}%
To begin with, we present a number of properties and estimates of the
operators $L$, $\nu (\xi )$ and $K$, which can be found in \cite{[Caflisch],
[Chen], [GolsePoupand], [Grad], [LinWangWu], [Strain-Guo]}.

\begin{lemma}
\label{basic} Let $-2<\gamma <0$. For any $g\in L_{\sigma }^{2}$, we have
the coercivity of the linearized collision operator $L$, that is, there
exists a positive constant $\nu _{0}$ such that
\begin{equation}
\left \langle g,Lg\right \rangle _{\xi }\leq -\nu _{0}\left \vert \mathrm{P}%
_{1}g\right \vert _{L_{\sigma }^{2}}^{2}\text{.}  \label{coercivity}
\end{equation}%
For the multiplicative operator $\nu (\xi )$, there are positive constants $%
\nu _{0}$ and $\nu _{1}$ such that
\begin{equation}
\nu _{0}\left \langle \xi \right \rangle ^{\gamma }\leq \nu (\xi )\leq \nu
_{1}\left \langle \xi \right \rangle ^{\gamma }\text{,}  \label{nu-gamma}
\end{equation}%
and for each multi-index $\alpha $,
\begin{equation}
|\partial _{\xi }^{\alpha }\nu (\xi )|\lesssim \left \langle \xi \right
\rangle ^{\gamma -|\alpha |}\text{.}
\end{equation}%
For the integral operator $K$,
\begin{equation*}
Kf=-K_{1}f+K_{2}f=\int_{{\mathbb{R}}^{3}}-k_{1}(\xi ,\xi _{\ast })f(\xi
_{\ast })d\xi _{\ast }+\int_{{\mathbb{R}}^{3}}k_{2}(\xi ,\xi _{\ast })f(\xi
_{\ast })d\xi _{\ast }\text{,}
\end{equation*}%
the kernels $k_{1}(\xi ,\xi _{\ast })$ and $k_{2}(\xi ,\xi _{\ast })$
satisfy
\begin{equation*}
k_{1}(\xi ,\xi _{\ast })\lesssim |\xi -\xi _{\ast }|^{\gamma }\exp \left \{ -%
\frac{1}{4}\left( |\xi |^{2}+|\xi _{\ast }|^{2}\right) \right \} \text{,}
\end{equation*}%
and%
\begin{equation*}
k_{2}(\xi ,\xi _{\ast })=a\left( \xi ,\xi _{\ast },\kappa \right) \exp
\left( -\frac{(1-\kappa )}{8}\left[ \frac{\left( \left \vert \xi \right
\vert ^{2}-\left \vert \xi _{\ast }\right \vert ^{2}\right) ^{2}}{\left
\vert \xi -\xi _{\ast }\right \vert ^{2}}+\left \vert \xi -\xi _{\ast
}\right \vert ^{2}\right] \right) \text{,}
\end{equation*}%
for any $0<\kappa <1$, together with
\begin{equation*}
a(\xi ,\xi _{\ast },\kappa )\leq \left \{
\begin{array}{ll}
C_{\kappa }|\xi -\xi _{\ast }|^{-1}(1+|\xi |+|\xi _{\ast }|)^{\gamma -1}%
\text{,}%
\vspace {3mm}
& \text{if }-1<\gamma <0\text{,} \\
C_{\kappa }|\xi -\xi _{\ast }|^{-1}\left \vert \ln |\xi -\xi _{\ast }|\right
\vert (1+|\xi |+|\xi _{\ast }|)^{\gamma -1}\text{,}%
\vspace {3mm}
& \text{if }\gamma =-1\text{,} \\
C_{\kappa }|\xi -\xi _{\ast }|^{\gamma }(1+|\xi |+|\xi _{\ast }|)^{\gamma -1}%
\text{,} & \text{if }-2<\gamma <-1\text{,}%
\end{array}%
\right.
\end{equation*}%
and their derivatives as well have similar estimates, i.e.,
\begin{equation*}
|\nabla _{\xi }k_{1}(\xi ,\xi _{\ast })|\text{, }|\nabla _{\xi _{\ast
}}k_{1}(\xi ,\xi _{\ast })|\lesssim |\xi -\xi _{\ast }|^{\gamma -1}\exp
\left \{ -\frac{1}{4}\left( |\xi |^{2}+|\xi _{\ast }|^{2}\right) \right \}
\text{,}
\end{equation*}%
\begin{equation*}
|\nabla _{\xi }k_{2}(\xi ,\xi _{\ast })|\text{, }|\nabla _{\xi _{\ast
}}k_{2}(\xi ,\xi _{\ast })|\lesssim |\nabla _{\xi }a\left( \xi ,\xi _{\ast
},\kappa \right) |\exp \left( -\frac{(1-\kappa )}{8}\left[ \frac{\left(
\left \vert \xi \right \vert ^{2}-\left \vert \xi _{\ast }\right \vert
^{2}\right) ^{2}}{\left \vert \xi -\xi _{\ast }\right \vert ^{2}}+\left
\vert \xi -\xi _{\ast }\right \vert ^{2}\right] \right) \text{,}
\end{equation*}%
with
\begin{equation*}
|\nabla _{\xi }a\left( \xi ,\xi _{\ast },\kappa \right) |\leq \left \{
\begin{array}{ll}
C_{\kappa }\frac{\left \vert \xi \right \vert }{|\xi -\xi _{\ast }|^{2}}%
(1+|\xi |+|\xi _{\ast }|)^{\gamma -1}\text{,}%
\vspace {3mm}
& \text{if }-1<\gamma <0\text{,} \\
C_{\kappa }\frac{\left \vert \xi \right \vert }{|\xi -\xi _{\ast }|^{2}}%
\left \vert \ln |\xi -\xi _{\ast }|\right \vert (1+|\xi |+|\xi _{\ast
}|)^{\gamma -1}\text{,}%
\vspace {3mm}
& \text{if }\gamma =-1\text{,} \\
C_{\kappa }\frac{\left \vert \xi \right \vert }{|\xi -\xi _{\ast
}|^{1-\gamma }}(1+|\xi |+|\xi _{\ast }|)^{\gamma -1}\text{,} & \text{if }%
-2<\gamma <-1\text{.}%
\end{array}%
\right.
\end{equation*}
\end{lemma}

Immediately from Lemma \ref{basic}, we have the following lemma.

\begin{lemma}
\label{basic2}Let $-2<\gamma <0$ and $\tau \in \mathbb{R}$. Then
\begin{equation}
\int_{\mathbb{R}^{3}}\left \vert k\left( \xi ,\xi _{\ast }\right) \right
\vert \left \langle \xi _{\ast }\right \rangle ^{\tau }d\xi _{\ast }\lesssim
\left \langle \xi \right \rangle ^{\tau +\gamma -2}\, \text{,}\quad \int_{%
\mathbb{R}^{3}}\left \vert k\left( \xi ,\xi _{\ast }\right) \right \vert
\left \langle \xi \right \rangle ^{\tau }d\xi \lesssim \left \langle \xi
_{\ast }\right \rangle ^{\tau +\gamma -2}\text{,}  \label{k-integ}
\end{equation}%
\begin{equation}
\int_{\mathbb{R}^{3}}\left \vert \nabla _{\xi }k\left( \xi ,\xi _{\ast
}\right) \right \vert \left \langle \xi _{\ast }\right \rangle ^{\tau }d\xi
_{\ast }\lesssim \left \langle \xi \right \rangle ^{\tau +\gamma -1}\, \text{%
,}\quad \int_{\mathbb{R}^{3}}\left \vert \nabla _{\xi }k\left( \xi ,\xi
_{\ast }\right) \right \vert \left \langle \xi \right \rangle ^{\tau }d\xi
\lesssim \left \langle \xi _{\ast }\right \rangle ^{\tau +\gamma -1}\text{.}
\label{grad-k-integ 1}
\end{equation}%
Consequently, we have
\begin{equation}
\left \vert Kg\right \vert _{H_{\xi }^{1}}\lesssim \left \vert g\right \vert
_{L_{\xi ,\gamma -1}^{2}}\,,\quad \left \vert K\nabla _{\xi }g\right \vert
_{L_{\xi }^{2}}\lesssim \left \vert g\right \vert _{L_{\xi ,\gamma -1}^{2}}%
\text{,}
\end{equation}%
and
\begin{equation}
|Kg|_{L_{\xi ,\tau +2-\gamma }^{q}}\lesssim |g|_{L_{\xi ,\tau }^{q}}\text{, }%
1\leq q\leq \infty \text{.}  \label{K-Lp}
\end{equation}
\end{lemma}

\begin{lemma}
\label{basic22}Let $\tau \in {\mathbb{R}}$. Then if $-3/2<\gamma <0$,%
\begin{equation}
\int_{\mathbb{R}^{3}}(1+|\xi _{\ast }|)^{\tau }k^{2}\left( \xi ,\xi _{\ast
}\right) d\xi _{\ast }\lesssim \left \langle \xi \right \rangle ^{\tau
+2\gamma -3}\text{,}  \label{k^2-integ}
\end{equation}%
and if $-2<\gamma \leq -3/2$,
\begin{equation}
\int_{\mathbb{R}^{3}}(1+|\xi _{\ast }|)^{\tau }k^{q}\left( \xi ,\xi _{\ast
}\right) d\xi _{\ast }\lesssim \left \langle \xi \right \rangle ^{\tau
+q(\gamma -1)-1}  \label{k^p0-integ}
\end{equation}%
provided $1\leq q<\frac{3}{-\gamma }$. Consequently, if $-3/2<\gamma <0$,
\begin{equation}
|Kg|_{L_{\xi ,\tau -\gamma +3/2}^{\infty }}\leq C|g|_{L_{\xi ,\tau }^{2}}%
\text{\thinspace ,}  \label{gain-weight}
\end{equation}%
and if $-2<\gamma \leq -3/2$,
\begin{equation}
|Kg|_{L_{\xi ,\tau +2-\gamma -\frac{1}{q}}^{\infty }}\leq C|g|_{L_{\xi ,\tau
}^{q}}  \label{gain-weight1}
\end{equation}%
provided $q>\frac{3}{3+\gamma }$.
\end{lemma}

Taking $q=3$ in (\ref{gain-weight1}) and $q=1$ in (\ref{K-Lp}),
respectively, we have,%
\begin{equation*}
|Kg|_{L_{\xi ,1-\gamma }^{\infty }}\leq C|g|_{L_{\xi }^{3}}\text{,}
\end{equation*}%
and
\begin{equation*}
|Kg|_{L_{\xi ,1-\gamma }^{1}}\leq C|g|_{L_{\xi }^{1}}\text{.}
\end{equation*}%
Applying the Riesz-Thorin interpolation theorem to the linear operator $%
\left \langle \xi \right \rangle ^{1-\gamma }K$, we obtain the following
estimate, which is useful in the proofs of Theorems \ref{prop: nonlinear}
and \ref{theorem-linear} whenever $-2<\gamma \leq -3/2$.

\begin{lemma}
\label{L4-L2}For $-2<\gamma \leq -3/2$,
\begin{equation}
|Kg|_{L_{\xi ,1-\gamma }^{4}}\leq C|g|_{L_{\xi }^{2}}\text{.}
\label{gain-weight2}
\end{equation}
\end{lemma}

To proceed, we need the estimates associated with the weight function $w_{i}$%
, $i=1$, $2$, $3$. By straightforward computation, $w_{1}$ and $w_{2}$ have
the derivative estimates as below.

\begin{lemma}
\label{Derivative of wi}Let $-2<\gamma <0$. For the weight function $w_{1}$,
we have
\begin{equation}
\left \vert w_{1}^{-1}\partial _{t}w_{1}\right \vert \lesssim \delta M\left
\langle \xi \right \rangle _{D}^{\gamma -1}\text{,}\quad \left \vert
w_{1}^{-1}\nabla _{x}w_{1}\right \vert \lesssim \delta \left \langle \xi
\right \rangle _{D}^{\gamma -1}\text{,}\quad \left \vert w_{1}^{-1}\nabla
_{\xi }w_{1}\right \vert \lesssim \left \langle \xi \right \rangle
_{D}^{-2}\left \vert \xi \right \vert \text{,}%
\vspace {3mm}
\label{der-w1-first}
\end{equation}%
\begin{equation}
\left \vert w_{1}^{-1}\nabla _{x}\left( \partial _{t}w_{1}\right) \right
\vert \lesssim \delta ^{2}M\left \langle \xi \right \rangle _{D}^{2\gamma -2}%
\text{,}\quad \left \vert w_{1}^{-1}\nabla _{x}\left( \xi \cdot \nabla
_{x}w_{1}\right) \right \vert \lesssim \delta ^{2}\left \langle \xi \right
\rangle _{D}^{2\gamma -2}\left \vert \xi \right \vert \text{,}%
\vspace {3mm}
\label{der-w1-mixed-1}
\end{equation}%
\begin{equation}
\left \vert w_{1}^{-1}\nabla _{\xi }\left( \partial _{t}w_{1}\right) \right
\vert \lesssim \delta M\left \langle \xi \right \rangle _{D}^{\gamma
-3}\left \vert \xi \right \vert \text{\thinspace ,}\quad \left \vert
w_{1}^{-1}\nabla _{\xi }\left( \xi \cdot \nabla _{x}w_{1}\right) \right
\vert \lesssim \delta \left \langle \xi \right \rangle _{D}^{\gamma -1}\text{%
.}  \label{der-w1-mixed-2}
\end{equation}%
For the weight function $w_{2}$, its exponent $\rho (t,x,\xi )$ satisfies
\begin{equation*}
\left \vert \partial _{t}\rho \right \vert \lesssim \delta M\left \langle
\xi \right \rangle ^{\gamma -1}\text{,}\quad \left \vert \nabla _{x}\rho
\right \vert \lesssim \delta \left \langle \xi \right \rangle ^{\gamma -1}%
\text{,}\quad \left \vert \nabla _{\xi }\rho \right \vert \lesssim \left
\langle \xi \right \rangle ^{p-2}\left \vert \xi \right \vert \text{,}%
\vspace {3mm}%
\end{equation*}%
\begin{equation*}
\left \vert \nabla _{x}\left( \partial _{t}\rho \right) \right \vert
\lesssim \delta ^{2}M\left \langle \xi \right \rangle ^{-p+2\gamma -2}\text{,%
}\quad \left \vert \nabla _{x}\left( \xi \cdot \nabla _{x}\rho \right)
\right \vert \lesssim \delta ^{2}\left \langle \xi \right \rangle
^{-p+2\gamma -2}\left \vert \xi \right \vert \text{,}%
\vspace {3mm}%
\end{equation*}%
\begin{equation*}
\left \vert \nabla _{\xi }\left( \partial _{t}\rho \right) \right \vert
\lesssim \delta M\left \langle \xi \right \rangle ^{\gamma -3}\left \vert
\xi \right \vert \text{,}\quad \left \vert \nabla _{\xi }\left( \xi \cdot
\nabla _{x}\rho \right) \right \vert \lesssim \delta \left \langle \xi
\right \rangle ^{\gamma -1}\text{,}
\end{equation*}%
where $0<p\leq 2$.
\end{lemma}

Moreover, we have

\begin{lemma}
\label{ratio}For $-2<\gamma <0$,
\begin{equation}
\left \vert \frac{w_{1}\left( t,x,\xi \right) }{w_{1}\left( t,x,\xi _{\ast
}\right) }-1\right \vert \leq CD^{-\left \{ p\wedge 2\right \} }\left[
1+\left \vert \left \vert \xi \right \vert ^{2}-\left \vert \xi _{\ast
}\right \vert ^{2}\right \vert \right] ^{\frac{p}{2}}\text{, }p\geq 1\text{,}
\label{mu-ratio-2}
\end{equation}%
\begin{equation}
\left \vert \frac{w_{2}(t,x,\xi )}{w_{2}(t,x,\xi _{\ast })}-1\right \vert
\leq \epsilon c_{p}\left \vert \left \vert \xi \right \vert ^{2}-\left \vert
\xi _{\ast }\right \vert ^{2}\right \vert ^{\frac{p}{2}}\exp \left( \epsilon
c_{p}\left \vert \left \vert \xi \right \vert ^{2}-\left \vert \xi _{\ast
}\right \vert ^{2}\right \vert ^{\frac{p}{2}}\right) \text{, }0<p\leq 2\text{%
,}\,
\end{equation}%
\begin{equation}
\left \vert \frac{w_{3}\left( t,x,\xi \right) }{w_{3}\left( t,x,\xi _{\ast
}\right) }-1\right \vert \leq \hat{\varepsilon}\left \vert \left \vert \xi
\right \vert ^{2}-\left \vert \xi _{\ast }\right \vert ^{2}\right \vert ^{%
\frac{p_{1}}{2}}\exp \left( \hat{\varepsilon}\left \vert \left \vert \xi
\right \vert ^{2}-\left \vert \xi _{\ast }\right \vert ^{2}\right \vert ^{%
\frac{p_{1}}{2}}\right) \text{, }0<p_{1}\leq 2\text{.}  \label{mu-ratio-w0}
\end{equation}%
Here the constants $c_{p}>0$ and $C>0$ are dependent only upon $p$ and $%
\gamma $.
\end{lemma}

\begin{proof}
Let $s=\left \langle \xi \right \rangle _{D}$ and $s_{1}=\left \langle \xi
_{\ast }\right \rangle _{D}$. Then%
\begin{equation*}
\left \vert \frac{\partial w_{1}}{\partial s}\left( t,x,s\right) \right
\vert \leq C_{1}s^{p-1}\text{,}\
\end{equation*}%
and thus%
\begin{eqnarray*}
\left \vert w_{1}\left( t,x,s\right) -w_{1}\left( t,x,s_{1}\right) \right
\vert &=&\left \vert \left( s-s_{1}\right) \int_{0}^{1}\partial
_{s}w_{1}\left( t,x,\theta s+\left( 1-\theta \right) s_{1}\right) d\theta
\right \vert \\
&\leq &C_{1}\left \vert \left( s-s_{1}\right) \int_{0}^{1}\left( \theta
s+\left( 1-\theta \right) s_{1}\right) ^{p-1}d\theta \right \vert \\
&\leq &C_{1}^{\prime }\left \vert \left( s^{p}-s_{1}^{p}\right) \right \vert
\text{.}
\end{eqnarray*}%
Also, since $w_{1}(t,x,\xi _{\ast })\gtrsim \left \langle \xi _{\ast
}\right
\rangle _{D}^{p}$ and $D\geq 1$, we can deduce that for $1\leq p<2$%
,
\begin{eqnarray*}
\left \vert \frac{w_{1}(t,x,\xi )}{w_{1}(t,x,\xi _{\ast })}-1\right \vert
&=&\left \vert \frac{w_{1}(t,x,\xi )-w_{1}(t,x,\xi _{\ast })}{w_{1}(t,x,\xi
_{\ast })}\right \vert \\
&\lesssim &\frac{\left \vert \left \langle \xi \right \rangle _{D}^{p}-\left
\langle \xi _{\ast }\right \rangle _{D}^{p}\right \vert }{\left \langle \xi
_{\ast }\right \rangle _{D}^{p}}\lesssim \frac{1}{D^{p}}\left \vert \left
\vert \xi \right \vert ^{2}-\left \vert \xi _{\ast }\right \vert ^{2}\right
\vert ^{p/2}\text{,}
\end{eqnarray*}%
and for $p\geq 2$,
\begin{eqnarray*}
\left \vert \frac{w_{1}(t,x,\xi )}{w_{1}(t,x,\xi _{\ast })}-1\right \vert
&\lesssim &\frac{\left \vert \left \langle \xi \right \rangle _{D}^{p}-\left
\langle \xi _{\ast }\right \rangle _{D}^{p}\right \vert }{\left \langle \xi
_{\ast }\right \rangle _{D}^{p}} \\
&\lesssim &\frac{\left \vert \left \vert \xi \right \vert ^{2}-\left \vert
\xi _{\ast }\right \vert ^{2}\right \vert }{\left \langle \xi _{\ast }\right
\rangle _{D}^{p}}\int_{0}^{1}\left( \theta \left \langle \xi \right \rangle
_{D}^{2}+\left( 1-\theta \right) \left \langle \xi _{\ast }\right \rangle
_{D}^{2}\right) ^{\frac{p}{2}-1}d\theta \\
&\lesssim &\frac{\left \vert \left \vert \xi \right \vert ^{2}-\left \vert
\xi _{\ast }\right \vert ^{2}\right \vert }{D^{2}}\int_{0}^{1}\frac{\left(
D^{2}+\left \vert \xi _{\ast }\right \vert ^{2}+\theta \left( \left \vert
\xi \right \vert ^{2}-\left \vert \xi _{\ast }\right \vert ^{2}\right)
\right) ^{\frac{p}{2}-1}}{\left( D^{2}+\left \vert \xi _{\ast }\right \vert
^{2}\right) ^{\frac{p}{2}-1}}d\theta \\
&\lesssim &\frac{1}{D^{2}}\left \vert \left \vert \xi \right \vert
^{2}-\left \vert \xi _{\ast }\right \vert ^{2}\right \vert \left( 1+\left
\vert \left \vert \xi \right \vert ^{2}-\left \vert \xi _{\ast }\right \vert
^{2}\right \vert \right) ^{\frac{p}{2}-1}\text{.}
\end{eqnarray*}%
Combining the above two estimates, we can conclude (\ref{mu-ratio-2}).

Similar to $w_{1}$, we have
\begin{equation*}
\left \vert \rho \left( t,x,\xi \right) -\rho \left( t,x,\xi _{\ast }\right)
\right \vert \leq c_{p}\left \vert \left \langle \xi \right \rangle
^{p}-\left \langle \xi _{\ast }\right \rangle ^{p}\right \vert \leq
c_{p}\left \vert \left \vert \xi \right \vert ^{2}-\left \vert \xi _{\ast
}\right \vert ^{2}\right \vert ^{p/2}
\end{equation*}%
for $0<p\leq 2$, so that
\begin{eqnarray*}
\left \vert \frac{w_{2}(t,x,\xi )}{w_{2}(t,x,\xi _{\ast })}-1\right \vert
&=&\left \vert e^{\epsilon \left( \rho \left( t,x,\xi \right) -\rho \left(
t,x,\xi _{\ast }\right) \right) }-1\right \vert \\
&\leq &\epsilon \left \vert \rho \left( t,x,\xi \right) -\rho \left( t,x,\xi
_{\ast }\right) \right \vert e^{\epsilon \left \vert \rho \left( t,x,\xi
\right) -\rho \left( t,x,\xi _{\ast }\right) \right \vert } \\
&\leq &\epsilon c_{p}\left \vert \left \vert \xi \right \vert ^{2}-\left
\vert \xi _{\ast }\right \vert ^{2}\right \vert ^{\frac{p}{2}}\exp \left(
\epsilon c_{p}\left \vert \left \vert \xi \right \vert ^{2}-\left \vert \xi
_{\ast }\right \vert ^{2}\right \vert ^{\frac{p}{2}}\right) \text{,}
\end{eqnarray*}%
as desired. As for $w_{3}$, the proof is straightforward, that is,%
\begin{eqnarray*}
\left \vert \frac{w_{3}(t,x,\xi )}{w_{3}(t,x,\xi _{\ast })}-1\right \vert
&=&\left \vert e^{\hat{\varepsilon}\left( \left \langle \xi \right \rangle
^{p_{1}}-\left \langle \xi _{\ast }\right \rangle ^{p_{1}}\right) }-1\right
\vert \\
&\leq &\hat{\varepsilon}\left \vert \left \langle \xi \right \rangle
^{p_{1}}-\left \langle \xi _{\ast }\right \rangle ^{p_{1}}\right \vert \exp
\left( \hat{\varepsilon}\left \vert \left \langle \xi \right \rangle
^{p_{1}}-\left \langle \xi _{\ast }\right \rangle ^{p_{1}}\right \vert
\right) \\
&\leq &\hat{\varepsilon}\left \vert \left \vert \xi \right \vert ^{2}-\left
\vert \xi _{\ast }\right \vert ^{2}\right \vert ^{\frac{p_{1}}{2}}\exp
\left( \hat{\varepsilon}\left \vert \left \vert \xi \right \vert ^{2}-\left
\vert \xi _{\ast }\right \vert ^{2}\right \vert ^{\frac{p_{1}}{2}}\right)
\end{eqnarray*}%
for $0<p_{1}\leq 2$. The proof of this lemma is completed.
\end{proof}

With the help of the estimates on the weight functions, we obtain some
useful estimates regarding the integral operator $K$ in the weighted spaces.
For simplicity of notations, we define $K_{w_{i}}=w_{i}\left( t,x,\xi
\right) Kw_{i}^{-1}\left( t,x,\xi _{\ast }\right) $, $i=1$, $2$,$\ 3$.

\begin{lemma}
\label{lemma: estimate of weighted-kw}Let $\tau \in \mathbb{R}$. For $%
-2<\gamma <0$ and $i=1$, $2$,$\ 3$, we have
\begin{equation}
\int_{\mathbb{R}^{3}}\left \vert w_{i}(t,x,\xi )k(\xi ,\xi _{\ast
})w_{i}^{-1}(t,x,\xi _{\ast })\right \vert \left \langle \xi _{\ast }\right
\rangle ^{\tau }d\xi _{\ast }\lesssim \left \langle \xi \right \rangle
^{\tau +\gamma -2}\text{,}  \label{Kw2}
\end{equation}%
\begin{equation}
\int_{\mathbb{R}^{3}}\left \vert w_{i}(t,x,\xi )k(\xi ,\xi _{\ast
})w_{i}^{-1}(t,x,\xi _{\ast })\right \vert \left \langle \xi \right \rangle
^{\tau }d\xi \lesssim \left \langle \xi _{\ast }\right \rangle ^{\tau
+\gamma -2}\text{,}  \label{Kw3}
\end{equation}%
\begin{equation}
\int_{\mathbb{R}^{3}}\left \vert w_{i}(t,x,\xi )\left( \nabla _{\xi }k(\xi
,\xi _{\ast })\right) w_{i}^{-1}(t,x,\xi _{\ast })\right \vert \left \langle
\xi _{\ast }\right \rangle ^{\tau }d\xi _{\ast }\lesssim \left \langle \xi
\right \rangle ^{\tau +\gamma -1}\, \text{,}  \label{Kw1}
\end{equation}%
\begin{equation}
\int_{\mathbb{R}^{3}}\left \vert w_{i}(t,x,\xi )\left( \nabla _{\xi }k(\xi
,\xi _{\ast })\right) w_{i}^{-1}(t,x,\xi _{\ast })\right \vert \left \langle
\xi \right \rangle ^{\tau }d\xi \lesssim \left \langle \xi _{\ast }\right
\rangle ^{\tau +\gamma -1}\text{,}  \label{Kw4}
\end{equation}%
uniformly in $t$ and $x$; consequently, we have
\begin{equation}
\left \vert K_{w_{i}}q(t,x,\xi )\right \vert _{L_{\xi ,\tau +2-\gamma
}^{\infty }}\lesssim \left \vert q(t,x,\cdot )\right \vert _{L_{\xi ,\tau
}^{\infty }}\text{,}  \label{Kw6}
\end{equation}%
\begin{equation}
\left \vert \left \langle \xi \right \rangle ^{\tau }K_{w_{i}}q(t,x,\xi
)\right \vert _{L_{\xi }^{2}}\lesssim \left \vert \left \langle \xi \right
\rangle ^{\tau -2+\gamma }q(t,x,\cdot )\right \vert _{L_{\xi }^{2}}\text{,}\,
\label{Kw7}
\end{equation}%
\begin{equation}
\left \vert \left \langle \xi \right \rangle ^{\tau }\nabla _{\xi
}K_{w_{i}}q(t,x,\xi )\right \vert _{L_{\xi }^{2}}\lesssim \left \vert \left
\langle \xi \right \rangle ^{\tau -1+\gamma }q(t,x,\cdot )\right \vert
_{L_{\xi }^{2}}\text{.}  \label{Kw8}
\end{equation}

Furthermore, if $-3/2<\gamma <0$,%
\begin{equation}
\left \vert K_{w_{i}}q(t,x,\xi )\right \vert _{L_{\xi ,\tau }^{\infty
}}\lesssim \left \vert \left \langle \xi \right \rangle ^{\tau -3/2+\gamma
}q(t,x,\cdot )\right \vert _{L_{\xi }^{2}}\text{;}  \label{eq: bootstrap 1}
\end{equation}%
if $-2<\gamma \leq -3/2$,
\begin{equation}
|K_{w_{i}}q(t,x,\xi )|_{L_{\xi ,\tau }^{\infty }}\lesssim |\left \langle \xi
\right \rangle ^{\tau -2+\gamma +1/s}q(t,x,\cdot )|_{L_{\xi }^{s}}\text{,}\,
\label{eq: bootstrap 2}
\end{equation}%
provided $s>\frac{3}{3+\gamma }$, and%
\begin{equation}
|K_{w_{i}}q(t,x,\xi )|_{L_{\xi ,1-\gamma }^{4}}\lesssim |q(t,x,\cdot
)|_{L_{\xi }^{2}}\text{.}  \label{eq: bootstrap 3}
\end{equation}
\end{lemma}

\begin{proof}
Since $k=-k_{1}+k_{2}$ and the estimate for $k_{1}$ can be obtained easily,
we just prove (\ref{Kw2}) for $k_{2}$ whenever the weight function is $w_{1}$
and then a similar argument can be applied to the estimates (\ref{Kw2})-(\ref%
{Kw4}) for the weight functions $w_{i}$, $i=1$, $2$, $3$. Now, rewrite
\begin{eqnarray*}
&&w_{1}(t,x,\xi )k_{2}(\xi ,\xi _{\ast })w_{1}^{-1}(t,x,\xi _{\ast
})-k_{2}(\xi ,\xi _{\ast }) \\
&=&\left \{ a\left( \xi ,\xi _{\ast },\frac{1}{2}\right) \exp \left( -\frac{1%
}{32}\left[ \frac{\left( \left \vert \xi \right \vert ^{2}-\left \vert \xi
_{\ast }\right \vert ^{2}\right) ^{2}}{\left \vert \xi -\xi _{\ast }\right
\vert ^{2}}+\left \vert \xi -\xi _{\ast }\right \vert ^{2}\right] \right)
\right \} \\
&&\times \left \{ \exp \left( -\frac{1}{32}\left[ \frac{\left( \left \vert
\xi \right \vert ^{2}-\left \vert \xi _{\ast }\right \vert ^{2}\right) ^{2}}{%
\left \vert \xi -\xi _{\ast }\right \vert ^{2}}+\left \vert \xi -\xi _{\ast
}\right \vert ^{2}\right] \right) \times \left( \frac{w_{1}(t,x)}{w_{1\ast
}(t,x)}-1\right) \right \} \\
&\equiv &q_{2}(\xi ,\xi _{\ast })s(D,\xi ,\xi _{\ast })\text{.}\,
\end{eqnarray*}%
By the Cauchy-Schwarz inequality,
\begin{equation*}
\frac{\left( \left \vert \xi \right \vert ^{2}-\left \vert \xi _{\ast
}\right \vert ^{2}\right) ^{2}}{\left \vert \xi -\xi _{\ast }\right \vert
^{2}}+\left \vert \xi -\xi _{\ast }\right \vert ^{2}\geq 2\left \vert \left
\vert \xi \right \vert ^{2}-\left \vert \xi _{\ast }\right \vert ^{2}\right
\vert \text{.}
\end{equation*}%
In view of (\ref{mu-ratio-2}), we obtain
\begin{equation*}
\sup_{\xi ,\xi _{\ast }}\left \vert s(D,\xi ,\xi _{\ast })\right \vert
\rightarrow 0\text{ as }D\rightarrow \infty \text{,}
\end{equation*}%
which implies that $\sup_{\xi ,\xi _{\ast }}\left \vert s(D,\xi ,\xi _{\ast
})\right \vert <1$ for all $D\geq 1$ sufficiently large. Moreover, in view
of Lemma \ref{basic2}, we know
\begin{equation*}
\int_{\mathbb{R}^{3}}\left \vert k_{2}(\xi ,\xi _{\ast })\right \vert \left
\langle \xi _{\ast }\right \rangle ^{\tau }d\xi _{\ast }\lesssim \left
\langle \xi \right \rangle ^{\tau +\gamma -2}\text{,}\,
\end{equation*}%
\begin{equation*}
\int_{\mathbb{R}^{3}}\left \vert q_{2}(\xi ,\xi _{\ast })\right \vert \left
\langle \xi _{\ast }\right \rangle ^{\tau }d\xi _{\ast }\lesssim \left
\langle \xi \right \rangle ^{\tau +\gamma -2}\text{.}\,
\end{equation*}%
Therefore,
\begin{eqnarray*}
&&\int_{\mathbb{R}^{3}}\left \vert w_{1}(t,x,\xi )k_{2}(\xi ,\xi _{\ast
})w_{1}^{-1}(t,x,\xi _{\ast })\right \vert \left \langle \xi _{\ast }\right
\rangle ^{\tau }d\xi _{\ast } \\
&\lesssim &\int_{\mathbb{R}^{3}}\left \vert k_{2}(\xi ,\xi _{\ast })\right
\vert \left \langle \xi _{\ast }\right \rangle ^{\tau }d\xi _{\ast }+\int_{%
\mathbb{R}^{3}}\left \vert q_{2}(\xi ,\xi _{\ast })s(D,\xi ,\xi _{\ast
})\right \vert \left \langle \xi _{\ast }\right \rangle ^{\tau }d\xi _{\ast }
\\
&\lesssim &\left \langle \xi \right \rangle ^{\tau +\gamma -2}\text{.}\,
\end{eqnarray*}%
Combining the above estimates (\ref{Kw2})-(\ref{Kw4}) together with (\ref%
{der-w1-first}), we can deduce (\ref{Kw6})-(\ref{Kw8}).

Mimicking the proof of (\ref{Kw2}), together with (\ref{gain-weight})-(\ref%
{gain-weight2}), we obtain (\ref{eq: bootstrap 1})-(\ref{eq: bootstrap 3}).
The proof of this lemma is completed.
\end{proof}

\begin{remark}
\label{remark-ration} Similar to the proof of Lemma \ref{ratio}, together
with the conservation of energy, we have
\begin{equation}
\left \vert \frac{w_{1}\left( t,x,\xi \right) }{w_{1}\left( t,x,\xi _{\ast
}^{\prime }\right) }-1\right \vert \lesssim D^{-\{2\wedge p\}}\left[ 1+\left
\vert \left \vert \xi ^{\prime }\right \vert ^{2}-\left \vert \xi _{\ast
}\right \vert ^{2}\right \vert \right] ^{\frac{p}{2}}\text{, }p\geq 1\text{,}
\label{mu-ratio-1}
\end{equation}%
\begin{equation}
\left \vert \frac{w_{1}\left( t,x,\xi \right) }{w_{1}\left( t,x,\xi ^{\prime
}\right) }-1\right \vert \lesssim D^{-\{2\wedge p\}}\left[ 1+\left \vert
\left \vert \xi _{\ast }^{\prime }\right \vert ^{2}-\left \vert \xi _{\ast
}\right \vert ^{2}\right \vert \right] ^{\frac{p}{2}}\text{, }p\geq 1\text{,}
\label{mu-ratio-3}
\end{equation}%
\begin{equation}
\left \vert \frac{w_{2}(t,x,\xi )}{w_{2}(t,x,\xi _{\ast }^{\prime })}%
-1\right \vert \lesssim \epsilon c_{p}\left \vert \left \vert \xi ^{\prime
}\right \vert ^{2}-\left \vert \xi _{\ast }\right \vert ^{2}\right \vert ^{%
\frac{p}{2}}\exp \left( \epsilon c_{p}\left \vert \left \vert \xi ^{\prime
}\right \vert ^{2}-\left \vert \xi _{\ast }\right \vert ^{2}\right \vert ^{%
\frac{p}{2}}\right) \, \text{, }0<p\leq 2\text{,}
\end{equation}
\begin{equation}
\left \vert \frac{w_{2}(t,x,\xi )}{w_{2}(t,x,\xi ^{\prime })}-1\right \vert
\lesssim \epsilon c_{p}\left \vert \left \vert \xi _{\ast }^{\prime }\right
\vert ^{2}-\left \vert \xi _{\ast }\right \vert ^{2}\right \vert ^{\frac{p}{2%
}}\exp \left( \epsilon c_{p}\left \vert \left \vert \xi _{\ast }^{\prime
}\right \vert ^{2}-\left \vert \xi _{\ast }\right \vert ^{2}\right \vert ^{%
\frac{p}{2}}\right) \text{, }0<p\leq 2\text{.}
\end{equation}
Here the constant $c_{p}>0$ is the same as in Lemma \ref{ratio}. On the
other hand, for the weight function $w_{3}$, we have%
\begin{eqnarray}
\left \vert \frac{w_{3}\left( t,x,\xi \right) }{w_{3}\left( t,x,\xi _{\ast
}\right) }-1\right \vert &=&\left \vert e^{\hat{\varepsilon}\left( \left
\langle \xi \right \rangle ^{p_{1}}-\left \langle \xi _{\ast }\right \rangle
^{p_{1}}\right) }-1\right \vert \leq \exp \left( \hat{\varepsilon}\left
\vert \left \langle \xi \right \rangle ^{p_{1}}-\left \langle \xi _{\ast
}\right \rangle ^{p_{1}}\right \vert \right)  \notag \\
&=&\exp \left( \hat{\varepsilon}\left \vert \left( 1+\left \vert \xi \right
\vert ^{2}\right) ^{\frac{p_{1}}{2}}-\left( 1+\left \vert \xi _{\ast }\right
\vert ^{2}\right) ^{\frac{p_{1}}{2}}\right \vert \right)  \notag \\
&\leq &\exp \left( \hat{\varepsilon}\left \vert \left \vert \xi \right \vert
^{2}-\left \vert \xi _{\ast }\right \vert ^{2}\right \vert ^{\frac{p_{1}}{2}%
}\right) \text{,}
\end{eqnarray}%
since $0<p_{1}/2\leq 1$. By the conservation of energy,
\begin{equation}
\left \vert \frac{w_{3}\left( t,x,\xi \right) }{w_{3}\left( t,x,\xi _{\ast
}^{\prime }\right) }-1\right \vert \leq \exp \left( \hat{\varepsilon}\left
\vert \left \vert \xi ^{\prime }\right \vert ^{2}-\left \vert \xi _{\ast
}\right \vert ^{2}\right \vert ^{\frac{p_{1}}{2}}\right) \text{, }%
0<p_{1}\leq 2\text{,}  \label{ratio-w01}
\end{equation}%
\begin{equation}
\left \vert \frac{w_{3}\left( t,x,\xi \right) }{w_{3}\left( t,x,\xi ^{\prime
}\right) }-1\right \vert \leq \exp \left( \hat{\varepsilon}\left \vert \left
\vert \xi _{\ast }^{\prime }\right \vert ^{2}-\left \vert \xi _{\ast }\right
\vert ^{2}\right \vert ^{\frac{p_{1}}{2}}\right) \, \text{, }0<p_{1}\leq 2%
\text{.}  \label{ratio-w02}
\end{equation}
\end{remark}

Furthermore, we consider the linear operator $\mathcal{L}_{w_{i}}$, $i=1$, $%
2 $, defined as
\begin{equation*}
\mathcal{L}_{w_{i}}h=-\xi \cdot \nabla _{x}h+(\partial _{t}w_{i}+\xi \cdot
\nabla _{x}w_{i})w_{i}^{-1}h+L_{w_{i}}h\text{.}
\end{equation*}%
By straightforward computation, we obtain the energy estimate for the linear
part as below.

\begin{lemma}[Weighted energy estimate for the linear part]
\label{weighted-linear}Let $-2<\gamma <0$. If $\de>0$ is sufficiently small,
and $D$, $M\geq 1$ are sufficiently large with $\de M$ sufficiently small,
then
\begin{align*}
& \sum_{j=0}^{2}\int_{{\mathbb{R}}^{3}}\int_{{\mathbb{R}}^{3}} \nabla _{x}^{j}h\nabla
_{x}^{j}\mathcal{L}_{w_{1}}h dx d\xi\\
\leq & -\left( \nu _{0}-C_{1}D^{-2}-C_{2}\delta -C_{3}\delta M\right)
\sum_{j=0}^{2}\int_{{\mathbb{R}}^{3}}\int_{{\mathbb{R}}^{3}} \left \langle \xi \right \rangle ^{\gamma} \left(\mathrm{P}_{1}\nabla _{x}^{j}h \right)^2 dx d\xi \\
& -\left( C_{4}\delta M-C_{2}\delta -C_{1}D^{-2}\right)
\sum_{j=0}^{2}\int_{H_{+}^{D}}\left[ \delta \left( \left \langle x\right
\rangle -Mt\right) \right] ^{-1}\left \vert \mathrm{P}_{0}\nabla
_{x}^{j}h\right \vert ^{2}dx d\xi  \\
& +\left( C_{1}D^{-2}+C_{2}\delta +C_{5}\delta M\right)
\sum_{j=0}^{2}\int_{H_{0}^{D}}\left \vert \mathrm{P}_{0}\nabla
_{x}^{j}h\right \vert ^{2}dx d\xi
+C_{1}D^{-2}\sum_{j=0}^{2}\int_{H_{-}^{D}}\left \vert \mathrm{P}_{0}\nabla
_{x}^{j}h\right \vert ^{2}dx d\xi \text{,}
\end{align*}%
where
\begin{eqnarray*}
H_{+}^{D} &=&\{ \left( x,\xi \right) :\delta \left( \left \langle x\right
\rangle -Mt\right) >2\left \langle \xi \right \rangle _{D}^{1-\gamma }\}%
\text{,}%
\vspace {3mm}
\\
H_{0}^{D} &=&\{ \left( x,\xi \right) :\left \langle \xi \right \rangle
_{D}^{1-\gamma }\leq \delta \left( \left \langle x\right \rangle -Mt\right)
\leq 2\left \langle \xi \right \rangle _{D}^{1-\gamma }\} \text{,}%
\vspace {3mm}
\\
H_{-}^{D} &=&\{ \left( x,\xi \right) :\delta \left( \left \langle x\right
\rangle -Mt\right) <\left \langle \xi \right \rangle _{D}^{1-\gamma }\}
\text{.}
\end{eqnarray*}%
If $\epsilon >0$, $\de>0$ are sufficiently small and $M$ is sufficiently
large such that $\de M$ is large but $\de M\ll \epsilon ^{-1}$, then
\begin{eqnarray*}
&&\sum_{j=0}^{2}\int_{{\mathbb{R}}^{3}}\int_{{\mathbb{R}}^{3}} \nabla _{x}^{j}h \nabla
_{x}^{j}\mathcal{L}_{w_{2}}h dx d\xi \\
&\leq &-\left( \nu _{0}-\epsilon \delta C_{2}-\epsilon \delta MC_{3}\right)
\sum_{j=0}^{2}\int_{{\mathbb{R}}^{3}}\int_{{\mathbb{R}}^{3}}\left \langle \xi \right \rangle^{\gamma} \left(\mathrm{P}_{1}\nabla _{x}^{j}h\right )^{2}
dxd\xi\\
&&-\epsilon \left( \delta MC_{4}-\delta C_{2}-C_{1}\right)
\sum_{j=0}^{2}\int_{H_{+}^{1}}\left[ \delta \left( \left \langle x\right
\rangle -Mt\right) \right] ^{\frac{\gamma -1}{p+1-\gamma }}\left \vert
\mathrm{P}_{0}\nabla _{x}^{j}h\right \vert ^{2}dx d\xi \\
&&+\epsilon \left( \delta C_{2}+\delta MC_{5}+C_{1}\right)
\sum_{j=0}^{2}\int_{H_{0}^{1}}\left \vert \mathrm{P}_{0}\nabla
_{x}^{j}h\right \vert ^{2} dx d\xi+\epsilon
C_{1}\sum_{j=0}^{2}\int_{H_{-}^{1}}\left \vert \mathrm{P}_{0}\nabla
_{x}^{j}h\right \vert ^{2}dx d\xi \text{,}
\end{eqnarray*}%
where
\begin{eqnarray*}
H_{+}^{1} &=&\{ \left( x,\xi \right) :\delta \left( \left \langle x\right
\rangle -Mt\right) >2\left \langle \xi \right \rangle ^{p+1-\gamma }\} \text{%
,}%
\vspace {3mm}
\\
H_{0}^{1} &=&\{ \left( x,\xi \right) :\left \langle \xi \right \rangle
^{p+1-\gamma }\leq \delta \left( \left \langle x\right \rangle -Mt\right)
\leq 2\left \langle \xi \right \rangle ^{p+1-\gamma }\} \text{,}%
\vspace {3mm}
\\
H_{-}^{1} &=&\{ \left( x,\xi \right) :\delta \left( \left \langle x\right
\rangle -Mt\right) <\left \langle \xi \right \rangle ^{p+1-\gamma }\} \text{.%
}
\end{eqnarray*}
\end{lemma}

The rest of this section is devoted to estimates for the nonlinear operators
$\Gamma $ and $\Gamma _{w_{i}}$. Before going on, we point out an essential
lemma, which is proved by Guo \cite[Lemma 2]{[Yan Guo]} and is used
frequently in the following discussion. In addition, we split $\Gamma $ into
two parts $\Gamma _{gain}$ and $\Gamma _{loss}$ as below:%
\begin{eqnarray*}
\Gamma (g,h) &\equiv &\Gamma _{gain}(g,h)-\Gamma _{loss}(g,h) \\
&=&\frac{1}{2}\int_{\mathbb{R}^{3}\times \mathbb{S}^{2}}B(\vartheta )|\xi
-\xi _{\ast }|^{\gamma }\mathcal{M}_{\ast }^{1/2}\left[ g_{\ast }^{\prime
}h^{\prime }+g^{\prime }h_{\ast }^{\prime }\right] d\xi _{\ast }d\omega \\
&&-\frac{1}{2}\int_{\mathbb{R}^{3}\times \mathbb{S}^{2}}B(\vartheta )|\xi
-\xi _{\ast }|^{\gamma }\mathcal{M}_{\ast }^{1/2}\left[ g_{\ast }h+gh_{\ast }%
\right] d\xi _{\ast }d\omega \text{.}
\end{eqnarray*}

\begin{lemma}
\label{[Guo]}\cite[Lemma 2]{[Yan Guo]} Let $\varsigma >-3$, $l\left( \xi
\right) \in C^{\infty }\left( \mathbb{R}^{3}\right) $ and $g\left( \xi
\right) \in C^{\infty }\left( \mathbb{R}^{3}\backslash \{0\} \right) $.
Assume that for any multi-index $\alpha $, there is $C_{\alpha }>0$ such
that
\begin{equation*}
\left \vert \partial ^{\alpha }g\left( \xi \right) \right \vert \leq
C_{\alpha }\left \vert \xi \right \vert ^{\varsigma -\left \vert \alpha
\right \vert }\text{,}
\end{equation*}%
\begin{equation*}
\left \vert \partial ^{\alpha }l\left( \xi \right) \right \vert \leq
C_{\alpha }e^{-\tau _{\alpha }\left \vert \xi \right \vert ^{2}}\text{,}
\end{equation*}%
with some $\tau _{\alpha }>0$. Then there is $C_{\alpha }^{\ast }>0$ such
that%
\begin{equation*}
\left \vert \partial ^{\alpha }\left( g\ast l\right) \left( \xi \right)
\right \vert \leq C_{\alpha }^{\ast }\left \langle \xi \right \rangle
^{\varsigma -\left \vert \alpha \right \vert }\text{.}
\end{equation*}
\end{lemma}

\begin{lemma}
\label{Lemma-Gamma-sup}Let $-2<\gamma <0$, $\hat{\varepsilon}\geq 0$, $%
0<p_{1}\leq 2$ and $\lambda \geq 0$. Then
\begin{equation}
|\Gamma _{loss}(g,h)|_{L_{\xi }^{\infty }\left( \left \langle \xi \right
\rangle ^{\lambda }e^{\hat{\varepsilon}\left \langle \xi \right \rangle
^{p_{1}}}\right) }\lesssim |g|_{L_{\xi }^{\infty }}|h|_{L_{\xi }^{\infty
}\left( \left \langle \xi \right \rangle ^{\lambda +\gamma }e^{\hat{%
\varepsilon}\left \langle \xi \right \rangle ^{p_{1}}}\right) }+|h|_{L_{\xi
}^{\infty }}|g|_{L_{\xi }^{\infty }\left( \left \langle \xi \right \rangle
^{\lambda +\gamma }e^{\hat{\varepsilon}\left \langle \xi \right \rangle
^{p_{1}}}\right) }\text{,}  \label{Gamma-sup-1}
\end{equation}%
\begin{equation}
|\Gamma _{gain}(g,h)|_{L_{\xi }^{\infty }\left( \left \langle \xi \right
\rangle ^{\lambda }e^{\hat{\varepsilon}\left \langle \xi \right \rangle
^{p_{1}}}\right) }\lesssim |g|_{L_{\xi }^{\infty }\left( \left \langle \xi
\right \rangle ^{\lambda +\gamma -1}e^{\hat{\varepsilon}\left \langle \xi
\right \rangle ^{p_{1}}}\right) }|h|_{L_{\xi }^{\infty }\left( \left \langle
\xi \right \rangle ^{\lambda +\gamma -1}e^{\hat{\varepsilon}\left \langle
\xi \right \rangle ^{p_{1}}}\right) }\text{.}  \label{Gamma-sup-2}
\end{equation}%
In particular,
\begin{equation}
\left \vert \nu ^{-1}\Gamma (g,h)\right \vert _{L_{\xi ,\lambda }^{\infty
}}\lesssim |g|_{L_{\xi ,\lambda }^{\infty }}|h|_{L_{\xi ,\lambda }^{\infty }}%
\text{.}  \label{Gamma-sup}
\end{equation}
\end{lemma}

\proof%
It readily follows from Lemma \ref{[Guo]} that%
\begin{eqnarray*}
&&\left \vert \left \langle \xi \right \rangle ^{\lambda }e^{\hat{\varepsilon%
}\left \langle \xi \right \rangle ^{p_{1}}}\Gamma _{loss}(g,h)\right \vert \\
&\lesssim &\frac{1}{2}\left \langle \xi \right \rangle ^{\lambda} e^{\hat{%
\varepsilon}\left \langle \xi \right \rangle ^{p_{1}}}\int_{\mathbb{R}%
^{3}\times \mathbb{S}^{2}}B(\vartheta )|\xi -\xi _{\ast }|^{\gamma }\mathcal{%
M}_{\ast }^{1/2}\left[ \left \vert g_{\ast }\right \vert \left \vert h\right
\vert +\left \vert g\right \vert \left \vert h_{\ast }\right \vert \right]
d\xi _{\ast }d\omega \\
&\lesssim &\left \vert g\right \vert _{L_{\xi }^{\infty }}\left \vert
h\right \vert _{L_{\xi }^{\infty }\left( \left \langle \xi \right \rangle
^{\lambda +\gamma }e^{\hat{\varepsilon}\left \langle \xi \right \rangle
^{p_{1}}}\right) }+\left \vert h\right \vert _{L_{\xi }^{\infty }}\left
\vert g\right \vert _{L_{\xi }^{\infty }\left( \left \langle \xi \right
\rangle ^{\lambda +\gamma }e^{\hat{\varepsilon}\left \langle \xi \right
\rangle ^{p_{1}}}\right) }\text{,}
\end{eqnarray*}%
so that (\ref{Gamma-sup-1}) holds. Since the conservation of energy implies
that $\left \langle \xi \right \rangle \lesssim \left \langle \xi ^{\prime
}\right \rangle \left \langle \xi _{\ast }^{\prime }\right \rangle $ and $%
\left \langle \xi \right \rangle ^{p_{1}}\leq \left \langle \xi ^{\prime
}\right \rangle ^{p_{1}}+\left \langle \xi _{\ast }^{\prime }\right \rangle
^{p_{1}}$, we can obtain (\ref{Gamma-sup-2}) by following the argument as in
\cite[Proposition 5.1]{[CaflischII]}. Finally, (\ref{Gamma-sup}) follows by
taking $\hat{\varepsilon}=0$ and replacing $\lambda $ by $\lambda -\gamma $
simultaneously in (\ref{Gamma-sup-1}) and (\ref{Gamma-sup-2}). The proof of
this lemma is completed.$%
\hfill%
\square $

\begin{lemma}
\begin{equation}
\left \vert \left \langle f,\Gamma (g,h)\right \rangle _{\xi }\right \vert
\lesssim \left \vert f\right \vert _{L_{\sigma }^{2}}\left( \left \vert
g\right \vert _{L_{\sigma }^{2}}\left \vert h\right \vert _{L_{\xi }^{\infty
}}+\left \vert g\right \vert _{L_{\xi }^{\infty }}\left \vert h\right \vert
_{L_{\sigma }^{2}}\right) \text{,}  \label{Gamma-product}
\end{equation}%
\begin{equation}
\left \vert \nu ^{-1}\Gamma (g,h)\right \vert _{L_{\xi }^{2}}\lesssim \left
\vert g\right \vert _{L_{\xi }^{\infty }}\left \vert h\right \vert _{L_{\xi
}^{2}}+\left \vert g\right \vert _{L_{\xi }^{2}}\left \vert h\right \vert
_{L_{\xi }^{\infty }}\text{.}  \label{Gamma-L2-1}
\end{equation}
\end{lemma}

\proof%
The idea of the proof comes from \cite[Lemma 3]{[Strain-Guo]} and we give
the complete proof in the Appendix section.$%
\hfill%
\square $

\begin{lemma}
\label{lemma: inf estimate of gamma_w} Let $\lambda \geq 0$. Then%
\begin{eqnarray}
&&\left \vert \nu ^{-1}\Gamma _{w_{1}}(g,h)\right \vert _{L_{\xi ,\lambda
}^{\infty }}\lesssim |g|_{L_{\xi ,\lambda }^{\infty }}|\left \langle \xi
\right \rangle ^{p}h|_{L_{\xi ,\lambda }^{\infty }}\text{,}%
\vspace {3mm}
\label{Gamma-w1-sup} \\
&&\left \vert \nu ^{-1}\Gamma _{w_{1}}(g,h)\right \vert _{L_{\xi
}^{2}}\lesssim \left \vert g\right \vert _{L_{\xi }^{\infty }}\left \vert
\left \langle \xi \right \rangle ^{p}h\right \vert _{L_{\xi }^{2}}+\left
\vert g\right \vert _{L_{\xi }^{2}}\left \vert \left \langle \xi \right
\rangle ^{p}h\right \vert _{L_{\xi }^{\infty }}\text{,}%
\vspace {3mm}
\label{Gamma-w1-L2} \\
&&\left \vert \left \langle f,\Gamma _{w_{1}}(g,h)\right \rangle _{\xi
}\right \vert \lesssim \left \vert f\right \vert _{L_{\sigma }^{2}}\left(
\left \vert g\right \vert _{L_{\sigma }^{2}}\left \vert \left \langle \xi
\right \rangle ^{p}h\right \vert _{L_{\xi }^{\infty }}+\left \vert g\right
\vert _{L_{\xi }^{\infty }}\left \vert \left \langle \xi \right \rangle
^{p}h\right \vert _{L_{\sigma }^{2}}\right) \text{,}
\label{Gamma-w1-w1-product}
\end{eqnarray}%
where $p\geq 1$;%
\begin{eqnarray}
&&\left \vert \nu ^{-1}\Gamma _{w_{2}}(g,h)\right \vert _{L_{\xi ,\lambda
}^{\infty }}\lesssim |g|_{L_{\xi ,\lambda }^{\infty }}|\left \langle \xi
\right \rangle ^{p}e^{\epsilon c_{p}\left \langle \xi \right \rangle
^{p}}h|_{L_{\xi ,\lambda }^{\infty }}\, \text{,}%
\vspace {3mm}
\label{Gamma-w2-sup} \\
&&\left \vert \nu ^{-1}\Gamma _{w_{2}}(g,h)\right \vert _{L_{\xi
}^{2}}\lesssim \left \vert g\right \vert _{L_{\xi }^{\infty }}\left \vert
\left \langle \xi \right \rangle ^{p}e^{\epsilon c_{p}\left \langle \xi
\right \rangle ^{p}}h\right \vert _{L_{\xi }^{2}}+\left \vert g\right \vert
_{L_{\xi }^{2}}\left \vert \left \langle \xi \right \rangle ^{p}e^{\epsilon
c_{p}\left \langle \xi \right \rangle ^{p}}h\right \vert _{L_{\xi }^{\infty
}}\text{,}%
\vspace {3mm}
\label{Gamma-w2-L2} \\
&&\left \vert \left \langle f,\Gamma _{w_{2}}(g,h)\right \rangle _{\xi
}\right \vert \lesssim \left \vert f\right \vert _{L_{\sigma }^{2}}\left(
\left \vert g\right \vert _{L_{\sigma }^{2}}\left \vert \left \langle \xi
\right \rangle ^{p}e^{\epsilon c_{p}\left \langle \xi \right \rangle
^{p}}h\right \vert _{L_{\xi }^{\infty }}+\left \vert g\right \vert _{L_{\xi
}^{\infty }}\left \vert \left \langle \xi \right \rangle ^{p}e^{\epsilon
c_{p}\left \langle \xi \right \rangle ^{p}}h\right \vert _{L_{\sigma
}^{2}}\right) \text{,}  \label{Gamma-w2-product}
\end{eqnarray}%
where $0<p\leq 2$ and the constant $c_{p}>0$ is the same as in Lemma \ref%
{ratio};
\begin{eqnarray}
\left \vert \nu ^{-1}\Gamma _{w_{3}}(g,h)\right \vert _{L_{\xi ,\lambda
}^{\infty }} &\lesssim &\left \vert g\right \vert _{L_{\xi ,\lambda
}^{\infty }}\left \vert e^{\hat{\varepsilon}\left \langle \xi \right \rangle
^{p_{1}}}h\right \vert _{L_{\xi ,\lambda }^{\infty }}\text{,}%
\vspace {3mm}
\label{Gamma-w3-sup} \\
\left \vert \nu ^{-1}\Gamma _{w_{3}}(g,h)\right \vert _{L_{\xi }^{2}}
&\lesssim &\left \vert g\right \vert _{L_{\xi }^{\infty }}\left \vert e^{%
\hat{\varepsilon}\left \langle \xi \right \rangle ^{p_{1}}}h\right \vert
_{L_{\xi }^{2}}+\left \vert g\right \vert _{L_{\xi }^{2}}\left \vert e^{\hat{%
\varepsilon}\left \langle \xi \right \rangle ^{p_{1}}}h\right \vert _{L_{\xi
}^{\infty }}\text{,}  \label{Gamma-w3-L2}
\end{eqnarray}%
where $0<p_{1}\leq 2$.
\end{lemma}

\begin{proof}
Direct calculation shows that for $i=1$, $2$, $3$,
\begin{align*}
& \quad \Gamma _{w_{i}}(g,h)-\Gamma (g,h) \\
& =\int_{{\mathbb{R}}^{3}}\int_{S^{2}}B(\theta )|\xi -\xi _{\ast }|^{\gamma }%
\sqrt{\mathcal{M}_{\ast }}\left[ g^{\prime }h_{\ast }^{\prime }\left( \frac{%
w_{i}}{w_{i}^{\prime }}-1\right) +h^{\prime }g_{\ast }^{\prime }\left( \frac{%
w_{i}}{w_{i\ast }^{\prime }}-1\right) -g_{\ast }h\left( \frac{w_{i}}{%
w_{i\ast }}-1\right) \right] d\omega d\xi _{\ast }\text{.}
\end{align*}%
On the other hand, the conservation of energy implies that $\left \langle
\xi \right \rangle ^{\beta }\lesssim \left \langle \xi _{\ast }^{\prime
}\right \rangle ^{\beta }\left \langle \xi ^{\prime }\right \rangle ^{\beta
} $ for $\beta \geq 0$. Using these facts together with (\ref{Gamma-sup})--(%
\ref{Gamma-L2-1}), Lemma \ref{ratio} and Remark \ref{remark-ration}, we get
the desired estimates.
\end{proof}

\section{Weighted linearized Boltzmann equation with a source term}

\label{WR}

In this section, we are concerned with the following inhomogeneous problem:
\begin{equation}
\left \{
\begin{array}{l}
\displaystyle \pa_{t}u+\xi \cdot \nabla _{x}u-\left[ \partial
_{t}w_{i}(t,x,\xi )+\xi \cdot \nabla _{x}w_{i}(t,x,\xi )\right]
w_{i}^{-1}u=L_{w_{i}}u+\Gamma _{w_{i}}(g_{i},h_{i})\, \text{,} \\[4mm]
\displaystyle u(0,x,\xi )=\eta f_{w_{i}0}\text{,}%
\end{array}%
\right.
\end{equation}%
for $i=1$, $2$, where $g_{i}$ and $h_{i}$ are prescribed. The proofs are
almost the same, so that we focus on the case in which the weight function
is $w_{1}$\ and just state the result for the weight function $w_{2}$\ (see
Theorem \ref{theorem-linear-exp}). Now, let $p\geq 1$. We are concerned with
the following inhomogeneous equation:
\begin{equation}
\left \{
\begin{array}{l}
\displaystyle \partial _{t}u+\xi \cdot \nabla _{x}u+\tilde{\nu}%
u=K_{w_{1}}u+\Gamma _{w_{1}}(g_{1},h_{1})\text{,} \\[4mm]
\displaystyle u(0,x,\xi )=\eta f_{w_{1}0}\text{.}%
\end{array}%
\right.  \label{inhom}
\end{equation}%
After choosing $\de>0$ and $\de M$ small enough, we have
\begin{equation*}
\tilde{\nu}(t,x,\xi )=\nu (\xi )-\left[ \partial _{t}w_{1}(t,x,\xi )+\xi
\cdot \nabla _{x}w_{1}(t,x,\xi )\right] w_{1}^{-1}\geq \frac{\nu (\xi )}{2}%
\text{,}
\end{equation*}%
due to (\ref{der-w1-first}). Let $T>0$ and $\beta >3/2$. Assume that $%
f_{w_{1}0}\in L_{\xi ,\beta }^{\infty }L_{x}^{2}\cap L_{\xi ,\beta }^{\infty
}L_{x}^{\infty }$. Also assume the source term $\Gamma _{w_{1}}(g_{1},h_{1})$
satisfies
\begin{equation}
C_{g_{1},T}^{\infty }=\sup_{0\leq t\leq T}(1+t)^{-A}\left \Vert g_{1}\right
\Vert _{L_{\xi ,\beta }^{\infty }L_{x}^{\infty }}<\infty \, \text{,}\quad
C_{g_{1},T}^{2}=\sup_{0\leq t\leq T}\left \Vert g_{1}\right \Vert _{L_{\xi
,\beta }^{\infty }L_{x}^{2}}<\infty \text{,}  \label{g1}
\end{equation}%
for some constant $A\geq 1/2$, and
\begin{equation}
C_{h_{1},T}^{\infty }=\sup_{0\leq t\leq T}(1+t)^{\frac{3}{2}}\left \Vert
\left \langle \xi \right \rangle ^{p}h_{1}\right \Vert _{L_{\xi ,\beta
}^{\infty }L_{x}^{\infty }}<\infty \text{.}  \label{g2}
\end{equation}

Here we mention that throughout this section we abbreviate \textquotedblleft
$\leq C$\textquotedblright \ to \textquotedblleft $\lesssim $%
\textquotedblright \ in which all the constants $C$\ are independent of $T$.

\begin{theorem}
\label{theorem-linear} Let $\beta >3/2$ and $0<\varsigma \ll 1$. Assume that
$f_{w_{1}0}\in L_{\xi ,\beta }^{\infty }L_{x}^{2}\cap L_{\xi ,\beta
}^{\infty }L_{x}^{\infty }$ and that $g_{1},h_{1}$ satisfy $(\ref{g1})$ and $%
(\ref{g2})$, respectively. Then the solution $u$ to the equation $(\ref%
{inhom})$ satisfies
\begin{align}
& \left \Vert u\right \Vert _{L_{\xi ,\beta }^{\infty }L_{x}^{\infty }}
\label{inhom-estimate} \\
& \lesssim \eta \left \Vert f_{w_{1}0}\right \Vert _{L_{\xi ,\beta }^{\infty
}L_{x}^{\infty }}+(1+t)^{-3/2+A+\varsigma }C_{g_{1,T}}^{\infty
}C_{h_{1,T}}^{\infty }  \notag \\
& \quad +\left[ \left( 1+\delta M\right) \left( \eta \left \Vert
f_{w_{1}0}\right \Vert _{L_{\xi ,\beta }^{\infty
}L_{x}^{2}}+C_{g_{1,T}}^{2}C_{h_{1,T}}^{\infty }\right) \right] \cdot \left
\{
\begin{array}{ll}
(1+t)^{2}\text{,}%
\vspace {3mm}
& \text{if }-1<\gamma <0\text{,} \\
(1+t)^{2+\varsigma }\text{,}%
\vspace {3mm}
& \text{if }\gamma =-1\text{,} \\
(1+t)^{7+\frac{5}{\gamma }}\text{,} & \text{if }-2<\gamma <-1\text{,}%
\end{array}%
\right.  \notag
\end{align}%
$0\leq t\leq T$.
\end{theorem}

To prove this theorem, we design a Picard-type iteration, treating $%
K_{w_{1}}u$ as source term. Specifically, the zeroth order approximation $%
u^{(0)}$ is defined as
\begin{equation*}
\left \{
\begin{array}{l}
\partial _{t}u^{(0)}+\xi \cdot \nabla _{x}u^{(0)}+\tilde{\nu}u^{(0)}=\Gamma
_{w_{1}}(g_{1},h_{1})\text{,}%
\vspace {3mm}
\\
u^{(0)}(0,x,\xi )=\eta f_{w_{1}0}\text{,}%
\end{array}%
\right.
\end{equation*}%
and the difference $u-u^{(0)}$ satisfies

\begin{equation*}
\left \{
\begin{array}{l}
\partial _{t}(u-u^{(0)})+\xi \cdot \nabla _{x}(u-u^{(0)})+\tilde{\nu}%
(u-u^{(0)})=K_{w_{1}}(u-u^{(0)})+K_{w_{1}}u^{(0)}\text{,}%
\vspace {3mm}
\\
(u-u^{(0)})(0,x,\xi )=0\text{.}%
\end{array}%
\right.
\end{equation*}%
We can define the $i^{\text{th}}$ order approximation $u^{(i)}$, $i\geq 1$,
inductively as
\begin{equation*}
\left \{
\begin{array}{l}
\partial _{t}u^{(i)}+\xi \cdot \nabla _{x}u^{(i)}+\tilde{\nu}%
u^{(i)}=K_{w_{1}}u^{(i-1)}\text{,}%
\vspace {3mm}
\\
u^{(i)}(0,x,\xi )=0\text{.}%
\end{array}%
\right.
\end{equation*}%
Now, the wave part and the remainder part can be defined as follows:
\begin{equation}
W_{w_{1}}^{(m)}=\sum_{i=0}^{m}u^{(i)}\text{,}\quad \mathcal{R}%
_{w_{1}}^{(m)}=u-W^{(m)}\text{,}
\end{equation}%
$\mathcal{R}_{w_{1}}^{(m)}$ solving the equation
\begin{equation}
\left \{
\begin{array}{l}
\partial _{t}\mathcal{R}_{w_{1}}^{(m)}+\xi \cdot \nabla _{x}\mathcal{R}%
_{w_{1}}^{(m)}+\tilde{\nu}\mathcal{R}_{w_{1}}^{(m)}=K_{w_{1}}\mathcal{R}%
_{w_{1}}^{(m)}+K_{w_{1}}u^{(m)}\text{,}%
\vspace {3mm}
\\[4mm]
\mathcal{R}_{w_{1}}^{(m)}(0,x,\xi )=0\text{.}%
\end{array}%
\right.  \label{remainder-eq}
\end{equation}

In the sequel, we shall estimate the wave part and remainder part in order.

\subsection{Estimates on the wave part}

Let us consider the time-related damped transport equation
\begin{equation}
\left \{
\begin{array}{l}
\displaystyle \partial _{t}h+\xi \cdot \nabla _{x}h+\tilde{\nu}h=0\text{,}%
\vspace {3mm}
\\[4mm]
\displaystyle h(0,x,\xi )=h_{0}(x,\xi )\text{,}%
\end{array}%
\right.  \label{eq: time-relative damped transport eq}
\end{equation}%
and denote the solution operator of the time-related damped transport
equation (\ref{eq: time-relative damped transport eq}) by $\mathbb{S}%
_{w_{1}}(t)$, namely,
\begin{equation}
\mathbb{S}_{w_{1}}(t)h_{0}(x,\xi )=h_{0}(x-t\xi ,\xi )\exp \left(
-\int_{0}^{t}\tilde{\nu}(r,x-(t-r)\xi ,\xi )dr\right) \text{.}
\label{eq: def of S(t)}
\end{equation}%
Next, consider the inhomogeneous problem
\begin{equation}
\left \{
\begin{array}{l}
\displaystyle \partial _{t}h+\xi \cdot \nabla _{x}h+\tilde{\nu}(t,x,\xi
)h=q(t,x,\xi )\text{,}%
\vspace {3mm}
\\[4mm]
\displaystyle h(0,x,\xi )=0\text{,}%
\end{array}%
\right.  \label{eq: time-relative damped transport eq with source}
\end{equation}%
and then we have
\begin{equation*}
h(t,x,\xi )=\int_{0}^{t}q(s,x-(t-s)\xi ,\xi )\exp \left( -\int_{s}^{t}\tilde{%
\nu}(r,x-(t-r)\xi ,\xi )dr\right) ds\text{.}
\end{equation*}%
Furthermore, we define the operator $\mathbb{S}_{w_{1}}(t;s)$ as
\begin{equation}
\mathbb{S}_{w_{1}}(t;s)q(s,x,\xi )\equiv q(s,x-(t-s)\xi ,\xi )\exp \left(
-\int_{s}^{t}\tilde{\nu}(r,x-(t-r)\xi ,\xi )dr\right) \text{,}
\label{eq: def of S(t;s)}
\end{equation}%
for $0\leq s\leq t$, so that the solution $h$ to (\ref{eq: time-relative
damped transport eq with source}) can be represented by
\begin{equation*}
h(t,x,\xi )=\int_{0}^{t}\mathbb{S}_{w_{1}}(t;s)q(s,x,\xi )ds\text{.}
\end{equation*}%
Under this notation, we as well have
\begin{equation*}
\mathbb{S}_{w_{1}}(t;0)f_{0}(x,\xi )=\mathbb{S}_{w_{1}}(t)f_{0}(x,\xi )\text{%
.}
\end{equation*}

Thereupon, each item of the wave part $W_{w_{1}}^{\left( m\right)
}=\sum_{i=0}^{m}u^{\left( i\right) }$ can be expressed as
\begin{equation*}
u^{\left( 0\right) }=\eta \mathbb{S}_{w_{1}}(t)f_{w_{1}0}(x,\xi
)+\int_{0}^{t}\mathbb{S}_{w_{1}}(t;s)\Gamma _{w_{1}}(g_{1},h_{1})(s,x,\xi )ds%
\text{,}
\end{equation*}%
\begin{equation*}
u^{\left( i\right) }=\int_{0}^{t}\mathbb{S}_{w_{1}}(t;s)\left[
K_{w_{1}}u^{(i-1)}\right] (s,x,\xi )ds\text{, }i\geq 1\text{,}
\end{equation*}%
in terms of the operator $\mathbb{S}_{w_{1}}(t;s)$.

Through Lemma \ref{decay-Sw1}, it is easy to see some properties of the
operators $\mathbb{S}_{w_{1}}(t)$ and $\mathbb{S}_{w_{1}}(t;s)$ ($0\leq
s\leq t$).

\begin{lemma}
\label{decay-Sw1}\cite[Lemma 12.1]{[Caflisch]}%
\begin{equation*}
\sup_{\xi }e^{-t\left( 1+\left \vert \xi \right \vert \right) ^{-\alpha
}}\left( 1+\left \vert \xi \right \vert \right) ^{-\lambda }\leq C\left(
1+t\right) ^{-\lambda /\alpha }\text{,}
\end{equation*}%
for $t\geq 0$, $\alpha >0$, $\lambda >0$.
\end{lemma}

\begin{lemma}
\label{lemma pointwise estimate S(t;s)}Let $\tau \geq 0$ and $\lambda \geq 0$%
. Then
\begin{equation}
\left \Vert \mathbb{S}_{w_{1}}(t)h_{0}(x,\xi )\right \Vert _{L_{\xi ,\lambda
}^{\infty }L_{x}^{\infty }}\lesssim (1+s)^{\frac{\tau }{\gamma }}\left \Vert
h_{0}\right \Vert _{L_{\xi ,\lambda +\tau }^{\infty }L_{x}^{\infty }}\,
\text{,}  \label{S1}
\end{equation}%
\begin{equation}
\left \Vert \mathbb{S}_{w_{1}}(t;s)q(s,x,\xi )\right \Vert _{L_{\xi ,\lambda
}^{\infty }L_{x}^{\infty }}\lesssim (1+t-s)^{\frac{\tau }{\gamma }}\left
\Vert q(s,\cdot ,\cdot )\right \Vert _{L_{\xi ,\lambda +\tau }^{\infty
}L_{x}^{\infty }}\text{\thinspace ,}  \label{S2}
\end{equation}%
\begin{equation}
\left \Vert \mathbb{S}_{w_{1}}(t;s)q(s,x,\xi )\right \Vert _{L_{\xi ,\lambda
}^{\infty }L_{x}^{2}}\lesssim (1+t-s)^{\frac{\tau }{\gamma }}\left \Vert
q(s,\cdot ,\cdot )\right \Vert _{L_{\xi ,\lambda +\tau }^{\infty }L_{x}^{2}}%
\text{\thinspace ,}  \label{S5}
\end{equation}%
\begin{equation}
\left \Vert \mathbb{S}_{w_{1}}(s)h_{0}(x,\xi )\right \Vert _{L^{2}}\lesssim
(1+\textcolor{black}{t})^{\frac{\tau }{\gamma }}\left \Vert \left \langle \xi \right \rangle
^{\tau }h_{0}\right \Vert _{L^{2}}\, \text{,}  \label{S3}
\end{equation}%
\begin{equation}
\left \Vert \mathbb{S}_{w_{1}}(t;s)q(s,x,\xi )\right \Vert _{L^{2}}\lesssim
(1+t-s)^{\frac{\tau }{\gamma }}\left \Vert \left \langle \xi \right \rangle
^{\tau }q(s,\cdot ,\cdot )\right \Vert _{L^{2}}\, \text{,}  \label{S4}
\end{equation}%
for $0\leq s\leq t\leq T$.
\end{lemma}

Now we are ready to prove the $L_{\xi ,\beta }^{\infty }L_{x}^{\infty }$
estimate and $L^{2}$ estimate for the wave part $W_{w_{1}}^{(m)}=%
\sum_{i=0}^{m}u^{\left( i\right) }$.

\begin{lemma}[$L_{\protect \xi ,\protect \beta }^{\infty }L_{x}^{\infty }$
estimate of $u^{(i)}$]
\label{pointwise-u}Let $\beta >3/2$, $0<\varsigma \ll 1$. Assume that $%
f_{w_{1}0}\in L_{\xi ,\beta }^{\infty }L_{x}^{\infty }$ and that $g_{1}$ and
$h_{1}$ satisfy $(\ref{g1})$ and $(\ref{g2})$. Then for $i\in \mathbb{N}\cup
\left \{ 0\right \} $,
\begin{equation}
\left \Vert u^{(i)}\right \Vert _{L_{\xi ,\beta }^{\infty }L_{x}^{\infty
}}\lesssim \eta \left \Vert f_{w_{1}0}\right \Vert _{L_{\xi ,\beta }^{\infty
}L_{x}^{\infty }}+(1+t)^{-\frac{3}{2}+A+\varsigma }C_{g_{1},T}^{\infty
}C_{h_{1},T}^{\infty }\text{,}  \label{sup-uj}
\end{equation}%
$0\leq t\leq T$.
\end{lemma}

\begin{proof}
In view of (\ref{Gamma-w1-sup}), (\ref{S1}), (\ref{S2}), together with the
assumptions of (\ref{g1}) and (\ref{g2}), we have
\begin{eqnarray}
&&\left \vert \left \langle \xi \right \rangle ^{\beta }u^{(0)}(t,x,\xi
)\right \vert  \notag \\
&\leq &\eta \left \vert \left \langle \xi \right \rangle ^{\beta }\mathbb{S}%
_{w_{1}}(t)f_{w_{1}0}(x,\xi )\right \vert +\int_{0}^{t}\left \vert \left
\langle \xi \right \rangle ^{\beta }\mathbb{S}_{w_{1}}(t;s)\Gamma
_{w_{1}}(g_{1},h_{1})(s,x,\xi )\right \vert ds  \label{u0-infty} \\
&\lesssim &\eta \left \Vert f_{w_{1}0}\right \Vert _{L_{\xi ,\beta }^{\infty
}L_{x}^{\infty }}+\int_{0}^{t}(1+t-s)^{-\frac{\gamma }{\gamma }}\left \Vert
\nu ^{-1}(\xi )\Gamma _{w_{1}}(g_{1},h_{1})(s,\cdot ,\cdot )\right \Vert
_{L_{\xi ,\beta }^{\infty }L_{x}^{\infty }}ds  \notag \\
&\lesssim &\eta \left \Vert f_{w_{1}0}\right \Vert _{L_{\xi ,\beta }^{\infty
}L_{x}^{\infty }}+\int_{0}^{t}(1+t-s)^{-\frac{\gamma }{\gamma }}(1+s)^{-%
\frac{3}{2}+A}C_{g_{1},T}^{\infty }C_{h_{1},T}^{\infty }ds  \notag \\
&\lesssim &\eta \left \Vert f_{w_{1}0}\right \Vert _{L_{\xi ,\beta }^{\infty
}L_{x}^{\infty }}+\left[ (1+t)^{-\frac{3}{2}+A}\ln \left( 1+t\right) \right]
C_{g_{1},T}^{\infty }C_{h_{1},T}^{\infty }  \notag \\
&\lesssim &\eta \left \Vert f_{w_{1}0}\right \Vert _{L_{\xi ,\beta }^{\infty
}L_{x}^{\infty }}+(1+t)^{-\frac{3}{2}+A+\varsigma }C_{g_{1},T}^{\infty
}C_{h_{1},T}^{\infty }\text{.}  \notag
\end{eqnarray}%
This completes the estimate for $u^{(0)}$.

For $u^{(1)}$, it follows from (\ref{Kw6}), (\ref{S2}) and (\ref{u0-infty})
that
\begin{eqnarray*}
\left \vert \left \langle \xi \right \rangle ^{\beta }u^{(1)}(t,x,\xi
)\right \vert &\leq &\int_{0}^{t}\left \vert \mathbb{S}_{w_{1}}(t;s)\left%
\langle \xi \right \rangle ^{\beta }K_{w_{1}}u^{(0)}(s,x,\xi )\right \vert ds
\\
&\lesssim &\int_{0}^{t}(1+t-s)^{\frac{2-\gamma }{\gamma }}\left \Vert \left
\langle \xi \right \rangle ^{-\gamma +2}K_{w_{1}}u^{(0)}(s,\cdot )\right
\Vert _{L_{\xi ,\beta }^{\infty }L_{x}^{\infty }}ds \\
&\lesssim &\int_{0}^{t}(1+t-s)^{\frac{2-\gamma }{\gamma }}\left \Vert
u^{(0)}(s,\cdot )\right \Vert _{L_{\xi ,\beta }^{\infty }L_{x}^{\infty }}ds
\\
&\lesssim &\eta \left \Vert f_{w_{1}0}\right \Vert _{L_{\xi ,\beta }^{\infty
}L_{x}^{\infty }}+(1+t)^{-\frac{3}{2}+A+\varsigma }C_{g_{1},T}^{\infty
}C_{h_{1},T}^{\infty }\text{.}
\end{eqnarray*}%
Since for $i\geq 2$,
\begin{equation*}
\left \vert \left \langle \xi \right \rangle ^{\beta }u^{(i)}(t,x,\xi
)\right \vert =\left \vert \int_{0}^{t}\mathbb{S}_{w_{1}}(t;s)\left \langle
\xi \right \rangle ^{\beta }K_{w_{1}}u^{(i-1)}(s,x,\xi )ds\right \vert \text{%
, }
\end{equation*}%
we can prove
\begin{equation*}
\left \Vert u^{(i)}\right \Vert _{L_{\xi ,\beta }^{\infty }L_{x}^{\infty
}}\lesssim \eta \left \Vert f_{w_{1}0}\right \Vert _{L_{\xi ,\beta }^{\infty
}L_{x}^{\infty }}+(1+t)^{-\frac{3}{2}+A+\varsigma }C_{g_{1},T}^{\infty
}C_{h_{1},T}^{\infty }\text{,}
\end{equation*}%
by induction on $i\geq 1$.
\end{proof}

\begin{lemma}[$L^{2}$ estimate of $u^{(i)}$, $i\geq 0$]
\label{lemma L2 estimate of u^(j)} \label{L2-u}Let $\beta >3/2$. Assume that
$f_{w_{1}0}\in L_{\xi ,\beta }^{\infty }L_{x}^{2}\cap L_{\xi ,\beta
}^{\infty }L_{x}^{\infty }$ and that $g_{1}$ and $h_{1}$ satisfy $(\ref{g1})$
and $(\ref{g2})$. Then for $i\in \mathbb{N}\cup \left \{ 0\right \} $,%
\begin{equation*}
\left \Vert u^{(i)}\right \Vert _{L^{2}}\lesssim \eta \left \Vert
f_{w_{1}0}\right \Vert _{L_{\xi ,\beta }^{\infty
}L_{x}^{2}}+(1+t)^{-1}C_{g_{1},T}^{2}C_{h_{1},T}^{\infty }\text{,}
\end{equation*}%
$0\leq t\leq T$.
\end{lemma}

\begin{proof}
In view of (\ref{Gamma-w1-L2}) and the assumptions of (\ref{g1}) and (\ref%
{g2}),%
\begin{eqnarray}
\left \Vert \left \langle \xi \right \rangle ^{-\gamma }\Gamma
_{w_{1}}(g_{1},h_{1})(s,\cdot )\right \Vert _{L^{2}} &\lesssim &\left \Vert
g_{1}\right \Vert _{L_{\xi ,\beta }^{\infty }L_{x}^{2}}\left \Vert \left
\langle \xi \right \rangle ^{p}h_{1}\right \Vert _{L_{\xi ,\beta }^{\infty
}L_{x}^{\infty }}  \notag \\
&\lesssim &(1+s)^{-\frac{3}{2}}C_{g_{1},T}^{2}C_{h_{1},T}^{\infty }\text{.}
\label{eq: gamma L2 g1 g2}
\end{eqnarray}%
Therefore, using (\ref{Kw7}), (\ref{S3}), and (\ref{S4}) gives
\begin{eqnarray}
\left \Vert u^{(0)}\right \Vert _{L^{2}} &=&\left \Vert \eta \mathbb{S}%
_{w_{1}}(t)f_{w_{1}0}(x,\xi )+\int_{0}^{t}\mathbb{S}_{w_{1}}(t;s)\Gamma
_{w_{1}}(g_{1},h_{1})(s,x,\xi )ds\, \right \Vert _{L^{2}}  \label{u(0)-L2} \\
&\lesssim &\eta \left \Vert f_{w_{1}0}\right \Vert _{L^{2}}+\left(
\int_{0}^{t}\left( 1+t-s\right) ^{-1}(1+s)^{-\frac{3}{2}}ds\right)
C_{g_{1},T}^{2}C_{h_{1},T}^{\infty }  \notag \\
&\lesssim &\eta \left \Vert f_{w_{1}0}\right \Vert _{L_{\xi ,\beta }^{\infty
}L_{x}^{2}}+(1+t)^{-1}C_{g_{1},T}^{2}C_{h_{1},T}^{\infty }\text{.}  \notag
\end{eqnarray}

Note that%
\begin{equation*}
u^{(1)}(t,x,\xi )=\int_{0}^{t}\mathbb{S}_{w_{1}}(t;s)\left[ K_{w_{1}}u^{(0)}%
\right] (s,x,\xi )ds\text{.}
\end{equation*}%
Using (\ref{Kw7}), (\ref{S4}) and (\ref{u(0)-L2}) gives
\begin{eqnarray*}
\left \Vert u^{(1)}\right \Vert _{L^{2}} &\lesssim &\int_{0}^{t}\left(
1+t-s\right) ^{\frac{2-\gamma }{\gamma }}\left \Vert u^{(0)}\left( s,\cdot
,\cdot \right) \right \Vert _{L^{2}}ds \\
&\lesssim &\eta \left \Vert f_{w_{1}0}\right \Vert _{L_{\xi ,\beta }^{\infty
}L_{x}^{2}}+\left( \int_{0}^{t}\left( 1+t-s\right) ^{\frac{2-\gamma }{\gamma
}}(1+s)^{-1}ds\right) C_{g_{1},T}^{2}C_{h_{1},T}^{\infty } \\
&\lesssim &\eta \left \Vert f_{w_{1}0}\right \Vert _{L_{\xi ,\beta }^{\infty
}L_{x}^{2}}+(1+t)^{-1}C_{g_{1},T}^{2}C_{h_{1},T}^{\infty }\text{.}
\end{eqnarray*}%
Similarly, for $i\geq 2$,
\begin{equation*}
u^{(i)}(t,x,\xi )=\int_{0}^{t}\mathbb{S}_{w_{1}}(t;s)\left[
K_{w_{1}}u^{(i-1)}\right] (s,x,\xi )ds\text{, }
\end{equation*}%
and thus we can prove
\begin{equation*}
\left \Vert u^{(i)}\right \Vert _{L^{2}}\lesssim \eta \left \Vert
f_{w_{1}0}\right \Vert _{L_{\xi ,\beta }^{\infty
}L_{x}^{2}}+(1+t)^{-1}C_{g_{1},T}^{2}C_{h_{1},T}^{\infty }\text{,}
\end{equation*}%
inductively for all $i\in \mathbb{N}$, by using (\ref{Kw7}), (\ref{S4}).
\end{proof}

\subsection{Regularization estimate}

In the previous subsection, we have carried out the $L_{\xi ,\beta }^{\infty
}L_{x}^{\infty }$ ($\beta >3/2$) and $L^{2}$ estimates for the wave part $%
W_{w_{1}}^{(m)}$. To obtain the pointwise estimate on $\mathcal{R}%
_{w_{1}}^{(m)}$, we still need the $H^{2}_{x}$ regularization estimate for $%
\mathcal{R}_{w_{1}}^{(m)}$. In light of (\ref{remainder-eq}), we turn to the
$H^{2}_{x}$ regularization estimate for $u^{(m)}$ in advance. To proceed, we
introduce a differential operator:
\begin{equation*}
\mathcal{D}_{t}=t\nabla _{x}+\nabla _{\xi }\text{.}
\end{equation*}%
This operator $\mathcal{D}_{t}$ is important since it commutes with the free
transport operator, i.e.,
\begin{equation*}
\lbrack \mathcal{D}_{t},\partial _{t}+\xi \cdot \nabla _{x}]=0\text{,}
\end{equation*}%
where $[A,B]=AB-BA$ is the commutator.

We remark that the crucial operator $%
\mathcal{D}_{t}$ was firstly introduced in the paper by Gualdani, Mischler
and Mouhot \cite{[Gualdani]}, and Wu \cite{[Wu]} applied it to reprove the
Mixture Lemma used in \cite{[LeeLiuYu], [LiuYu], [LiuYu2], [LiuYu1]}.

The following lemma will be used to prove the regularization estimate:

\begin{lemma}
\label{lemma Regularization estimate need}For any $\tau \in \mathbb{R}$,
\begin{equation}
\left \Vert \mathbb{S}_{w_{1}}(t;s)\left[ \nabla _{x},K_{w_{1}}\right]
q(s,x,\xi )\right \Vert _{L^{2}}\lesssim (1+t-s)^{\frac{2-\gamma }{\gamma }%
}\left \Vert q(s,\cdot ,\cdot )\right \Vert _{L^{2}}\, \text{,}
\label{comm-Kw-x}
\end{equation}%
\begin{equation}
\left \Vert \mathcal{D}_{t-s}\mathbb{S}_{w_{1}}(t;s)q(s,x,\xi )\right \Vert
_{L^{2}}\lesssim \left( (1+t-s)^{\frac{\tau +1}{\gamma }}+\delta M(1+t-s)^{%
\frac{\tau +\gamma }{\gamma }}\right) \left \Vert \left \langle \xi \right
\rangle ^{\tau }q(s,\cdot ,\cdot )\right \Vert _{H_{\xi }^{1}L_{x}^{2}}\text{%
.}  \label{eq: Regularization estimate 3}
\end{equation}%
Consequently,
\begin{equation}
\left \Vert \mathcal{D}_{t-s}\mathbb{S}_{w_{1}}(t;s)K_{w_{1}}q(s,x,\xi
)\right \Vert _{L^{2}}\lesssim \left( (1+t-s)^{\frac{2-\gamma }{\gamma }%
}+\delta M(1+t-s)^{\frac{1}{\gamma }}\right) \left \Vert q(s,\cdot ,\cdot
)\right \Vert _{L^{2}}\text{.}  \label{eq: Regularization estimate 4}
\end{equation}
\end{lemma}

\begin{proof}
The estimate of (\ref{comm-Kw-x}) immediately follows from (\ref%
{der-w1-first}), (\ref{Kw7}), (\ref{S3}), and the estimate (\ref{eq:
Regularization estimate 4}) is a consequence of (\ref{Kw7}), (\ref{Kw8}),
and (\ref{eq: Regularization estimate 3}) by picking $\tau =-\gamma +1$.
Thus, it remains to prove (\ref{eq: Regularization estimate 3}). By the
definition of $\mathbb{S}_{w_{1}}(t;s)$,
\begin{eqnarray*}
&&\mathcal{D}_{t-s}\mathbb{S}_{w_{1}}(t;s)q(s,x,\xi ) \\
&=&\mathcal{D}_{t-s}\left( q(s,x-(t-s)\xi ,\xi )\right) \exp \left(
-\int_{s}^{t}\tilde{\nu}(r,x-(t-r)\xi ,\xi )dr\right) \\
&&+\left[ q(s,x-(t-s)\xi ,\xi )\left( \mathcal{D}_{t-s}\left( -\int_{s}^{t}%
\tilde{\nu}(r,x-(t-r)\xi ,\xi )dr\right) \right) \right. \\
&&\left. \cdot \exp \left( -\int_{s}^{t}\tilde{\nu}(r,x-(t-r)\xi ,\xi
)dr\right) \right] \text{.}
\end{eqnarray*}%
Since
\begin{eqnarray*}
&&\nabla _{\xi }\left( q(s,x-(t-s)\xi ,\xi )\right) \\
&=&(s-t)\nabla _{x}q(s,x-(t-s)\xi ,\xi )+\nabla _{\xi }q(s,x-(t-s)\xi ,\xi )%
\text{,}
\end{eqnarray*}%
we have
\begin{equation*}
\mathcal{D}_{t-s}\left( q(s,x-(t-s)\xi ,\xi )\right) =\nabla _{\xi
}q(s,x-(t-s)\xi ,\xi )\text{,}
\end{equation*}%
which implies that
\begin{eqnarray*}
&&\left \Vert \mathcal{D}_{t-s}\left( q(s,x-(t-s)\xi ,\xi )\right) \exp
\left( -\int_{s}^{t}\tilde{\nu}(r,x-(t-r)\xi ,\xi )dr\right) \right \Vert
_{L^{2}} \\
&\lesssim &(1+t-s)^{\frac{\tau }{\gamma }}\left \Vert \left \langle \xi
\right \rangle ^{\tau }\nabla _{\xi }q(s,\cdot ,\cdot )\right \Vert _{L^{2}}%
\text{.}
\end{eqnarray*}%
In view of (\ref{der-w1-first}),
\begin{equation*}
\left \vert \mathcal{D}_{t-s}\tilde{\nu}(r,x-(t-r)\xi ,\xi )\right \vert
\lesssim \left( t-s\right) \delta M\left \langle \xi \right \rangle ^{\gamma
}+(t-r)\delta M\left \langle \xi \right \rangle ^{\gamma }+\left \langle \xi
\right \rangle ^{\gamma -1}\text{,}
\end{equation*}%
so that%
\begin{eqnarray*}
&&\left \vert \exp \left( -\int_{s}^{t}\tilde{\nu}(r,x-(t-r)\xi ,\xi
)dr\right) \left( \mathcal{D}_{t-s}\left( -\int_{s}^{t}\tilde{\nu}%
(r,x-(t-r)\xi ,\xi )dr\right) \right) \right \vert \\
&\lesssim &\left( (t-s)\left \langle \xi \right \rangle ^{\gamma -1}+\delta
M(t-s)^{2}\left \langle \xi \right \rangle ^{\gamma }\right) \exp \left( -%
\frac{\nu (\xi )}{2}\left( t-s\right) \right) \text{.}
\end{eqnarray*}%
It implies that
\begin{eqnarray*}
&&\left \Vert q(s,x-(t-s)\xi ,\xi )\left( \mathcal{D}_{t-s}\left(
-\int_{s}^{t}\tilde{\nu}(r,x-(t-r)\xi ,\xi )dr\right) \right) \right. \\
&&\left. \cdot \exp \left( -\int_{s}^{t}\tilde{\nu}(r,x-(t-r)\xi ,\xi
)dr\right) \right \Vert _{L^{2}} \\
&\lesssim &\left( (1+t-s)^{\frac{\tau +1-\gamma }{\gamma }+1}+\delta
M(1+t-s)^{\frac{\tau -\gamma }{\gamma }+2}\right) \left \Vert \left \langle
\xi \right \rangle ^{\tau }q(s,\cdot ,\cdot )\right \Vert _{L^{2}}\text{.}
\end{eqnarray*}%
Therefore, we can conclude
\begin{eqnarray*}
&&\left \Vert \mathcal{D}_{t-s}\mathbb{S}_{w_{1}}(t;s)q(s,x,\xi )\right
\Vert _{L^{2}} \\
&\lesssim &\left( (1+t-s)^{\frac{\tau +1}{\gamma }}+\delta M(1+t-s)^{\frac{%
\tau +\gamma }{\gamma }}\right) \left \Vert \left \langle \xi \right \rangle
^{\tau }q(s,\cdot ,\cdot )\right \Vert _{H_{\xi }^{1}L_{x}^{2}}\text{.}
\end{eqnarray*}%
The proof of this lemma is completed.
\end{proof}

Now, we are ready to get the $H_{x}^{2}$ regularization estimate. In fact,
we find that it is enough to get the $H_{x}^{2}$ regularization estimate by
taking $m=6$.

\begin{lemma}[$H_{x}^{2}$ regularization estimate on $u^{(6)}$]
\label{lemma: H2 estimate of h6}Let $\varsigma $ be any positive number with
$0<\varsigma \ll 1$. Then there exists a constant $C_{\varsigma ,\gamma
,,p}>0$ such that
\begin{eqnarray*}
&&\left \Vert u^{(6)}(t,x,\xi )\right \Vert _{L_{\xi }^{2}H_{x}^{2}} \\
&\leq &C_{\varsigma ,\gamma ,p}\cdot \left \{
\begin{array}{ll}
\left( 1+\delta M\right) \left[ \eta \left \Vert f_{w_{1}0}\right \Vert
_{L_{\xi ,\beta }^{\infty }L_{x}^{2}}+C_{g_{1},T}^{2}C_{h_{1},T}^{\infty }%
\right] \text{,}%
\vspace {3mm}
& \text{if }-1<\gamma <0\text{,} \\
\left( 1+\delta M\right) \left[ \eta \left \Vert f_{w_{1}0}\right \Vert
_{L_{\xi ,\beta }^{\infty }L_{x}^{2}}+C_{g_{1},T}^{2}C_{h_{1},T}^{\infty }%
\right] \left( 1+t\right) ^{\varsigma }\text{,}%
\vspace {3mm}
& \text{if }\gamma =-1\text{,} \\
\left( 1+\delta M\right) \left[ \eta \left \Vert f_{w_{1}0}\right \Vert
_{L_{\xi ,\beta }^{\infty }L_{x}^{2}}+C_{g_{1},T}^{2}C_{h_{1},T}^{\infty }%
\right] (1+t)^{5+\frac{5}{\gamma }}\text{,} & \text{if }-2<\gamma <-1\text{,}%
\end{array}%
\right.
\end{eqnarray*}%
$0\leq t\leq T$.
\end{lemma}

\begin{proof}
In view of Lemma \ref{lemma L2 estimate of u^(j)},
\begin{equation*}
\left \Vert u^{(6)}(t,x,\xi )\right \Vert _{L^{2}}\lesssim \eta \left \Vert
f_{w_{1}0}\right \Vert _{L_{\xi ,\beta }^{\infty
}L_{x}^{2}}+(1+t)^{-1}C_{g_{1},T}^{2}C_{h_{1},T}^{\infty }\text{.}
\end{equation*}%
Next, we prove the estimate for the first $x$-derivative of $u^{(6)}$. Note
that%
\begin{eqnarray*}
&&\nabla _{x}u^{(6)}(t,x,\xi ) \\
&=&\nabla _{x}\int_{0}^{t}\int_{0}^{s_{1}}\int_{0}^{s_{2}}\mathbb{M}_{1}%
\mathbb{M}_{2}\left( \frac{s_{1}-s_{2}}{s_{1}-s_{3}}\mathbb{M}_{3}\right)
u^{(3)}(s_{3},\cdot ,\cdot )ds_{3}ds_{2}ds_{1} \\
&&+\nabla _{x}\int_{0}^{t}\int_{0}^{s_{1}}\int_{0}^{s_{2}}\mathbb{M}_{1}%
\mathbb{M}_{2}\left( \frac{s_{2}-s_{3}}{s_{1}-s_{3}}\mathbb{M}_{3}\right)
u^{(3)}(s_{3},\cdot ,\cdot )ds_{3}ds_{2}ds_{1} \\
&=&\int_{0}^{t}\int_{0}^{s_{1}}\int_{0}^{s_{2}}\frac{1}{s_{1}-s_{3}}\mathbb{M%
}_{1}\left( \mathcal{D}_{s_{1}-s_{2}}-\nabla _{\xi }\right) \mathbb{M}_{2}%
\mathbb{M}_{3}u^{(3)}\left( s_{3},\cdot ,\cdot \right) ds_{3}ds_{2}ds_{1} \\
&&+\int_{0}^{t}\int_{0}^{s_{1}}\int_{0}^{s_{2}}\frac{1}{s_{1}-s_{3}}\mathbb{M%
}_{1}\mathbb{M}_{2}\left( \mathcal{D}_{s_{2}-s_{3}}-\nabla _{\xi }\right)
\mathbb{M}_{3}u^{(3)}\left( s_{3},\cdot ,\cdot \right) ds_{3}ds_{2}ds_{1} \\
&&+\int_{0}^{t}T\left( s_{1},x,\xi ,t\right) ds_{1}\text{,}
\end{eqnarray*}%
where $\mathbb{M}_{i}=\mathbb{S}_{w_{1}}(s_{i-1};s_{i})[K_{w_{1}}]_{s_{i}}$
, $[K_{w_{1}}]_{s_{i}}=w_{1}\left( s_{i},x,\xi \right) Kw_{1}^{-1}\left(
s_{i},x,\xi _{\ast }\right) $, $s_{0}\equiv t$, and%
\begin{eqnarray*}
&&\int_{0}^{t}T\left( s_{1},x,\xi ,t\right) ds_{1} \\
&=&\int_{0}^{t}\int_{0}^{s_{1}}\int_{0}^{s_{2}}\left[ \nabla _{x},\mathbb{S}%
_{w_{1}}(t;s_{1})\right] [K_{w_{1}}]_{s_{1}}\mathbb{M}_{2}\mathbb{M}%
_{3}u^{(3)}\left( s_{3},\cdot ,\cdot \right) ds_{3}ds_{2}ds_{1} \\
&&+\int_{0}^{t}\int_{0}^{s_{1}}\int_{0}^{s_{2}}\mathbb{S}_{w_{1}}(t;s_{1})%
\left[ \nabla _{x},[K_{w_{1}}]_{s_{1}}\right] \mathbb{M}_{2}\mathbb{M}%
_{3}u^{(3)}\left( s_{3},\cdot ,\cdot \right) ds_{3}ds_{2}ds_{1} \\
&&+\int_{0}^{t}\int_{0}^{s_{1}}\int_{0}^{s_{2}}\frac{s_{2}-s_{3}}{s_{1}-s_{3}%
}\mathbb{M}_{1}\left[ \nabla _{x},\mathbb{S}_{w_{1}}(s_{1};s_{2})\right]
[K_{w_{1}}]_{s_{2}}\mathbb{M}_{3}u^{(3)}\left( s_{3},\cdot ,\cdot \right)
ds_{3}ds_{2}ds_{1} \\
&&+\int_{0}^{t}\int_{0}^{s_{1}}\int_{0}^{s_{2}}\frac{s_{2}-s_{3}}{s_{1}-s_{3}%
}\mathbb{M}_{1}\mathbb{S}_{w_{1}}(s_{1};s_{2})\left[ \nabla
_{x},[K_{w_{1}}]_{s_{2}}\right] \mathbb{M}_{3}u^{(3)}\left( s_{3},\cdot
,\cdot \right) ds_{3}ds_{2}ds_{1}\text{.}
\end{eqnarray*}%
Note that
\begin{equation*}
\left \Vert \left[ \nabla _{x},\mathbb{S}_{w_{1}}(s_{i-1};s_{i})\right]
[K_{w_{1}}]_{s_{i}}q(s_{i},x,\xi )\right \Vert _{L^{2}}\lesssim
(1+s_{i-1}-s_{i})^{\frac{2-\gamma }{\gamma }}\left \Vert q(s_{i},\cdot
,\cdot )\right \Vert _{L^{2}}\,
\end{equation*}%
and
\begin{equation*}
\left \Vert \mathbb{S}_{w_{1}}(s_{i-1};s_{i})\left[ \nabla
_{x},[K_{w_{1}}]_{s_{i}}\right] q(s_{i},x,\xi )\right \Vert _{L^{2}}\lesssim
(1+s_{i-1}-s_{i})^{\frac{2-\gamma }{\gamma }}\left \Vert q(s_{i},\cdot
,\cdot )\right \Vert _{L^{2}}\text{,}
\end{equation*}%
for $i=1$, $2$. By (\ref{Kw7}), (\ref{Kw8}), (\ref{S4}), (\ref{eq:
Regularization estimate 4}) and Lemma \ref{lemma L2 estimate of u^(j)}, we
obtain
\begin{align*}
& \quad \left \Vert \nabla _{x}u^{(6)}(t,x,\xi )\right \Vert _{L^{2}} \\
& \lesssim \left( 1+\delta M\right)
\int_{0}^{t}\int_{0}^{s_{1}}\int_{0}^{s_{2}}\left \{ (1+t-s_{1})^{\frac{%
2-\gamma }{\gamma }}(1+s_{1}-s_{2})^{\frac{1}{\gamma }}(1+s_{2}-s_{3})^{%
\frac{1}{\gamma }}\left( 1+\frac{1}{s_{1}-s_{3}}\right) \right \} \\
& \cdot \left( \eta \left \Vert f_{w_{1}0}\right \Vert _{L_{\xi ,\beta
}^{\infty }L_{x}^{2}}+(1+s_{3})^{-1}C_{g_{1},T}^{2}C_{g_{2},T}^{\infty
}\right) ds_{3}ds_{2}ds_{1} \\
& \lesssim \mathbb{A}\cdot
\int_{0}^{t}\int_{0}^{s_{1}}\int_{0}^{s_{2}}(1+t-s_{1})^{\frac{2-\gamma }{%
\gamma }}(1+s_{1}-s_{2})^{\frac{1}{\gamma }}(1+s_{2}-s_{3})^{\frac{1}{\gamma
}}ds_{3}ds_{2}ds_{1} \\
& +\mathbb{A}\int_{0}^{t}\int_{0}^{s_{1}}\int_{s_{3}}^{s_{1}}(1+t-s_{1})^{%
\frac{2-\gamma }{\gamma }}(1+s_{1}-s_{2})^{\frac{1}{\gamma }%
}(1+s_{2}-s_{3})^{\frac{1}{\gamma }}\frac{1}{s_{1}-s_{3}}ds_{2}ds_{3}ds_{1}
\\
& \lesssim \mathbb{A}\int_{0}^{t}\int_{0}^{s_{1}}%
\int_{0}^{s_{2}}(1+t-s_{1})^{\frac{2-\gamma }{\gamma }}(1+s_{1}-s_{2})^{%
\frac{1}{\gamma }}(1+s_{2}-s_{3})^{\frac{1}{\gamma }}ds_{3}ds_{2}ds_{1} \\
& +\mathbb{A}\int_{0}^{t}\int_{0}^{s_{1}}(1+t-s_{1})^{\frac{2-\gamma }{%
\gamma }}(1+s_{1}-s_{3})^{\frac{1}{\gamma }}ds_{3}ds_{1} \\
& \lesssim \left \{
\begin{array}{ll}
\mathbb{A}\text{,}%
\vspace {3mm}
& \text{if }-1<\gamma <0\text{,} \\
\mathbb{A}\left( 1+t\right) ^{\varsigma }\text{,}%
\vspace {3mm}
& \text{if }\gamma =-1\text{,} \\
\mathbb{A}(1+t)^{\frac{2}{\gamma }+2}\text{,} & \text{if }-2<\gamma <-1\text{%
,}%
\end{array}%
\right.
\end{align*}%
where $\mathbb{A=}\left( 1+\delta M\right) \left( \eta \left \Vert
f_{w_{1}0}\right \Vert _{L_{\xi ,\beta }^{\infty
}L_{x}^{2}}+C_{g_{1},T}^{2}C_{g_{2},T}^{\infty }\right) $. Here, the third
inequality holds since $\gamma <0$ and $(1+s_{1}-s_{2})(1+s_{2}-s_{3})\geq
1+s_{1}-s_{3}$ for $s_{3}\leq s_{2}\leq s_{1}$.

For $\left \Vert \nabla _{x}^{2}u^{(6)}(t,x,\xi )\right \Vert _{L^{2}}$,
rewrite
\begin{eqnarray*}
u^{(6)} &=&\int_{0}^{t}\int_{0}^{s_{1}}\int_{0}^{s_{2}}\mathbb{M}_{1}\mathbb{%
M}_{2}\mathbb{M}_{3}u^{(3)}(s_{3},\cdot ,\cdot )ds_{3}ds_{2}ds_{1} \\
&=&\int_{0}^{t}\int_{0}^{s_{1}}\int_{0}^{s_{2}}\int_{0}^{s_{3}}%
\int_{0}^{s_{4}}\int_{0}^{s_{5}}\mathbb{M}_{1}\mathbb{M}_{2}\mathbb{M}_{3}%
\mathbb{M}_{4}\mathbb{M}_{5}\mathbb{M}_{6}u^{(0)}\left( s_{6},\cdot ,\cdot
\right) ds\text{,}
\end{eqnarray*}%
where $ds=ds_{6}ds_{5}ds_{4}ds_{3}ds_{2}ds_{1}$, and then we can obtain
\begin{eqnarray*}
&&\left \Vert \nabla _{x}^{2}u^{(6)}(t,x,\xi )\right \Vert _{L^{2}} \\
&\lesssim &\mathbb{A\cdot }\int_{0}^{t}\int_{0}^{s_{1}}\int_{0}^{s_{2}}%
\int_{0}^{s_{4}}\int_{0}^{s_{5}}\int_{0}^{s_{6}}(1+t-s_{1})^{\frac{2-\gamma
}{\gamma }}(1+s_{1}-s_{2})^{\frac{1}{\gamma }}\cdots \left(
1+s_{5}-s_{6}\right) ^{\frac{1}{\gamma }} \\
&&\cdot \left( 1+\frac{1}{s_{1}-s_{3}}\right) \left( 1+\frac{1}{s_{4}-s_{6}}%
\right) ds \\
&\lesssim &\left \{
\begin{array}{ll}
\mathbb{A}\text{,}%
\vspace {3mm}
& \text{if }-1<\gamma <0\text{,} \\
\mathbb{A}\left( 1+t\right) ^{\varsigma }\text{,}%
\vspace {3mm}
& \text{if }\gamma =-1\text{,} \\
\mathbb{A}\left( 1+t\right) ^{5+\frac{5}{\gamma }}\text{,} & \text{if }%
-2<\gamma <-1\text{.}%
\end{array}%
\right.
\end{eqnarray*}%
by the same argument. The proof of this lemma is completed.
\end{proof}


\subsection{Estimate of the remainder part}

In this subsection, we return to deal with the remainder part $\mathcal{R}%
_{w_{1}}^{(6)}$.

\begin{proposition}[Regularization estimate on $\mathcal{R}_{w_{1}}^{(6)}$]
\label{Regularization estimate on R(6)}Let $\varsigma $ be any positive
number with $0<\varsigma \ll 1$. Then%
\begin{eqnarray*}
&&\left \Vert \mathcal{R}_{w_{1}}^{(6)}\right \Vert _{L_{\xi }^{2}H_{x}^{2}}
\\
&\lesssim &\left( 1+\delta M\right) \left( \eta \left \Vert f_{w_{1}0}\right
\Vert _{L_{\xi ,\beta }^{\infty
}L_{x}^{2}}+C_{g_{1},T}^{2}C_{h_{1},T}^{\infty }\right) \cdot \left \{
\begin{array}{ll}
\left( 1+t\right) ^{2}\text{,}%
\vspace {3mm}
& \text{if }-1<\gamma <0\text{,} \\
\left( 1+t\right) ^{2+\varsigma }\text{,}%
\vspace {3mm}
& \text{if }\gamma =-1\text{,} \\
\left( 1+t\right) ^{7+\frac{5}{\gamma }}\text{,} & \text{if }-2<\gamma <-1%
\text{.}%
\end{array}%
\right.
\end{eqnarray*}
\end{proposition}

\proof%
In view of Lemma \ref{weighted-linear},
\begin{eqnarray*}
\frac{1}{2}\frac{d}{dt}\left \Vert \mathcal{R}_{w_{1}}^{(6)}\right \Vert
_{L_{\xi }^{2}H_{x}^{2}}^{2} &\lesssim &\int_{H_{0}^{D}\cup
H_{-}^{D}}\sum_{|\alpha |\leq 2}\left \vert \partial _{x}^{\alpha }\mathrm{P}%
_{0}\mathcal{R}_{w_{1}}^{(6)}\right \vert ^{2}d\xi dx+\left \Vert \mathcal{R}%
_{w_{1}}^{(6)}\right \Vert _{L_{\xi }^{2}H_{x}^{2}}\left \Vert
K_{w_{1}}u^{(6)}\right \Vert _{L_{\xi }^{2}H_{x}^{2}} \\
&\lesssim &\sum_{|\alpha |\leq 2}\left \Vert w_{1}^{-1}\partial _{x}^{\alpha
}\mathcal{R}_{w_{1}}^{(6)}\right \Vert _{L^{2}}^{2}+\left \Vert \mathcal{R}%
_{w_{1}}^{(6)}\right \Vert _{L_{\xi }^{2}H_{x}^{2}}\left \Vert u^{(6)}\right
\Vert _{L_{\xi }^{2}H_{x}^{2}} \\
&\lesssim &\left \Vert w_{1}^{-1}\mathcal{R}_{w_{1}}^{(6)}\right \Vert
_{L_{\xi }^{2}H_{x}^{2}}^{2}+\left \Vert \mathcal{R}_{w_{1}}^{(6)}\right
\Vert _{L_{\xi }^{2}H_{x}^{2}}\left \Vert u^{(6)}\right \Vert _{L_{\xi
}^{2}H_{x}^{2}}\text{,}
\end{eqnarray*}%
the last inequality being valid since
\begin{equation*}
\sum_{|\alpha |\leq 2}\left \Vert w_{1}^{-1}\partial _{x}^{\alpha }\mathcal{R%
}_{w_{1}}^{(6)}\right \Vert _{L^{2}}\leq C\left \Vert w_{1}^{-1}\mathcal{R}%
_{w_{1}}^{(6)}\right \Vert _{L_{\xi }^{2}H_{x}^{2}}
\end{equation*}%
for some constant $C>0$.

Now we need to estimate $\left \Vert w_{1}^{-1}\mathcal{R}%
_{w_{1}}^{(6)}\right \Vert _{L_{\xi }^{2}H_{x}^{2}}$. Let $z=w_{1}^{-1}%
\mathcal{R}_{w_{1}}^{(6)}$ and then $z$ solves the equation
\begin{equation}
\partial _{t}z+\xi \cdot \nabla _{x}z=Lz+K\left( w_{1}^{-1}u^{(6)}\right)
\text{.}  \label{eq: equation of w-1Rw}
\end{equation}%
By the energy estimate and (\ref{coercivity}), we have%
\begin{eqnarray*}
\frac{1}{2}\frac{d}{dt}\left \Vert z\right \Vert _{L^{2}}^{2} &\leq &\int_{%
\mathbb{R}^{6}}zLzd\xi dx+\left \Vert zKw_{1}^{-1}u^{(6)}\right \Vert
_{L^{2}} \\
&\lesssim &\left \Vert z\right \Vert _{L^{2}}\left \Vert
w_{1}^{-1}u^{(6)}\right \Vert _{L^{2}}\lesssim \left \Vert z\right \Vert
_{L^{2}}\left \Vert u^{(6)}\right \Vert _{L^{2}}\text{ ,}
\end{eqnarray*}%
since $w_{1}(t,x,\xi )\geq 3\left \langle \xi \right \rangle _{D}^{p}\geq 1$%
. Therefore,
\begin{equation}
\frac{d}{dt}\left \Vert z\right \Vert _{L^{2}}\lesssim \left \Vert
u^{(6)}\right \Vert _{L^{2}}\text{.}
\end{equation}%
Moreover, in view of (\ref{eq: equation of w-1Rw}),%
\begin{equation*}
\partial _{t}\left( \partial _{x_{i}}z\right) +\xi \cdot \nabla _{x}\left(
\partial _{x_{i}}z\right) =L\left( \partial _{x_{i}}z\right) +K\left( \left(
\partial _{x_{i}}w_{1}^{-1}\right) u^{(6)}\right) +K\left(
w_{1}^{-1}\partial _{x_{i}}u^{(6)}\right) \text{,}
\end{equation*}%
\begin{eqnarray*}
\partial _{t}\left( \partial _{x_{i}x_{k}}^{2}z\right) +\xi \cdot \nabla
_{x}\left( \partial _{x_{i}x_{k}}^{2}z\right) &=&L\left( \partial
_{x_{i}x_{k}}^{2}z\right) +K\left( \left( \partial
_{x_{i}x_{k}}^{2}w_{1}^{-1}\right) u^{(6)}\right) +K\left( \left( \partial
_{x_{i}}w_{1}^{-1}\right) \partial _{x_{k}}u^{(6)}\right)
\vspace {3mm}
\\
&&+K\left( \left( \partial _{x_{k}}w_{1}^{-1}\right) \partial
_{x_{i}}u^{(6)}\right) +K\left( w_{1}^{-1}\partial
_{x_{i}x_{k}}^{2}u^{(6)}\right) \text{.}
\end{eqnarray*}%
By the energy estimate and (\ref{coercivity}) again, together with Lemma \ref%
{Derivative of wi}, we deduce that
\begin{equation*}
\frac{d}{dt}\sum_{|\alpha |\leq 2}\left \Vert \partial _{x}^{\alpha }z\right
\Vert _{L^{2}}\lesssim \left \Vert u^{(6)}\right \Vert _{L_{\xi
}^{2}H_{x}^{2}}\text{.}
\end{equation*}%
Since $z\left( 0,x,\xi \right) =0$, it implies that
\begin{equation*}
\sum_{|\alpha |\leq 2}\left \Vert w_{1}^{-1}\partial _{x}^{\alpha }\mathcal{R%
}_{w_{1}}^{(6)}\right \Vert _{L^{2}}\lesssim \sum_{|\alpha |\leq 2}\left
\Vert \partial _{x}^{\alpha }z\right \Vert _{L^{2}}\lesssim
\int_{0}^{t}\left \Vert u^{(6)}(\cdot ,s)\right \Vert _{L_{\xi
}^{2}H_{x}^{2}}ds\text{.}
\end{equation*}

According to Lemma \ref{lemma: H2 estimate of h6}, whenever $-1<\gamma <0$,
we have%
\begin{align*}
& \int_{0}^{t}\left \Vert u^{(6)}(s,\cdot )\right \Vert _{L_{\xi
}^{2}H_{x}^{2}}ds \\
& \lesssim \int_{0}^{t}\left( 1+\delta M\right) \left[ \eta \left \Vert
f_{w_{1}0}\right \Vert _{L_{\xi ,\beta }^{\infty
}L_{x}^{2}}+(1+s)^{-1}C_{g_{1},T}^{2}C_{h_{1},T}^{\infty }\right] ds \\
& \lesssim \left( 1+\delta M\right) \left( 1+t\right) \left( \eta \left
\Vert f_{w_{1}0}\right \Vert _{L_{\xi ,\beta }^{\infty
}L_{x}^{2}}+C_{g_{1},T}^{2}C_{h_{1},T}^{\infty }\right) \text{,}
\end{align*}%
and thus
\begin{eqnarray*}
\frac{d}{dt}\left \Vert \mathcal{R}_{w_{1}}^{(6)}\right \Vert _{L_{\xi
}^{2}H_{x}^{2}}^{2} &\lesssim &\left( 1+t\right) ^{2}\left[ \left( 1+\delta
M\right) \left( \eta \left \Vert f_{w_{1}0}\right \Vert _{L_{\xi ,\beta
}^{\infty }L_{x}^{2}}+C_{g_{1},T}^{2}C_{h_{1},T}^{\infty }\right) \right]
^{2} \\
&&+\left \Vert \mathcal{R}_{w_{1}}^{(6)}\right \Vert _{L_{\xi
}^{2}H_{x}^{2}}\left( 1+t\right) \left( 1+\delta M\right) \left( \eta \left
\Vert f_{w_{1}0}\right \Vert _{L_{\xi ,\beta }^{\infty
}L_{x}^{2}}+C_{g_{1},T}^{2}C_{h_{1},T}^{\infty }\right) \text{.}
\end{eqnarray*}%
As a consequence,
\begin{equation*}
\left \Vert \mathcal{R}_{w_{1}}^{(6)}\right \Vert _{L_{\xi
}^{2}H_{x}^{2}}\lesssim \left( 1+t\right) ^{2}\left( 1+\delta M\right)
\left( \eta \left \Vert f_{w_{1}0}\right \Vert _{L_{\xi ,\beta }^{\infty
}L_{x}^{2}}+C_{g_{1},T}^{2}C_{h_{1},T}^{\infty }\right)
\end{equation*}%
for $-1<\gamma <0$. The other cases $-2<\gamma <-1$ and $\gamma =-1$ can be
obtained by the same argument and the proof of the proposition is completed.$%
\hfill%
\square $

Therefore, in view of Proposition \ref{Regularization estimate on R(6)}, the
Sobolev inequality implies that
\begin{eqnarray}
&&\left \Vert \mathcal{R}_{w_{1}}^{(6)}\right \Vert _{L_{\xi
}^{2}L_{x}^{\infty }}  \notag  \label{R^(6)-L2_xi-L^inf_x} \\
&\lesssim &\left \Vert \mathcal{R}_{w_{1}}^{(6)}\right \Vert _{L_{\xi
}^{2}H_{x}^{2}}^{3/4}\left \Vert \mathcal{R}_{w_{1}}^{(6)}\right \Vert
_{L^{2}}^{1/4} \\
&\lesssim &\left( 1+\delta M\right) \left( \eta \left \Vert f_{w_{1}0}\right
\Vert _{L_{\xi ,\beta }^{\infty
}L_{x}^{2}}+C_{g_{1},T}^{2}C_{h_{1},T}^{\infty }\right) \cdot \left \{
\begin{array}{ll}
\left( 1+t\right) ^{2}\text{,}%
\vspace {3mm}
& \text{if }-1<\gamma <0\text{,} \\
\left( 1+t\right) ^{2+\varsigma }\text{,}%
\vspace {3mm}
& \text{if }\gamma =-1\text{,} \\
\left( 1+t\right) ^{7+\frac{5}{\gamma }}\text{,} & \text{if }-2<\gamma <-1%
\text{,}%
\end{array}%
\right.  \notag
\end{eqnarray}%
for any $0<\varsigma \ll 1$. Together with Lemma \ref{pointwise-u}, we
obtain the estimate of $\left \Vert u\right \Vert _{L_{\xi
}^{2}L_{x}^{\infty }}$. Subsequently, we shall get the estimate for $%
\left
\Vert u\right \Vert _{L_{\xi ,\beta }^{\infty }L_{x}^{\infty }}$ via
the bootstrap argument and the details are given in the proof of Theorem \ref%
{theorem-linear}.

\subsection{Proof of Theorem \protect \ref{theorem-linear}}

Note that if $i\geq 2$, the solution $u$ to (\ref{inhom}) can be represented
as
\begin{equation*}
u(t,x,\xi )=W_{w_{1}}^{(i)}+\int_{0}^{t}\mathbb{S}_{w_{1}}(t,s)K_{w_{1}}%
\mathcal{R}_{w_{1}}^{(i-1)}(s)ds=W_{w_{1}}^{(i)}+\mathcal{R}_{w_{1}}^{(i)}%
\text{.}
\end{equation*}%
In view of Lemma \ref{pointwise-u}, it remains to estimate $\left \Vert
\mathcal{R}_{w_{1}}^{(i)}\right \Vert _{L_{\xi ,\beta }^{\infty
}L_{x}^{\infty }}$ for some $i$ in order to obtain the estimate for $%
\left
\Vert u\right \Vert _{L_{\xi ,\beta }^{\infty }L_{x}^{\infty }}$. To
obtain the estimate for $\left \Vert \mathcal{R}_{w_{1}}^{(i)}\right
\Vert
_{L_{\xi ,\beta }^{\infty }L_{x}^{\infty }}$, we consider $\gamma $ in two
different cases: $-2<\gamma \leq -3/2$ and $-3/2<\gamma <0$.\

In the case $-2<\gamma \leq -3/2$, in view of (\ref{eq: bootstrap 2}), (\ref%
{eq: bootstrap 3}) and (\ref{S2}), together with the fact that%
\begin{equation*}
\left \Vert \mathbb{S}_{w_{1}}(s,s_{1})q(s_{1},x,\xi )\right \Vert _{L_{\xi
}^{4}L_{x}^{\infty }}\lesssim (1+s-s_{1})^{\frac{1-\gamma }{\gamma }}\left
\Vert q(s_{1},\cdot ,\cdot )\right \Vert _{L_{\xi ,1-\gamma
}^{4}L_{x}^{\infty }}\text{,}
\end{equation*}%
we have%
\begin{align*}
\left \Vert \mathcal{R}_{w_{1}}^{(8)}\right \Vert _{L_{\xi }^{\infty
}L_{x}^{\infty }}& =\left \Vert \int_{0}^{t}\int_{0}^{s}\mathbb{S}%
_{w_{1}}(t,s)\left[ K_{w_{1}}\right] _{s}\mathbb{S}_{w_{1}}(s,s_{1})\left[
K_{w_{1}}\right] _{s_{1}}\mathcal{R}_{w_{1}}^{(6)}(s_{1})ds_{1}ds\right
\Vert _{L_{\xi }^{\infty }L_{x}^{\infty }} \\
& \lesssim \int_{0}^{t}\int_{0}^{s}(1+t-s)^{\frac{7/4-\gamma }{\gamma }%
}(1+s-s_{1})^{\frac{1-\gamma }{\gamma }}\left \Vert \mathcal{R}%
_{w_{1}}^{(6)}(s_{1})\right \Vert _{L_{\xi }^{2}L_{x}^{\infty }}ds_{1}ds \\
& \lesssim \int_{0}^{t}\int_{0}^{s}(1+t-s)^{\frac{7/4-\gamma }{\gamma }%
}(1+s-s_{1})^{\frac{1-\gamma }{\gamma }}\left( 1+s_{1}\right) ^{7+\frac{5}{%
\gamma }} \\
& \quad \quad \cdot \left( 1+\delta M\right) \left( \eta \left \Vert
f_{w_{1}0}\right \Vert _{L_{\xi ,\beta }^{\infty
}L_{x}^{2}}+C_{g_{1},T}^{2}C_{h_{1},T}^{\infty }\right) ds_{1}ds \\
& \lesssim \left( 1+t\right) ^{7+\frac{5}{\gamma }}\left( 1+\delta M\right)
\left( \eta \left \Vert f_{w_{1}0}\right \Vert _{L_{\xi ,\beta }^{\infty
}L_{x}^{2}}+C_{g_{1},T}^{2}C_{h_{1},T}^{\infty }\right) \text{.}
\end{align*}%
Combining this with Lemma \ref{pointwise-u}, it follows
\begin{align*}
\left \Vert u\right \Vert _{L_{\xi }^{\infty }L_{x}^{\infty }}& \leq \left
\Vert W_{w_{1}}^{(8)}\right \Vert _{L_{\xi }^{\infty }L_{x}^{\infty }}+\left
\Vert \mathcal{R}_{w_{1}}^{(8)}\right \Vert _{L_{\xi }^{\infty
}L_{x}^{\infty }} \\
& \lesssim \eta \left \Vert f_{w_{1}0}\right \Vert _{L_{\xi ,\beta }^{\infty
}L_{x}^{\infty }}+(1+t)^{-3/2+A+\varsigma }C_{g_{1},T}^{\infty
}C_{h_{1},T}^{\infty }\text{ } \\
& \quad +(1+t)^{7+\frac{5}{\gamma }}\left( 1+\delta M\right) \left( \eta
\left \Vert f_{w_{1}0}\right \Vert _{L_{\xi ,\beta }^{\infty
}L_{x}^{2}}+C_{g_{1},T}^{2}C_{h_{1},T}^{\infty }\right) \text{.}
\end{align*}%
Note that
\begin{equation}
u(t,x,\xi )=u^{\left( 0\right) }\left( t,x,\xi \right) +\int_{0}^{t}\mathbb{S%
}_{w_{1}}(t,s)K_{w_{1}}u(s)ds\text{.}  \label{u-implicit}
\end{equation}%
Hence, through (\ref{Kw6}), (\ref{S2}) and Lemma \ref{pointwise-u}, we infer
\begin{align*}
\left \Vert u\right \Vert _{L_{\xi ,\beta }^{\infty }L_{x}^{\infty }}&
\lesssim \eta \left \Vert f_{w_{1}0}\right \Vert _{L_{\xi ,\beta }^{\infty
}L_{x}^{\infty }}+(1+t)^{-3/2+A+\varsigma }C_{g_{1},T}^{\infty
}C_{h_{1},T}^{\infty }\text{ } \\
& \quad +(1+t)^{7+\frac{5}{\gamma }}\left( 1+\delta M\right) \left( \eta
\left \Vert f_{w_{1}0}\right \Vert _{L_{\xi ,\beta }^{\infty
}L_{x}^{2}}+C_{g_{1},T}^{2}C_{h_{1},T}^{\infty }\right) \text{,}
\end{align*}%
by the bootstrap argument.

In the case $-3/2<\gamma <0$, we decompose $u$ as $u=W_{w_{1}}^{(7)}+%
\mathcal{R}_{w_{1}}^{(7)}$. In view of (\ref{eq: bootstrap 1}) and (\ref{S2}%
),
\begin{eqnarray*}
\left \Vert \mathcal{R}_{w_{1}}^{(7)}\right \Vert _{L_{\xi }^{\infty
}L_{x}^{\infty }} &=&\left \Vert \int_{0}^{t}\mathbb{S}_{w_{1}}(t,s)\left[
K_{w_{1}}\right] _{s}\mathcal{R}_{w_{1}}^{(6)}(s)ds\right \Vert _{L_{\xi
}^{\infty }L_{x}^{\infty }} \\
&\lesssim &\int_{0}^{t}\left( 1+t-s\right) ^{\frac{\frac{3}{2}-\gamma }{%
\gamma }}\left \Vert \mathcal{R}_{w_{1}}^{(6)}\right \Vert _{L_{\xi
}^{2}L_{x}^{\infty }}ds\text{.}
\end{eqnarray*}%
Hence, we obtain the estimate of $\left \Vert u\right \Vert _{L_{\xi
}^{\infty }L_{x}^{\infty }}$ by using (\ref{R^(6)-L2_xi-L^inf_x}) and Lemma %
\ref{pointwise-u}. Again, through (\ref{Kw6}), (\ref{S2}), (\ref{u-implicit}%
), Lemma \ref{pointwise-u}, we can conclude that
\begin{align*}
\left \Vert u\right \Vert _{L_{\xi ,\beta }^{\infty }L_{x}^{\infty }}&
\lesssim \eta \left \Vert f_{w_{1}0}\right \Vert _{L_{\xi ,\beta }^{\infty
}L_{x}^{\infty }}+(1+t)^{-\frac{3}{2}+A+\varsigma }C_{g_{1,}T}^{\infty
}C_{h_{1,}T}^{\infty } \\
& \quad +\left( 1+\delta M\right) \left( \eta \left \Vert f_{w_{1}0}\right
\Vert _{L_{\xi ,\beta }^{\infty
}L_{x}^{2}}+C_{g_{1},T}^{2}C_{h_{1},T}^{\infty }\right) \cdot \left \{
\begin{array}{ll}
(1+t)^{2}\text{,}%
\vspace {3mm}
& \text{if }-1<\gamma <0\text{,} \\
(1+t)^{2+\varsigma }\text{,}%
\vspace {3mm}
& \text{if }\gamma =-1\text{,} \\
(1+t)^{7+\frac{5}{\gamma }}\text{,} & \text{if }\frac{-3}{2}<\gamma <-1\text{%
,}%
\end{array}%
\right.
\end{align*}%
by the bootstrap argument. This completes the proof of Theorem \ref%
{theorem-linear}.$%
\hfill%
\square $

\subsection{The result for the exponential weight function $w_{2}$}

For the exponential weight, we consider the inhomogeneous equation:
\begin{equation}
\left \{
\begin{array}{l}
\displaystyle \pa_{t}v+\xi \cdot \nabla _{x}v-\epsilon \left[ \partial
_{t}\rho +\xi \cdot \nabla _{x}\rho \right] v=L_{w_{2}}v+\Gamma
_{w_{2}}(g_{2},h_{2})\text{,}\, \\[4mm]
\displaystyle v(0,x,\xi )=\eta f_{w_{2}0}\text{.}%
\end{array}%
\right.  \label{inhom-exp}
\end{equation}%
Let $T>0$, $\beta >3/2$, and $0<p\leq 2$. Assume that $f_{w_{2}0}\in L_{\xi
,\beta }^{\infty }L_{x}^{2}\cap L_{\xi ,\beta }^{\infty }L_{x}^{\infty }$.
Also assume that $g_{2}$ and $h_{2}$ satisfy
\begin{equation}
\hat{C}_{g_{2},T}^{\infty }=\sup_{0\leq t\leq T}(1+t)^{-B}\left \Vert
g_{2}\right \Vert _{L_{\xi ,\beta }^{\infty }L_{x}^{\infty }}<\infty
\,,\quad \hat{C}_{g_{2},T}^{2}=\sup_{0\leq t\leq T}\left \Vert g_{2}\right
\Vert _{L_{\xi ,\beta }^{\infty }L_{x}^{2}}<\infty \text{,}  \label{h1}
\end{equation}%
for some constant $B>0$, and
\begin{equation}
\hat{C}_{h_{2},T}^{\infty }=\sup_{0\leq t\leq T}(1+t)^{\frac{3}{2}}\left
\Vert \left \langle \xi \right \rangle ^{p}e^{\epsilon c_{p}\left \langle
\xi \right \rangle ^{p}}h_{2}\right \Vert _{L_{\xi ,\beta }^{\infty
}L_{x}^{\infty }}<\infty \, \text{.}  \label{h2}
\end{equation}%
Here the constant $c_{p}>0$ is the same as in Lemma \ref{ratio}.

Under these assumptions, following a similar argument as in the previous
subsections, one can get the following theorem:

\begin{theorem}
\label{theorem-linear-exp}Let $\beta >3/2$ and $0<\varsigma \ll 1$. Assume
that $f_{w_{2}0}\in L_{\xi ,\beta }^{\infty }L_{x}^{2}\cap L_{\xi ,\beta
}^{\infty }L_{x}^{\infty }$, and that $g_{2},h_{2}$ satisfy $(\ref{h1})$ and
$(\ref{h2})$, respectively. Then the solution $v$ to $(\ref{inhom-exp})$
satisfies
\begin{align}
\left \Vert v\right \Vert _{L_{\xi ,\beta }^{\infty }L_{x}^{\infty }}&
\lesssim \eta \left \Vert f_{w_{2}0}\right \Vert _{L_{\xi ,\beta }^{\infty
}L_{x}^{\infty }}+(1+t)^{-\frac{3}{2}+B+\varsigma }\hat{C}_{g_{2},T}^{\infty
}\hat{C}_{h_{2},T}^{\infty }  \label{inhom-estimate-exp} \\
& \quad +\left( 1+\delta M\right) \left( \eta \left \Vert f_{w_{2}0}\right
\Vert _{L_{\xi ,\beta }^{\infty }L_{x}^{2}}+\hat{C}_{g_{2},T}^{2}\hat{C}%
_{h_{2},T}^{\infty }\right) \cdot \left \{
\begin{array}{ll}
(1+t)^{2}\text{,}%
\vspace {3mm}
& \text{if }-1<\gamma <0\text{,} \\
(1+t)^{2+\varsigma }\text{,}%
\vspace {3mm}
& \text{if }\gamma =-1\text{,} \\
(1+t)^{7+\frac{5}{\gamma }}\text{,} & \text{if }-2<\gamma <-1\text{,}%
\end{array}%
\right.  \notag
\end{align}%
for $0\leq t\leq T$.
\end{theorem}

\section{Proof of Theorem \protect \ref{thm:main}}

\label{Proof: thm:main}

To demonstrate Theorem \ref{thm:main}, we also need to study the large time
behavior of the solution of the Boltzmann equation (\ref{bot.1.d}) in
suitable velocity weight. The result is stated in Theorem \ref{prop:
nonlinear} but we postpone its proof to the next section. With the help of
Theorems \ref{prop: nonlinear}, \ref{theorem-linear}, and \ref%
{theorem-linear-exp}, we are in the position to prove Theorem \ref{thm:main}.

\subsection{Proof of Theorem \protect \ref{thm:main}}

In light of Theorem \ref{theorem-linear}, we need to estimate $\left \Vert
f_{w_{1}}\right \Vert _{L_{\xi ,\beta }^{\infty }L_{x}^{2}}$ and $%
\left
\Vert f_{w_{2}}\right \Vert _{L_{\xi ,\beta }^{\infty }L_{x}^{2}}$.
Let $T>0$. In view of (\ref{Gamma-w1-w1-product}) and $\left \langle \xi
\right
\rangle ^{\gamma /2}\leq 1$, we get
\begin{align*}
& \quad \int_{{\mathbb{R}}^{3}}\int_{{\mathbb{R}}^{3}} f_{w_{1}}\Gamma_{w_{1}}(f_{w_{1}},f)dxdv \\
& \lesssim \left \Vert f_{w_{1}}\right \Vert _{L_{\sigma
}^{2}L_{x}^{2}}\left( \left \Vert f_{w_{1}}\right \Vert _{L_{\sigma
}^{2}L_{x}^{2}}\left \Vert \left \langle \xi \right \rangle ^{p}f\right
\Vert _{L_{\xi }^{\infty }L_{x}^{\infty }}+\left \Vert f_{w_{1}}\right \Vert
_{L_{\xi }^{\infty }L_{x}^{2}}\left \Vert \left \langle \xi \right \rangle
^{p}f\right \Vert _{L_{\sigma }^{2}L_{x}^{\infty }}\right) \\
& \lesssim \left( \left \Vert \mathrm{P}_{1}f_{w_{1}}\right \Vert
_{L_{\sigma }^{2}L_{x}^{2}}^{2}+\left \Vert \mathrm{P}_{0}f_{w_{1}}\right
\Vert _{L_{\sigma }^{2}L_{x}^{2}}^{2}\right) \left \Vert \left \langle \xi
\right \rangle ^{p}f\right \Vert _{L_{\xi }^{\infty }L_{x}^{\infty }}+\left
\Vert f_{w_{1}}\right \Vert _{L_{\sigma }^{2}L_{x}^{2}}\left \Vert
f_{w_{1}}\right \Vert _{L_{\xi }^{\infty }L_{x}^{2}}\left \Vert \left
\langle \xi \right \rangle ^{p}f\right \Vert _{L_{\sigma }^{2}L_{x}^{\infty
}} \\
& \lesssim \left \Vert \left \langle \xi \right \rangle ^{p+\beta +\gamma
/2}f\left( t\right) \right \Vert _{L_{\xi }^{\infty }L_{x}^{\infty }}\left(
\left \Vert \mathrm{P}_{1}f_{w_{1}}\right \Vert _{L_{\sigma
}^{2}L_{x}^{2}}^{2}+\left[ C_{f_{w_{1},}T}^{2}\right] ^{2}\right) \text{,}
\end{align*}%
for $0\leq t\leq T$, where $C_{f_{w_{1},}T}^{2}:=\sup_{0\leq t\leq
T}\left
\Vert f_{w_{1}}\right \Vert _{L_{\xi ,\beta }^{\infty }L_{x}^{2}}$.
By Theorem \ref{prop: nonlinear} with $\hat{\varepsilon}=0$,%
\begin{eqnarray*}
\left \Vert \left \langle \xi \right \rangle ^{p+\beta +\gamma /2}f\left(
t\right) \right \Vert _{L_{\xi }^{\infty }L_{x}^{\infty }} &\leq &\left
\Vert \left \langle \xi \right \rangle ^{p+\beta }f\left( t\right) \right
\Vert _{L_{\xi }^{\infty }L_{x}^{\infty }} \\
&\lesssim &\eta \left( 1+t\right) ^{-\frac{3}{2}}\left( \left \Vert
f_{0}\right \Vert _{L_{\xi ,p+\beta +3j}^{\infty }L_{x}^{1}}+\left \Vert
f_{0}\right \Vert _{L_{\xi ,p+\beta +3j}^{\infty }L_{x}^{\infty }}\right) \\
&\lesssim &\eta \left( 1+t\right) ^{-\frac{3}{2}}\left \Vert f_{0}\right
\Vert _{L_{\xi ,p+\beta +3j}^{\infty }L_{x}^{\infty }}\text{,}
\end{eqnarray*}%
since $f_{0}\left( \cdot ,\xi \right) $ has compact support contained in the
unit ball centered at the origin for all $\xi $. After choosing $\delta $, $%
\eta >0$ sufficiently small, $D$, $M\geq 1$ sufficiently large with $\delta
M $ $\ll \nu _{0}$ and $D^{-1}\ll \delta M$, it follows from Lemma \ref%
{weighted-linear} and Theorem \ref{prop: nonlinear} that
\begin{align*}
\frac{d}{dt}\left \Vert f_{w_{1}}\right \Vert _{L^{2}}^{2}& \leq -\left( \nu
_{0}-C_{1}D^{-2}-C_{2}\delta -C_{3}\delta M-C_{6}\eta \left \Vert
f_{0}\right \Vert _{L_{\xi ,p+\beta +3j}^{\infty }L_{x}^{\infty }}\right)
\int_{{\mathbb{R}}^{3}}\int_{{\mathbb{R}}^{3}} \left \langle \xi \right \rangle ^{\gamma}%
\left( \mathrm{P}_{1}f_{w_{1}} \right)^{2} dxd\xi \\
& \quad -\left( C_{4}\delta M-C_{2}\delta -C_{1}D^{-2}\right)
\int_{H_{+}^{D}}\left[ \delta \left( \left \langle x\right \rangle
-Mt\right) \right] ^{-1}\left \vert \mathrm{P}_{0}f_{w_{1}}\right \vert
^{2} dx d\xi \\
& \quad +\left( C_{1}D^{-2}+C_{2}\delta +C_{5}\delta M\right)
\int_{H_{0}^{D}}\left \vert \mathrm{P}_{0}f_{w_{1}}\right \vert ^{2}
dx d\xi+C_{1}D^{-2}\int_{H_{-}^{D}}\left \vert \mathrm{P}_{0}f_{w_{1}}\right
\vert ^{2}dx d\xi \\
& \quad +C_{6}\eta \left( 1+t\right) ^{-\frac{3}{2}}\left \Vert f_{0}\right
\Vert _{L_{\xi ,p+\beta +3j}^{\infty }L_{x}^{\infty }}\left[
C_{f_{w_{1},}T}^{2}\right] ^{2} \\
& \leq C_{6}\eta \left( 1+t\right) ^{-\frac{3}{2}}\left \Vert f_{0}\right
\Vert _{L_{\xi ,p+\beta +3j}^{\infty }L_{x}^{\infty }}\left[
C_{f_{w_{1},}T}^{2}\right] ^{2}+C_{8}\left \Vert f\right \Vert _{L^{2}}^{2}
\\
& \lesssim \eta \left( 1+t\right) ^{-\frac{3}{2}}\left \Vert f_{0}\right
\Vert _{L_{\xi ,p+\beta +3j}^{\infty }L_{x}^{\infty }}\left[
C_{f_{w_{1},}T}^{2}\right] ^{2}+\eta ^{2}\left( 1+t\right) ^{-3/2}\left
\Vert f_{0}\right \Vert _{L_{\xi ,p+\beta +3j}^{\infty }L_{x}^{\infty }}^{2}%
\text{,}
\end{align*}%
the last inequality being valid since $\beta >\frac{3}{2}$ and $f_{0}\left(
\cdot ,\xi \right) $ has compact support contained in the unit ball centered
at the origin for all $\xi $. Therefore, for $0\leq t\leq T$,
\begin{eqnarray}
\left \Vert f_{w_{1}}\left( t\right) \right \Vert _{L^{2}}\, &\lesssim &\eta
\left \Vert f_{w_{1}0}\right \Vert _{L^{2}}+\eta ^{1/2}\left \Vert
f_{0}\right \Vert _{L_{\xi ,p+\beta +3j}^{\infty }L_{x}^{\infty
}}^{1/2}C_{f_{w_{1},}T}^{2}+\eta \left \Vert f_{0}\right \Vert _{L_{\xi
,p+\beta +3j}^{\infty }L_{x}^{\infty }}  \label{fw1-L2} \\
&\lesssim &\eta \left \Vert f_{0}\right \Vert _{L_{\xi ,p+\beta +3j}^{\infty
}L_{x}^{\infty }}+\eta ^{1/2}\left \Vert f_{0}\right \Vert _{L_{\xi ,p+\beta
+3j}^{\infty }L_{x}^{\infty }}^{1/2}C_{f_{w_{1},}T}^{2}\text{.}  \notag
\end{eqnarray}

Next, In terms of the operator $\mathbb{S}_{w_{1}}(t;s)$, $f_{w_{1}}$ can be
rewritten as
\begin{equation}
f_{w_{1}}\left( t\right) =\eta \mathbb{S}_{w_{1}}(t)f_{w_{1}0}+\int_{0}^{t}%
\mathbb{S}_{w_{1}}(t;s)K_{w_{1}}f_{w_{1}}(s)+\mathbb{S}_{w_{1}}(t;s)\Gamma
_{w_{1}}(f_{w_{1}},f)(s)ds\text{,}  \label{fw1-rep}
\end{equation}%
for $0\leq t\leq T$. In the sequel, we shall utilize this representation to
establish the estimate for $\left \Vert f_{w_{1}}(t)\right \Vert _{L_{\xi
,\beta }^{\infty }L_{x}^{2}}$ in two cases $-3/2<\gamma <0$ and $-2<\gamma
\leq -3/2$ separately.$%
\vspace {3mm}%
$ \newline
\textbf{Case I: }$-3/2<\gamma <0$\textbf{.} By (\ref{eq: N. 1. a}), (\ref%
{eq: bootstrap 1}), (\ref{Gamma-w1-sup}), (\ref{S5}), and (\ref{fw1-L2}),
\begin{eqnarray*}
&&\left \Vert f_{w_{1}}(t)\right \Vert _{L_{\xi }^{\infty }L_{x}^{2}} \\
&\leq &\eta \left \Vert \mathbb{S}_{w_{1}}(t)f_{w_{1}0}\right \Vert _{L_{\xi
}^{\infty }L_{x}^{2}}+\int_{0}^{t}\left \Vert \mathbb{S}%
_{w_{1}}(t;s)K_{w_{1}}f_{w_{1}}(s)\right \Vert _{L_{\xi }^{\infty
}L_{x}^{2}}+\left \Vert \mathbb{S}_{w_{1}}(t;s)\Gamma
_{w_{1}}(f_{w_{1}},f)(s)\right \Vert _{L_{\xi }^{\infty }L_{x}^{2}}ds \\
&\leq &\eta \left \Vert f_{w_{1}0}\right \Vert _{L_{\xi }^{\infty
}L_{x}^{2}}+C_{\gamma ,p}\int_{0}^{t}(1+t-s)^{\frac{3/2-\gamma }{\gamma }%
}\left \Vert f_{w_{1}}(s)\right \Vert _{L^{2}}ds \\
&&+C_{\gamma ,p}\int_{0}^{t}(1+t-s)^{-\frac{\gamma }{\gamma }}\left \Vert
f_{w_{1}}\left( s\right) \right \Vert _{L_{\xi }^{\infty }L_{x}^{2}}\left
\Vert \left \langle \xi \right \rangle ^{p}f\left( s\right) \right \Vert
_{L_{\xi }^{\infty }L_{x}^{\infty }}ds\text{ } \\
&\leq &\eta \left \Vert f_{w_{1}0}\right \Vert _{L_{\xi }^{\infty
}L_{x}^{2}}+C_{1,\gamma ,p,\beta ,j}\left[ \eta \left \Vert f_{0}\right
\Vert _{L_{\xi ,p+\beta +3j}^{\infty }L_{x}^{\infty }}+\eta ^{1/2}\left
\Vert f_{0}\right \Vert _{L_{\xi ,p+\beta +3j}^{\infty }L_{x}^{\infty
}}^{1/2}C_{f_{w_{1},}T}^{2}\right] \\
&&+C_{1,\gamma ,p,\beta ,j}\int_{0}^{t}(1+t-s)^{-\frac{\gamma }{\gamma }%
}\eta (1+s)^{-\frac{3}{2}}\left( \left \Vert f_{0}\right \Vert _{L_{\xi
,p+\beta +3j}^{\infty }L_{x}^{\infty }}\right) ds\cdot \sup_{0\leq s\leq
T}\left \Vert f_{w_{1}}(s)\right \Vert _{L_{\xi }^{\infty }L_{x}^{2}}\text{ }
\\
&\leq &C_{1,\gamma ,p,\beta ,j}^{\prime }\left[ \eta \left \Vert f_{0}\right
\Vert _{L_{\xi ,p+\beta +3j}^{\infty }L_{x}^{\infty }}+\eta ^{1/2}\left
\Vert f_{0}\right \Vert _{L_{\xi ,p+\beta +3j}^{\infty }L_{x}^{\infty
}}^{1/2}C_{f_{w_{1},}T}^{2}\right] \\
&&+C_{1,\gamma ,p,\beta ,j}^{\prime }\eta \left( \left \Vert f_{0}\right
\Vert _{L_{\xi ,p+\beta +3j}^{\infty }L_{x}^{\infty }}\right) \cdot
\sup_{0\leq s\leq T}\left \Vert f_{w_{1}}(s)\right \Vert _{L_{\xi }^{\infty
}L_{x}^{2}}\text{.}
\end{eqnarray*}%
Since $\eta >0$ is sufficiently small such that $C_{1,\gamma ,p,\beta
,j}^{\prime }\eta \left( \left \Vert f_{0}\right \Vert _{L_{\xi
,p+3j}^{\infty }L_{x}^{\infty }}\right) <1/2$, it follows that
\begin{equation*}
\left \Vert f_{w_{1}}(t)\right \Vert _{L_{\xi }^{\infty }L_{x}^{2}}\leq
2C_{1,\gamma ,p,\beta ,j}^{\prime }\left[ \eta \left \Vert f_{0}\right \Vert
_{L_{\xi ,p+\beta +3j}^{\infty }L_{x}^{\infty }}+\eta ^{1/2}\left \Vert
f_{0}\right \Vert _{L_{\xi ,p+\beta +3j}^{\infty }L_{x}^{\infty
}}^{1/2}C_{f_{w_{1},}T}^{2}\right] \text{.}
\end{equation*}%
In view of (\ref{fw1-rep}), we use a bootstrap argument to obtain
\begin{equation*}
\left \Vert f_{w_{1}}(t)\right \Vert _{L_{\xi ,\beta }^{\infty
}L_{x}^{2}}\leq C_{1,\gamma ,p,\beta ,j}^{\prime \prime }\left[ \eta \left
\Vert f_{0}\right \Vert _{L_{\xi ,p+\beta +3j}^{\infty }L_{x}^{\infty
}}+\eta ^{1/2}\left \Vert f_{0}\right \Vert _{L_{\xi ,p+\beta +3j}^{\infty
}L_{x}^{\infty }}^{1/2}C_{f_{w_{1},}T}^{2}\right]
\end{equation*}%
for $0\leq t\leq T$, through (\ref{Kw6}), (\ref{Gamma-w1-sup}) and (\ref{S5}%
). Since $\eta >0$ is sufficiently small, we have
\begin{equation*}
\left \Vert f_{w_{1}}(t)\right \Vert _{L_{\xi ,\beta }^{\infty
}L_{x}^{2}}\leq C_{f_{w_{1},}T}^{2}\leq C_{1,\gamma ,p,\beta ,j}^{\prime
\prime \prime }\eta \left \Vert f_{0}\right \Vert _{L_{\xi ,p+\beta
+3j}^{\infty }L_{x}^{\infty }}\text{.}
\end{equation*}%
$%
\vspace {3mm}%
$\newline
\textbf{Case II:} $-2<\gamma \leq -3/2$. Utilizing (\ref{Gamma-w1-sup}) and (%
\ref{Gamma-w1-L2}) with $\beta >3/2$ gives
\begin{eqnarray}
&&\left \Vert \exp \left( -\nu (\xi )(t-s)\right) \Gamma
_{w_{1}}(f_{w_{1}},f)(s)\right \Vert _{L_{\xi }^{4}L_{x}^{2}}  \label{eq: g1}
\\
&\leq &\left \Vert \exp \left( \frac{-\nu (\xi )(t-s)}{2}\right) \Gamma
_{w_{1}}(f_{w_{1}},f)(s)\right \Vert _{L_{\xi }^{2}L_{x}^{2}}^{1/2}\left
\Vert \exp \left( \frac{-\nu (\xi )(t-s)}{2}\right) \Gamma
_{w_{1}}(f_{w_{1}},f)(s)\right \Vert _{L_{\xi }^{\infty }L_{x}^{2}}^{1/2}
\notag \\
&\leq &C_{\gamma ,p,\beta }(1+t-s)^{-\frac{\gamma }{\gamma }}\left \Vert
f_{w_{1}}\right \Vert _{L_{\xi ,\beta }^{\infty }L_{x}^{2}}\left \Vert \left
\langle \xi \right \rangle ^{p}f\right \Vert _{L_{\xi ,\beta }^{\infty
}L_{x}^{\infty }}\text{.}  \notag
\end{eqnarray}%
Therefore, through (\ref{fw1-rep}), we have%
\begin{eqnarray*}
&&\left \Vert f_{w_{1}}(t)\right \Vert _{L_{\xi }^{4}L_{x}^{2}} \\
&\leq &\eta \left \Vert \mathbb{S}_{w_{1}}\left( t\right) f_{w_{1}0}\right
\Vert _{L_{\xi }^{4}L_{x}^{2}}+\int_{0}^{t}\left \Vert \mathbb{S}%
_{w_{1}}(t;s)K_{w_{1}}f_{w_{1}}(s)\right \Vert _{L_{\xi
}^{4}L_{x}^{2}}+\left \Vert \mathbb{S}_{w_{1}}(t;s)\Gamma
_{w_{1}}(f_{w_{1}},f)(s)\right \Vert _{L_{\xi }^{4}L_{x}^{2}}ds \\
&\leq &\eta \left \Vert f_{w_{1}0}\right \Vert _{L_{\xi
}^{4}L_{x}^{2}}+C_{\gamma ,p}\int_{0}^{t}(1+t-s)^{\frac{1-\gamma }{\gamma }%
}\left \Vert f_{w_{1}}(s)\right \Vert _{L_{\xi }^{2}L_{x}^{2}}ds \\
&&+C_{\gamma ,p}\int_{0}^{t}(1+t-s)^{-\frac{\gamma }{\gamma }}\left \Vert
f_{w_{1}}\right \Vert _{L_{\xi ,\beta }^{\infty }L_{x}^{2}}\left \Vert \left
\langle \xi \right \rangle ^{p}f\right \Vert _{L_{\xi ,\beta }^{\infty
}L_{x}^{\infty }}ds\text{ } \\
&\leq &\eta \left \Vert f_{w_{1}0}\right \Vert _{L_{\xi
}^{4}L_{x}^{2}}+C_{1,\gamma ,p,\beta ,j}\left( \eta \left \Vert f_{0}\right
\Vert _{L_{\xi ,p+\beta +3j}^{\infty }L_{x}^{\infty }}+\eta ^{1/2}\left
\Vert f_{0}\right \Vert _{L_{\xi ,p+\beta +3j}^{\infty }L_{x}^{\infty
}}^{1/2}C_{f_{w_{1},}T}^{2}\right) \\
&&+C_{1,\gamma ,p,\beta ,j}\int_{0}^{t}(1+t-s)^{-\frac{\gamma }{\gamma }%
}(1+s)^{-\frac{3}{2}}ds\cdot \left( \eta \left \Vert f_{0}\right \Vert
_{L_{\xi ,p+\beta +3j}^{\infty }L_{x}^{\infty }}\right) C_{f_{w_{1},}T}^{2}
\\
&\leq &C_{1,\gamma ,p,\beta ,j}^{\prime }\left[ \eta \left \Vert f_{0}\right
\Vert _{L_{\xi ,p+\beta +3j}^{\infty }L_{x}^{\infty }}+\eta ^{1/2}\left
\Vert f_{0}\right \Vert _{L_{\xi ,p+\beta +3j}^{\infty }L_{x}^{\infty
}}^{1/2}C_{f_{w_{1}},T}^{2}+\eta \left( \left \Vert f_{0}\right \Vert
_{L_{\xi ,p+\beta +3j}^{\infty }L_{x}^{\infty }}\right) C_{f_{w_{1},}T}^{2}%
\right] \\
&\leq &C_{1,\gamma ,p,\beta ,j}^{\prime \prime }\left( \eta \left \Vert
f_{0}\right \Vert _{L_{\xi ,p+\beta +3j}^{\infty }L_{x}^{\infty }}+\eta
^{1/2}\left \Vert f_{0}\right \Vert _{L_{\xi ,p+\beta +3j}^{\infty
}L_{x}^{\infty }}^{1/2}C_{f_{w_{1}},T}^{2}\right) \text{,}
\end{eqnarray*}%
due to (\ref{eq: N. 1. a}), (\ref{eq: bootstrap 3}), (\ref{fw1-L2}) and (\ref%
{eq: g1}), whenever $\eta >0\ $is sufficiently small. That is,
\begin{equation}
\left \Vert f_{w_{1}}(t)\right \Vert _{L_{\xi }^{4}L_{x}^{2}}\leq
C_{1,\gamma ,p,\beta ,j}^{\prime \prime }\left( \eta \left \Vert f_{0}\right
\Vert _{L_{\xi ,p+\beta +3j}^{\infty }L_{x}^{\infty }}+\eta ^{1/2}\left
\Vert f_{0}\right \Vert _{L_{\xi ,p+\beta +3j}^{\infty }L_{x}^{\infty
}}^{1/2}C_{f_{w_{1}},T}^{2}\right) \text{.}  \label{fw1-L4}
\end{equation}
Through (\ref{fw1-rep}) again, we infer%
\begin{eqnarray*}
&&\left \Vert f_{w_{1}}(t)\right \Vert _{L_{\xi }^{\infty }L_{x}^{2}} \\
&\leq &\eta \left \Vert f_{w_{1}0}\right \Vert _{L_{\xi }^{\infty
}L_{x}^{2}}+C_{\gamma ,p}\int_{0}^{t}(1+t-s)^{\frac{7/4-\gamma }{\gamma }%
}\left \Vert f_{w_{1}}(s)\right \Vert _{L_{\xi }^{4}L_{x}^{2}}ds \\
&&+C_{\gamma ,p}\int_{0}^{t}(1+t-s)^{-\frac{\gamma }{\gamma }}\left \Vert
f_{w_{1}}\left( s\right) \right \Vert _{L_{\xi }^{\infty }L_{x}^{2}}\left
\Vert \left \langle \xi \right \rangle ^{p}f\left( s\right) \right \Vert
_{L_{\xi }^{\infty }L_{x}^{\infty }}ds \\
&\leq &\eta \left \Vert f_{w_{1}0}\right \Vert _{L_{\xi }^{\infty
}L_{x}^{2}}+C_{\gamma ,p}\int_{0}^{t}(1+t-s)^{\frac{7/4-\gamma }{\gamma }%
}\left \Vert f_{w_{1}}(s)\right \Vert _{L_{\xi }^{4}L_{x}^{2}}ds \\
&&+C_{1,\gamma ,p,\beta ,j}\int_{0}^{t}(1+t-s)^{-\frac{\gamma }{\gamma }%
}(1+s)^{-\frac{3}{2}}ds\cdot \left( \eta \left \Vert f_{0}\right \Vert
_{L_{\xi ,p+\beta +3j}^{\infty }L_{x}^{\infty }}\right) ds\cdot \sup_{0\leq
t\leq T}\left \Vert f_{w_{1}}\left( s\right) \right \Vert _{L_{\xi }^{\infty
}L_{x}^{2}} \\
&\leq &C_{1,\gamma ,p,\beta ,j}^{\prime }\left[ \eta \left \Vert f_{0}\right
\Vert _{L_{\xi ,p+\beta +3j}^{\infty }L_{x}^{\infty }}+\eta ^{1/2}\left
\Vert f_{0}\right \Vert _{L_{\xi ,p+\beta +3j}^{\infty }L_{x}^{\infty
}}^{1/2}C_{f_{w_{1}},T}^{2}\right] \\
&&+C_{1,\gamma ,p,\beta ,j}^{\prime }\left( \eta \left \Vert f_{0}\right
\Vert _{L_{\xi ,p+\beta +3j}^{\infty }L_{x}^{\infty }}\right) \sup_{0\leq
t\leq T}\left \Vert f_{w_{1}}\left( s\right) \right \Vert _{L_{\xi }^{\infty
}L_{x}^{2}}\text{,}
\end{eqnarray*}%
by using (\ref{eq: N. 1. a}), (\ref{eq: bootstrap 2}), and (\ref{fw1-L4}).
Since $\eta >0$ is chosen sufficiently small, we get%
\begin{equation*}
\left \Vert f_{w_{1}}(t)\right \Vert _{L_{\xi }^{\infty }L_{x}^{2}}\leq
2C_{1,\gamma ,p,\beta ,j}^{\prime }\left( \eta \left \Vert f_{0}\right \Vert
_{L_{\xi ,p+\beta +3j}^{\infty }L_{x}^{\infty }}+\eta ^{1/2}\left \Vert
f_{0}\right \Vert _{L_{\xi ,p+\beta +3j}^{\infty }L_{x}^{\infty
}}^{1/2}C_{f_{w_{1},}T}^{2}\right) \text{.}
\end{equation*}%
Using the bootstrap argument, we eventually get
\begin{equation*}
\left \Vert f_{w_{1}}(t)\right \Vert _{L_{\xi ,\beta }^{\infty
}L_{x}^{2}}\leq C_{1,\gamma ,p,\beta ,j}^{\prime \prime }\left( \eta \left
\Vert f_{0}\right \Vert _{L_{\xi ,p+\beta +3j}^{\infty }L_{x}^{\infty
}}+\eta ^{1/2}\left \Vert f_{0}\right \Vert _{L_{\xi ,p+\beta +3j}^{\infty
}L_{x}^{\infty }}^{1/2}C_{f_{w_{1},}T}^{2}\right) \text{.}
\end{equation*}%
Since $\eta >0$ is chosen sufficiently small, we get%
\begin{equation*}
\left \Vert f_{w_{1}}(t)\right \Vert _{L_{\xi ,\beta }^{\infty
}L_{x}^{2}}\leq C_{f_{w_{1},}T}^{2}\leq C_{1,\gamma ,p,\beta ,j}^{\prime
\prime \prime }\eta \left \Vert f_{0}\right \Vert _{L_{\xi ,p+\beta
+3j}^{\infty }L_{x}^{\infty }}\text{,}
\end{equation*}%
for $0\leq t\leq T$.

Gathering \textbf{Case I} and \textbf{Case II}, we obtain
\begin{equation}
\left \Vert f_{w_{1}}(t)\right \Vert _{L_{\xi ,\beta }^{\infty
}L_{x}^{2}}\leq C_{f_{w_{1},}T}^{2}\leq C_{1,\gamma ,p,\beta ,j}^{\prime
\prime \prime }\eta \left \Vert f_{0}\right \Vert _{L_{\xi ,p+\beta
+3j}^{\infty }L_{x}^{\infty }}  \label{fw1-sup}
\end{equation}%
for $-2<\gamma <0$ and $p\geq 1$, where the constant $C_{1,\gamma ,p,\beta
,j}^{\prime \prime \prime }>0$ independent\ of $T$. As for $\left \Vert
f_{w_{2}}(t)\right \Vert _{L_{\xi ,\beta }^{\infty }L_{x}^{2}}$, we choose a
weight function $w_{2}$ satisfying $\epsilon $, $\delta >0$ sufficiently
small, $M>0$ sufficiently large such that $\epsilon c_{p}<\hat{\varepsilon}$%
, $\delta M>0$ large and $\delta M\ll \epsilon ^{-1}$. After that, by the
energy estimate, we deduce that
\begin{equation*}
\left \Vert f_{w_{2}}(t)\right \Vert _{L^{2}}\lesssim \eta \left \Vert
f_{w_{3}0}\right \Vert _{L_{\xi ,p+\beta +3j}^{\infty }L_{x}^{\infty }}+\eta
^{1/2}\left \Vert f_{w_{3}0}\right \Vert _{L_{\xi ,p+\beta +3j}^{\infty
}L_{x}^{\infty }}^{1/2}C_{f_{w_{2}},T}^{2}\text{,}
\end{equation*}%
where $w_{3}=\exp \left( \hat{\varepsilon}\left \langle \xi \right \rangle
^{p}\right) $, $0<p\leq 2$, and $\hat{C}_{f_{w_{2}},T}^{2}=\sup_{0\leq t\leq
T}\left \Vert f_{w_{2}}\right \Vert _{L_{\xi ,\beta }^{\infty }L_{x}^{2}}$.
Following the above bootstrap argument, we as well have
\begin{equation}
\left \Vert f_{w_{2}}(t)\right \Vert _{L_{\xi ,\beta }^{\infty
}L_{x}^{2}}\leq \hat{C}_{f_{w_{2}},T}^{2}\leq C_{2,\gamma ,p,\beta
,j}^{\prime \prime \prime }\eta \left \Vert f_{0}\right \Vert _{L_{\xi
,p+\beta +3j}^{\infty }L_{x}^{\infty }}\text{,}  \label{fw2-sup}
\end{equation}%
for $-2<\gamma <0$ and $0<p\leq 2$, where the constant $C_{1,\gamma ,p,\beta
,j}^{\prime \prime \prime }>0$ is independent\ of $T$.

Now, according to Theorem \ref{theorem-linear}, if $-1<\gamma <0$,
\begin{align*}
\left( 1+t\right) ^{-A}\left \Vert f_{w_{1}}\right \Vert _{L_{\xi ,\beta
}^{\infty }L_{x}^{\infty }}& \lesssim (1+t)^{-A}\left( \eta \left \Vert
f_{w_{1}0}\right \Vert _{L_{\xi ,\beta }^{\infty }L_{x}^{\infty }}+(1+t)^{-%
\frac{3}{2}+A+\varsigma }C_{f_{w_{1},T}}^{\infty }C_{f,T}^{\infty }\right) \\
& \quad +(1+t)^{2-A}\left[ \left( 1+\delta M\right) \left( \eta \left \Vert
f_{w_{1}0}\right \Vert _{L_{\xi ,\beta }^{\infty
}L_{x}^{2}}+C_{f_{w_{1}},T}^{2}C_{f,T}^{\infty }\right) \right] \text{,}
\end{align*}%
$0\leq t\leq T$. Taking $A=2$ and choosing $\eta >0$ sufficiently small,
together with (\ref{fw1-sup}), we get%
\begin{equation*}
C_{f_{w_{1}},T}^{\infty }=\sup_{0\leq t\leq T}\left( 1+t\right) ^{-2}\left
\Vert f_{w_{1}}\right \Vert _{L_{\xi ,\beta }^{\infty }L_{x}^{\infty
}}\lesssim \eta \left \Vert f_{0}\right \Vert _{L_{\xi ,p+\beta +3j}^{\infty
}L_{x}^{\infty }}\left( 1+\eta \left \Vert f_{0}\right \Vert _{L_{\xi
,p+\beta +3j}^{\infty }L_{x}^{\infty }}\right) \text{,}
\end{equation*}%
since $C_{f,T}^{\infty }\lesssim \eta \left( \left \Vert f_{0}\right \Vert
_{L_{\xi ,p+\beta +3j}^{\infty }L_{x}^{\infty }}\right) $ (due to (\ref{eq:
N. 1. a})). It implies that
\begin{equation*}
\left \Vert f_{w_{1}}\right \Vert _{L_{\xi ,\beta }^{\infty }L_{x}^{\infty
}}\lesssim \left( 1+t\right) ^{2}\eta \left \Vert f_{0}\right \Vert _{L_{\xi
,p+\beta +3j}^{\infty }L_{x}^{\infty }}\left( 1+\eta \left \Vert f_{0}\right
\Vert _{L_{\xi ,p+\beta +3j}^{\infty }L_{x}^{\infty }}\right) \text{,}
\end{equation*}%
for $0\leq t\leq T$ and then for $0\leq t<\infty $ since $T$ can be
arbitrarily large. Note that for $\left \langle x\right \rangle >2Mt$,\quad
\begin{equation*}
w_{1}\left( t,x,\xi \right) \gtrsim \left[ \delta \left( \left \langle
x\right \rangle -Mt\right) \right] ^{\frac{p}{1-\gamma }}\text{ and }\left
\langle x\right \rangle -Mt>\frac{\left \langle x\right \rangle }{3}+\frac{Mt%
}{3}\text{,}
\end{equation*}%
so that
\begin{eqnarray*}
&&\left \vert f(t,x)\right \vert _{L_{\xi ,\beta }^{\infty }} \\
&\lesssim &\eta (1+t)^{2}(\left \langle x\right \rangle +Mt)^{\frac{-p}{%
1-\gamma }}\left \Vert \left \langle \xi \right \rangle ^{p+\beta
+3j}f_{0}\right \Vert _{L_{\xi }^{\infty }L_{x}^{\infty }}\left( 1+\eta
\left \Vert \left \langle \xi \right \rangle ^{p+\beta +3j}f_{0}\right \Vert
_{L_{\xi }^{\infty }L_{x}^{\infty }}\right) \text{,}
\end{eqnarray*}%
for $-1<\gamma <0$ and $p\geq 1$. Likewise, we can obtain the estimate of $%
\left \vert f(t,x)\right \vert _{L_{\xi ,\beta }^{\infty }}$ in other cases
by taking $A=2+\varsigma $ whenever $\gamma =-1$ and $A=7+\frac{5}{\gamma }$
whenever $-2<\gamma <-1$ respectively in Theorem \ref{theorem-linear}. This
completes the proof of part (i). Imposing a certain exponential weight on
the initial data $f_{0}$, we also note that for $\left \langle
x\right
\rangle >2Mt$,
\begin{equation*}
\rho \left( t,x,\xi \right) \gtrsim \left[ \delta \left( \left \langle
x\right \rangle -Mt\right) \right] ^{\frac{p}{p+1-\gamma }}\text{ and }\left
\langle x\right \rangle -Mt>\frac{\left \langle x\right \rangle }{3}+\frac{Mt%
}{3}\text{,}
\end{equation*}%
where $0<p\leq 2$. Hence, part (ii) follows by taking $B=2$ whenever $%
-1<\gamma <0$, $B=2+\varsigma $ whenever $\gamma =-1$, and $B=7+\frac{5}{%
\gamma }$ whenever $-2<\gamma <-1$ respectively in Theorem \ref{prop:
nonlinear}, besides choosing $\eta >0$ sufficiently small in each case. The
proof of the theorem is completed.$%
\hfill%
\square $

\section{Proof of Theorem \protect \ref{prop: nonlinear}}

\label{proof: nonlinear}

This section is devoted to the large time behavior of the solution $f$ to (%
\ref{bot.1.d}) in certain weighted normed spaces. Our strategy is to study
the homogeneous/inhomogeneous linearized Boltzmann equation in the first two
subsections, and then to demonstrate the large time behavior via an
iteration scheme.

\subsection{Linear Boltzmann Equation}

Let $\mathbb{G}^{t}$ be the solution operator of the linearized Boltzmann
equation
\begin{equation}
\left \{
\begin{array}{ll}
\displaystyle \pa_{t}g+\xi \cdot \nabla _{x}g=Lg\text{,} & (t,x,\xi )\in {%
\mathbb{R}}^{+}\times {\mathbb{R}}^{3}\times {\mathbb{R}}^{3}\text{,} \\%
[4mm]
\displaystyle g(0,x,\xi )=g_{0}(x,\xi )\text{, } & (x,\xi )\in {\mathbb{R}}%
^{3}\times {\mathbb{R}}^{3}\text{,}%
\end{array}%
\right. \quad  \label{bot.1.e}
\end{equation}%
and let $\mathbb{S}^{t}$ be the solution operator of the damped transport
equation
\begin{equation}
\left \{
\begin{array}{ll}
\displaystyle \pa_{t}h+\xi \cdot \nabla _{x}h+\nu (\xi )h=0\text{,} &
(t,x,\xi )\in {\mathbb{R}}^{+}\times {\mathbb{R}}^{3}\times {\mathbb{R}}^{3}%
\text{,} \\[4mm]
\displaystyle h(0,x,\xi )=h_{0}(x,\xi )\text{,} & (x,\xi )\in {\mathbb{R}}%
^{3}\times {\mathbb{R}}^{3}\text{.}%
\end{array}%
\right.  \label{bot.1.g}
\end{equation}%
We will provide the estimate for the large time behavior of the solution $g$
to (\ref{bot.1.e}).

\begin{proposition}
\label{thm: linear}Let $-2<\gamma <0$, $0<p_{1}\leq 2$, $p_{2}>3/2$, $\hat{%
\eps}\geq 0$ sufficiently small, and let $j>0$ be any sufficiently large
number. Assume that $w_{3}g_{0}\in L_{\xi ,p_{2}+j}^{\infty }L_{x}^{\infty
}\cap L_{\xi ,p_{2}+j}^{\infty }L_{x}^{1}$. Then there are positive
constants $C_{i,\gamma ,\hat{\varepsilon},p_{1},p_{2},j}$ and $\bar{C}%
_{i,\gamma ,\hat{\varepsilon},p_{1},p_{2},j}$, $i=1$, $2$, such that the
solution $g$ to (\ref{bot.1.e}) satisfies%
\begin{eqnarray}
\left \Vert w_{3}g(t)\right \Vert _{L_{\xi ,p_{2}}^{\infty }L_{x}^{2}} &\leq
&C_{1,\gamma ,\hat{\varepsilon},p_{1},p_{2},j}(1+t)^{-\frac{3}{4}}\left(
\left \Vert w_{3}g_{0}\right \Vert _{L_{\xi ,p_{2}+j}^{\infty
}L_{x}^{1}}+\left \Vert w_{3}g_{0}\right \Vert _{L_{\xi ,p_{2}+j}^{\infty
}L_{x}^{\infty }}\right) \text{,}%
\vspace {3mm}
\label{eq: Greenweightepsilon1} \\
\left \Vert w_{3}g(t)\right \Vert _{L_{\xi ,p_{2}}^{\infty }L_{x}^{\infty }}
&\leq &C_{2,\gamma ,\hat{\varepsilon},p_{1},p_{2},j}(1+t)^{-\frac{3}{2}%
}\left( \left \Vert w_{3}g_{0}\right \Vert _{L_{\xi ,p_{2}+j}^{\infty
}L_{x}^{1}}+\left \Vert w_{3}g_{0}\right \Vert _{L_{\xi ,p_{2}+j}^{\infty
}L_{x}^{\infty }}\right) \text{.}  \label{eq: Greenweightepsilon2}
\end{eqnarray}%
Moreover,%
\begin{eqnarray}
\left \Vert w_{3}g(t)\right \Vert _{L_{\xi ,p_{2}+j}^{\infty }L_{x}^{2}}
&\leq &\bar{C}_{1,\gamma ,\hat{\varepsilon},p_{1},p_{2},j}\left( \left \Vert
w_{3}g_{0}\right \Vert _{L_{\xi ,p_{2}+j}^{\infty }L_{x}^{1}}+\left \Vert
w_{3}g_{0}\right \Vert _{L_{\xi ,p_{2}+j}^{\infty }L_{x}^{\infty }}\right)
\text{,}%
\vspace {3mm}
\label{eq: Greenweightalpha1} \\
\left \Vert w_{3}g(t)\right \Vert _{L_{\xi ,p_{2}+j}^{\infty }L_{x}^{\infty
}} &\leq &\bar{C}_{2,\gamma ,\hat{\varepsilon},p_{1},p_{2},j}\left( \left
\Vert w_{3}g_{0}\right \Vert _{L_{\xi ,p_{2}+j}^{\infty }L_{x}^{1}}+\left
\Vert w_{3}g_{0}\right \Vert _{L_{\xi ,p_{2}+j}^{\infty }L_{x}^{\infty
}}\right) \text{.}  \label{eq: Greenweightalpha2}
\end{eqnarray}
\end{proposition}

\proof%
By assumption, $g_{0}\in L_{\xi ,p_{2}+j}^{\infty }L_{x}^{\infty }\cap
L_{\xi ,p_{2}+j}^{\infty }L_{x}^{1}$. Then, following a similar argument as
in \cite[Propositions 7 and 15]{[LinWangWu]} (or see \cite{[Strain],
[Strain-Guo]}), we see that
\begin{eqnarray}
\left \Vert g(t)\right \Vert _{L_{\xi }^{2}L_{x}^{2}} &\lesssim &(1+t)^{-%
\frac{3}{4}}\left( \left \Vert \left \langle \xi \right \rangle
^{j}g_{0}\right \Vert _{L_{\xi }^{2}L_{x}^{1}}+\left \Vert \left \langle \xi
\right \rangle ^{j}g_{0}\right \Vert _{L_{\xi }^{2}L_{x}^{2}}\right)
\label{eq: GreenLx^2} \\
&\lesssim &(1+t)^{-\frac{3}{4}}\left( \left \Vert \left \langle \xi \right
\rangle ^{j}g_{0}\right \Vert _{L_{\xi ,p_{2}}^{\infty }L_{x}^{1}}+\left
\Vert \left \langle \xi \right \rangle ^{j}g_{0}\right \Vert _{L_{\xi
,p_{2}}^{\infty }L_{x}^{\infty }}\right) \text{,}  \notag
\end{eqnarray}%
and%
\begin{eqnarray}
\left \Vert g(t)\right \Vert _{L_{\xi }^{2}L_{x}^{\infty }} &\lesssim
&(1+t)^{-\frac{3}{2}}\left( \left \Vert \left \langle \xi \right \rangle
^{j}g_{0}\right \Vert _{L_{\xi }^{2}L_{x}^{2}}+\left \Vert \left \langle \xi
\right \rangle ^{j}g_{0}\right \Vert _{L_{\xi }^{2}L_{x}^{1}}+\left \Vert
\left \langle \xi \right \rangle ^{j}g_{0}\right \Vert _{L_{\xi
,p_{2}}^{\infty }L_{x}^{\infty }}\right) \text{ }  \label{eq: GreenLx^inf} \\
&\lesssim &(1+t)^{-\frac{3}{2}}\left( \left \Vert \left \langle \xi \right
\rangle ^{j}g_{0}\right \Vert _{L_{\xi ,p_{2}}^{\infty }L_{x}^{1}}+\left
\Vert \left \langle \xi \right \rangle ^{j}g_{0}\right \Vert _{L_{\xi
,p_{2}}^{\infty }L_{x}^{\infty }}\right) \text{.}  \notag
\end{eqnarray}

Now we prove (\ref{eq: Greenweightepsilon1}) and the others are similar. In
terms of the damped transport operator $\mathbb{S}^{t}$, $g$ can be written
as
\begin{equation}
g\left( t\right) =\mathbb{S}^{t}g_{0}+\int_{0}^{t}\mathbb{S}^{t-s}Kg(s)ds%
\text{.}  \label{f-rep}
\end{equation}%
Let $T>0$. For any $0\leq t\leq T$,
\begin{equation*}
w_{3}\left \vert g(t)\right \vert _{L_{x}^{2}}\leq w_{3}\left \vert \mathbb{S%
}^{t}g_{0}\right \vert _{L_{x}^{2}}+w_{3}\int_{0}^{t}\left \vert \mathbb{S}%
^{t-s}Kg(s)\right \vert _{L_{x}^{2}}ds=I+II\text{.}
\end{equation*}%
It is easy to see that
\begin{eqnarray}
I &\leq &\sup_{\xi }\left( w_{3}\left \vert \mathbb{S}^{t}g_{0}\right \vert
_{L_{x}^{2}}\right) \leq \left( \sup_{\xi }e^{-\nu (\xi )t}\left \langle \xi
\right \rangle ^{-j}\right) \left \Vert w_{3}g_{0}\right \Vert _{L_{\xi
,j}^{\infty }L_{x}^{2}}  \label{eq: L1} \\
&\leq &C_{\gamma ,j}(1+t)^{\frac{j}{\gamma }}\left \Vert w_{3}g_{0}\right
\Vert _{L_{\xi ,p_{2}+j}^{\infty }L_{x}^{2}}\leq C_{\gamma ,j}(1+t)^{-\frac{3%
}{4}}\left \Vert w_{3}g_{0}\right \Vert _{L_{\xi ,p_{2}+j}^{\infty
}L_{x}^{2}}\, \text{,}  \notag
\end{eqnarray}%
since $j$ is sufficiently large. For $II$, it follows from (\ref{gain-weight}%
), (\ref{gain-weight1}) and (\ref{Kw6}) that
\begin{eqnarray*}
&&w_{3}\left \vert \mathbb{S}^{t-s}Kg(s)\right \vert _{L_{x}^{2}}=e^{\hat{%
\varepsilon}\left \langle \xi \right \rangle ^{p_{1}}}e^{-(t-s)\nu (\xi
)}\left \vert Kg(s)\right \vert _{L_{x}^{2}} \\
&\leq &e^{-(t-s)\nu (\xi )}\left \langle \xi \right \rangle ^{\gamma -1}
\left[ \sup_{\left \vert \xi \right \vert \leq \tau }\left( e^{\hat{%
\varepsilon}\left \langle \xi \right \rangle ^{p_{1}}}\left \langle \xi
\right \rangle ^{-\gamma +1}\left \vert Kg(s)\right \vert
_{L_{x}^{2}}\right) +\sup_{\left \vert \xi \right \vert >\tau }\left( \left
\langle \xi \right \rangle ^{-1}e^{\hat{\varepsilon}\left \langle \xi \right
\rangle ^{p_{1}}}\left \langle \xi \right \rangle ^{2-\gamma }\left \vert
Kg(s)\right \vert _{L_{x}^{2}}\right) \right] \\
&\leq &C_{\gamma }(1+t-s)^{\frac{1-\gamma }{\gamma }}\left( e^{\hat{%
\varepsilon}\left \langle \tau \right \rangle ^{p_{1}}}\left \Vert
Kg(s)\right \Vert _{L_{\xi ,1-\gamma }^{\infty }L_{x}^{2}}+(1+\tau
)^{-1}\left \Vert w_{3}Kg(s)\right \Vert _{L_{\xi ,2-\gamma }^{\infty
}L_{x}^{2}}\right)
\vspace {3mm}
\\
&\leq &C_{\gamma ,\hat{\varepsilon},p_{1}}(1+t-s)^{\frac{1-\gamma }{\gamma }%
}\cdot \left \{
\begin{array}{ll}
\left( e^{\hat{\varepsilon}\tau ^{p_{1}}}\left \Vert g(s)\right \Vert
_{L_{\xi }^{2}L_{x}^{2}}+(1+\tau )^{-1}\left \Vert w_{3}g(s)\right \Vert
_{L_{\xi }^{\infty }L_{x}^{2}}\right) \text{,} & \text{if }\frac{-3}{2}%
<\gamma <0\text{,}%
\vspace {3mm}
\\
\left( e^{\hat{\varepsilon}\tau ^{p_{1}}}\left \Vert g(s)\right \Vert
_{L_{\xi }^{4}L_{x}^{2}}+(1+\tau )^{-1}\left \Vert w_{3}g(s)\right \Vert
_{L_{\xi }^{\infty }L_{x}^{2}}\right) \text{,} & \text{if }-2<\gamma \leq
\frac{-3}{2}\text{,}%
\end{array}%
\right.
\end{eqnarray*}%
for any $\tau >0$. Whenever $-2<\gamma \leq \frac{-3}{2}$, in view of (\ref%
{gain-weight2}), (\ref{eq: GreenLx^2}) and (\ref{f-rep}), we have%
\begin{eqnarray}
\left \Vert g(s)\right \Vert _{L_{\xi }^{4}L_{x}^{2}} &\leq &\left(
\sup_{\xi }e^{-s\nu (\xi )}\left \langle \xi \right \rangle ^{-j}\right)
\left \Vert g_{0}\right \Vert _{L_{\xi ,j}^{4}L_{x}^{2}}  \notag \\
&&+\int_{0}^{s}\sup_{\xi }\left( e^{-\nu \left( \xi \right) \left(
s-s^{\prime }\right) }\left \langle \xi \right \rangle ^{-\left( 1-\gamma
\right) }\right) \left \Vert Kg(s^{\prime })\right \Vert _{L_{\xi ,1-\gamma
}^{4}L_{x}^{2}}ds^{\prime }  \label{eq: L4} \\
&\lesssim &(1+s)^{\frac{j}{\gamma }}\left \Vert g_{0}\right \Vert _{L_{\xi
,j}^{4}L_{x}^{2}}+\int_{0}^{s}\left( 1+s-s^{\prime }\right) ^{\frac{1-\gamma
}{\gamma }}\left \Vert g(s^{\prime })\right \Vert _{L_{\xi
}^{2}L_{x}^{2}}ds^{\prime }  \notag \\
&\lesssim &(1+s)^{\frac{j}{\gamma }}\left \Vert g_{0}\right \Vert _{L_{\xi
,j}^{4}L_{x}^{2}}\notag\\
&&+\int_{0}^{s}(1+s-s^{\prime })^{\frac{1-\gamma }{\gamma }}(1+s^{\prime
})^{-\frac{3}{4}}\left( \left \Vert g_{0}\right \Vert _{L_{\xi
,p_{2}+j}^{\infty }L_{x}^{1}}+\left \Vert g_{0}\right \Vert _{L_{\xi
,p_{2}+j}^{\infty }L_{x}^{\infty }}\right) ds^{\prime }  \notag \\
&\lesssim &(1+s)^{-\frac{3}{4}}\left( \left \Vert g_{0}\right \Vert _{L_{\xi
,p_{2}+j}^{\infty }L_{x}^{1}}+\left \Vert g_{0}\right \Vert _{L_{\xi
,p_{2}+j}^{\infty }L_{x}^{\infty }}\right) \text{,}  \notag
\end{eqnarray}%
the last inequality being valid since $j>0$ is sufficiently large. Using (%
\ref{eq: GreenLx^2}) and (\ref{eq: L4}), we deduce
\begin{eqnarray}
&&II  \notag \\
&\leq &C_{\gamma ,p_{1},p_{2},j}e^{\hat{\varepsilon}\left \langle \tau
\right \rangle ^{p_{1}}}\left( \left \Vert g_{0}\right \Vert _{L_{\xi
,p_{2}+j}^{\infty }L_{x}^{1}}+\left \Vert g_{0}\right \Vert _{L_{\xi
,p_{2}+j}^{\infty }L_{x}^{\infty }}\right) \int_{0}^{t}(1+t-s)^{\frac{%
1-\gamma }{\gamma }}(1+s)^{-\frac{3}{4}}ds  \notag \\
&&+C_{\gamma ,p_{1},p_{2},j}\left( 1+\tau \right) ^{-1}\sup_{0\leq s\leq T}
\left[ \left( 1+s\right) ^{\frac{3}{4}}\left \Vert w_{3}g(s)\right \Vert
_{L_{\xi }^{\infty }L_{x}^{2}}\right] \cdot \int_{0}^{t}(1+t-s)^{\frac{%
1-\gamma }{\gamma }}(1+s)^{-\frac{3}{4}}ds  \label{eq: L2} \\
&\leq &C_{\gamma ,p_{1},p_{2},j}^{\prime }e^{\hat{\varepsilon}\left \langle
\tau \right \rangle ^{p_{1}}}(1+t)^{-\frac{3}{4}}\left( \left \Vert
w_{3}g_{0}\right \Vert _{L_{\xi ,p_{2}+j}^{\infty }L_{x}^{1}}+\left \Vert
w_{3}g_{0}\right \Vert _{L_{\xi ,p_{2}+j}^{\infty }L_{x}^{\infty }}\right)
\notag \\
&&+C_{\gamma ,p_{1},p_{2},j}^{\prime }\left( 1+\tau \right) ^{-1}(1+t)^{-%
\frac{3}{4}}\sup_{0\leq s\leq T}\left[ \left( 1+s\right) ^{\frac{3}{4}}\left
\Vert w_{3}g(s)\right \Vert _{L_{\xi }^{\infty }L_{x}^{2}}\right] \text{.}
\notag
\end{eqnarray}%
After selecting $\tau >0$ sufficiently large with $C_{\gamma
,p_{1},p_{2},j}^{\prime }\left( 1+\tau \right) ^{-1}<1/2$, we obtain
\begin{equation*}
\sup_{0\leq t\leq T}\left[ (1+t)^{\frac{3}{4}}\left \Vert w_{3}g(t)\right
\Vert _{L_{\xi }^{\infty }L_{x}^{2}}\right] \leq C_{1,\gamma ,\hat{%
\varepsilon},p_{1},p_{2},j}\left( \left \Vert w_{3}g_{0}\right \Vert
_{L_{\xi ,p_{2}+j}^{\infty }L_{x}^{1}}+\left \Vert w_{3}g_{0}\right \Vert
_{L_{\xi ,p_{2}+j}^{\infty }L_{x}^{\infty }}\right) \text{,}
\end{equation*}%
due to (\ref{eq: L1}) and (\ref{eq: L2}). It implies that
\begin{equation*}
\left \Vert w_{3}g(t)\right \Vert _{L_{\xi }^{\infty }L_{x}^{2}}\leq
C_{1,\gamma ,\hat{\varepsilon},p_{1},p_{2},j}(1+t)^{-\frac{3}{4}}\left(
\left \Vert w_{3}g_{0}\right \Vert _{L_{\xi ,p_{2}+j}^{\infty
}L_{x}^{1}}+\left \Vert w_{3}g_{0}\right \Vert _{L_{\xi ,p_{2}+j}^{\infty
}L_{x}^{\infty }}\right)
\end{equation*}%
for $0\leq t<\infty $, since $T>0$ is arbitrary.

Finally, through the bootstrap argument, we get%
\begin{equation*}
\left \Vert w_{3}g(t)\right \Vert _{L_{\xi ,p_{2}}^{\infty }L_{x}^{2}}\leq
C_{1,\gamma ,\hat{\varepsilon},p_{1},p_{2},j}(1+t)^{-\frac{3}{4}}\left(
\left \Vert w_{3}g_{0}\right \Vert _{L_{\xi ,p_{2}+j}^{\infty
}L_{x}^{1}}+\left \Vert w_{3}g_{0}\right \Vert _{L_{\xi ,p_{2}+j}^{\infty
}L_{x}^{\infty }}\right) \text{,}
\end{equation*}%
as desired. The proof of this proposition is completed.$%
\hfill%
\square $

\subsection{The Inhomogeneous Boltzmann Equation}

In this section, we further consider the inhomogeneous Boltzmann equation
\begin{equation}
\left \{
\begin{array}{l}
\displaystyle \partial _{t}g+\xi \cdot \nabla _{x}g=Lg+\Gamma (h_{1},h_{2})%
\text{,} \\[4mm]
\displaystyle g(0,x,\xi )=g_{0}(x,\xi )\text{.}%
\end{array}%
\right.  \label{bot.1.r}
\end{equation}

Now, let $0<p_{1}\leq 2$, $p_{2}>3/2$, $\hat{\varepsilon}\geq 0$
sufficiently small, and $j>0$ sufficiently large. We assume that $g_{0}$
satisfies
\begin{equation}
\left \Vert g_{0}\right \Vert _{L_{\xi }^{\infty }\left( \left \langle \xi
\right \rangle ^{p_{2}+2j}e^{\hat{\varepsilon}\left \langle \xi \right
\rangle ^{p_{1}}}\right) L_{x}^{1}}+\left \Vert g_{0}\right \Vert _{L_{\xi
}^{\infty }\left( \left \langle \xi \right \rangle ^{p_{2}+2j}e^{\hat{%
\varepsilon}\left \langle \xi \right \rangle ^{p_{1}}}\right) L_{x}^{\infty
}}\leq b_{0}\text{,}  \label{cond-f0}
\end{equation}%
and $h_{i}$ $\left( i=1\text{, }2\right) $ satisfies
\begin{align}
\sup_{t}& \left \{ (1+t)^{\frac{3}{4}}\left \Vert h_{i}(t)\right \Vert
_{L_{\xi }^{\infty }\left( \left \langle \xi \right \rangle ^{p_{2}}e^{\hat{%
\varepsilon}\left \langle \xi \right \rangle ^{p_{1}}}\right) L_{x}^{2}},%
\text{ }(1+t)^{\frac{3}{4}}\left \Vert h_{i}(t)\right \Vert _{L_{\xi
}^{\infty }\left( \left \langle \xi \right \rangle ^{p_{2}}e^{\hat{%
\varepsilon}\left \langle \xi \right \rangle ^{p_{1}}}\right) L_{x}^{\infty
}},\right.  \notag \\
& \left. \left \Vert h_{i}(t)\right \Vert _{L_{\xi }^{\infty }\left( \left
\langle \xi \right \rangle ^{p_{2}+2j}e^{\hat{\varepsilon}\left \langle \xi
\right \rangle ^{p_{1}}}\right) L_{x}^{2}},\text{ }\left \Vert
h_{i}(t)\right \Vert _{L_{\xi }^{\infty }\left( \left \langle \xi \right
\rangle ^{p_{2}+2j}e^{\hat{\varepsilon}\left \langle \xi \right \rangle
^{p_{1}}}\right) L_{x}^{\infty }}\right \} \leq b_{i}\text{, }
\label{Cond-hi}
\end{align}%
for some $b_{0}$, $b_{1}$, $b_{2}>0$. We will study the large time behavior
of the solution $g$ to (\ref{bot.1.r}) in some suitable norms (see
Proposition \ref{Prop: inhomo}).

Before proving Proposition \ref{Prop: inhomo}, we need some preliminary
results (Lemmas \ref{lemma: estimate of nonlinear term} and \ref{lemma:
eps_n}) regarding the nonlinear term $\Gamma $ under the assumption (\ref%
{Cond-hi}).

\begin{lemma}
\label{lemma: estimate of nonlinear term}Assume that $h_{1}$ and $h_{2}$
satisfy (\ref{Cond-hi}). Then%
\begin{eqnarray*}
\left \Vert \Gamma _{loss}(h_{1},h_{2})(t)\right \Vert _{L_{\xi }^{\infty
}\left( \left \langle \xi \right \rangle ^{p_{2}+2j-\gamma}e^{\hat{\varepsilon}%
\left \langle \xi \right \rangle ^{p_{1}}}\right) L_{x}^{1}} &\leq
&C_{1}(1+t)^{-\frac{3}{4}}b_{1}b_{2}\text{,}%
\vspace {3mm}
\\
\left \Vert \Gamma _{loss}(h_{1},h_{2})(t)\right \Vert _{L_{\xi }^{\infty
}\left( \left \langle \xi \right \rangle ^{p_{2}+2j-\gamma}e^{\hat{\varepsilon}%
\left \langle \xi \right \rangle ^{p_{1}}}\right) L_{x}^{2}} &\leq
&C_{1}(1+t)^{-\frac{3}{4}}b_{1}b_{2}\text{,}%
\vspace {3mm}
\\
\left \Vert \Gamma _{loss}(h_{1},h_{2})(t)\right \Vert _{L_{\xi }^{\infty
}\left( \left \langle \xi \right \rangle ^{p_{2}+2j-\gamma}e^{\hat{\varepsilon}%
\left \langle \xi \right \rangle ^{p_{1}}}\right) L_{x}^{\infty }} &\leq
&C_{1}(1+t)^{-\frac{3}{4}}b_{1}b_{2}\text{,}%
\vspace {3mm}
\\
\left \Vert \Gamma _{gain}(h_{1},h_{2})(t)\right \Vert _{L_{\xi }^{\infty
}\left( \left \langle \xi \right \rangle ^{p_{2}+2j-\gamma +1}e^{\hat{%
\varepsilon}\left \langle \xi \right \rangle ^{p_{1}}}\right) L_{X}} &\leq
&C_{2}b_{1}b_{2}\text{,}
\end{eqnarray*}%
where $L_{X}=L_{x}^{1}$, $L_{x}^{2}$ and $L_{x}^{\infty }$.
\end{lemma}

\proof%
By Lemma \ref{Lemma-Gamma-sup},%
\begin{eqnarray*}
&&\left \vert \Gamma _{loss}(h_{1},h_{2})\left( t\right) \right \vert
_{L_{\xi }^{\infty }\left( \left \langle \xi \right \rangle ^{p_{2}+2j-\gamma}e^{%
\hat{\varepsilon}\left \langle \xi \right \rangle ^{p_{1}}}\right) } \\
&\lesssim &|h_{1}|_{L_{\xi }^{\infty }}|h_{2}|_{L_{\xi }^{\infty }\left(
\left \langle \xi \right \rangle ^{p_{2}+2j}e^{\hat{\varepsilon}%
\left \langle \xi \right \rangle ^{p_{1}}}\right) }+|h_{2}|_{L_{\xi
}^{\infty }}|h_{1}|_{L_{\xi }^{\infty }\left( \left \langle \xi \right
\rangle ^{p_{2}+2j}e^{\hat{\varepsilon}\left \langle \xi \right
\rangle ^{p_{1}}}\right) }\text{,}
\end{eqnarray*}%
and%
\begin{eqnarray*}
&&\left \vert \Gamma _{gain}(h_{1},h_{2})(t)\right \vert _{L_{\xi }^{\infty
}\left( \left \langle \xi \right \rangle ^{p_{2}+2j-\gamma +1}e^{\hat{%
\varepsilon}\left \langle \xi \right \rangle ^{p_{1}}}\right) } \\
&\lesssim &|h_{1}|_{L_{\xi }^{\infty }\left( \left \langle \xi \right
\rangle ^{p_{2}+2j}e^{\hat{\varepsilon}\left \langle \xi \right \rangle
^{p_{1}}}\right) }|h_{2}|_{L_{\xi }^{\infty }\left( \left \langle \xi \right
\rangle ^{p_{2}+2j}e^{\hat{\varepsilon}\left \langle \xi \right \rangle
^{p_{1}}}\right) }\text{.}
\end{eqnarray*}%
Therefore, according to the assumption (\ref{Cond-hi}) and that $p_{2}>0$,
we obtain the desired estimates.$%
\hfill%
\square $

To prove Lemma \ref{lemma: eps_n}, we need an interpolation inequality:

\begin{lemma}
\label{lemma: eps2^n}
\begin{eqnarray*}
&&\left \Vert H(t,x,\xi )\right \Vert _{L_{\xi }^{\infty }\left( \left
\langle \xi \right \rangle ^{p_{2}+j}e^{\hat{\varepsilon}\left \langle \xi
\right \rangle ^{p_{1}}}\right) L_{X}} \\
&\leq &2\left \Vert H(t,x,\xi )\right \Vert _{L_{\xi }^{\infty }\left( \left
\langle \xi \right \rangle ^{p_{2}+2j}e^{\hat{\varepsilon}\left \langle \xi
\right \rangle ^{p_{1}}}\right) L_{X}}^{\frac{1}{2}}\left \Vert H(t,x,\xi
)\right \Vert _{L_{\xi }^{\infty }\left( \left \langle \xi \right \rangle
^{p_{2}}e^{\hat{\varepsilon}\left \langle \xi \right \rangle
^{p_{1}}}\right) L_{X}}^{\frac{1}{2}}\text{,}
\end{eqnarray*}%
where $L_{X}=L_{x}^{1}$ and $L_{x}^{\infty }$.
\end{lemma}

\begin{proof}
For any $\xi _{0}>0$,
\begin{eqnarray*}
&&\left \Vert H(t,x,\xi )\right \Vert _{L_{\xi }^{\infty }\left( \left
\langle \xi \right \rangle ^{p_{2}+j}e^{\hat{\varepsilon}\left \langle \xi
\right \rangle ^{p_{1}}}\right) L_{X}} \\
&\leq &\left \langle \xi _{0}\right \rangle ^{j}\sup_{|\xi |\leq \xi
_{0}}\left \vert \left \langle \xi \right \rangle ^{p_{2}}e^{\hat{\varepsilon%
}\left \langle \xi \right \rangle ^{p_{1}}}H(t,x,\xi )\right \vert
_{L_{X}}+\left \langle \xi _{0}\right \rangle ^{-j}\sup_{|\xi |>\xi
_{0}}\left \vert \left \langle \xi \right \rangle ^{p_{2}+2j}e^{\hat{%
\varepsilon}\left \langle \xi \right \rangle ^{p_{1}}}H(t,x,\xi )\right
\vert _{L_{X}} \\
&\leq &\left \langle \xi _{0}\right \rangle ^{j}\left \Vert H(t,x,\xi
)\right \Vert _{L_{\xi }^{\infty }\left( \left \langle \xi \right \rangle
^{p_{2}}e^{\hat{\varepsilon}\left \langle \xi \right \rangle
^{p_{1}}}\right) L_{X}}+\left \langle \xi _{0}\right \rangle ^{-j}\left
\Vert H(t,x,\xi )\right \Vert _{L_{\xi }^{\infty }\left( \left \langle \xi
\right \rangle ^{p_{2}+2j}e^{\hat{\varepsilon}\left \langle \xi \right
\rangle ^{p_{1}}}\right) L_{X}}\text{.}
\end{eqnarray*}%
We can get the desired result by choosing $\xi _{0}>0$ such that
\begin{equation*}
\left \langle \xi _{0}\right \rangle ^{j}=\left \Vert H(t,x,\xi )\right
\Vert _{L_{\xi }^{\infty }\left( \left \langle \xi \right \rangle
^{p_{2}+2j}e^{\hat{\varepsilon}\left \langle \xi \right \rangle
^{p_{1}}}\right) L_{X}}^{\frac{1}{2}}\left \Vert H(t,x,\xi )\right \Vert
_{L_{\xi }^{\infty }\left( \left \langle \xi \right \rangle ^{p_{2}}e^{\hat{%
\varepsilon}\left \langle \xi \right \rangle ^{p_{1}}}\right) L_{X}}^{-\frac{%
1}{2}}\text{.}
\end{equation*}
\end{proof}

\begin{lemma}
\label{lemma: eps_n}Assume that $h_{1}$ and $h_{2}$ satisfy (\ref{Cond-hi}).
Then there exists a positive constant $C_{\gamma ,\hat{\varepsilon}%
,p_{1},p_{2},j}$ such that
\begin{eqnarray*}
\left \Vert \Gamma (h_{1},h_{2})(t)\right \Vert _{L_{\xi }^{\infty }\left(
\left \langle \xi \right \rangle ^{p_{2}+j}e^{\hat{\varepsilon}\left \langle
\xi \right \rangle ^{p_{1}}}\right) L_{x}^{1}} &\leq &C_{\gamma ,\hat{%
\varepsilon},p_{1},p_{2},j}(1+t)^{-\frac{3}{4}}b_{1}b_{2}\text{,}%
\vspace {3mm}
\\
\left \Vert \Gamma (h_{1},h_{2})(t)\right \Vert _{L_{\xi }^{\infty }\left(
\left \langle \xi \right \rangle ^{p_{2}+j}e^{\hat{\varepsilon}\left \langle
\xi \right \rangle ^{p_{1}}}\right) L_{x}^{\infty }} &\leq &C_{\gamma ,\hat{%
\varepsilon},p_{1},p_{2},j}(1+t)^{-\frac{3}{4}}b_{1}b_{2}\text{.}
\end{eqnarray*}
\end{lemma}

\proof%
It readily follows from Lemma \ref{Lemma-Gamma-sup} that
\begin{eqnarray*}
\left \Vert \Gamma _{loss}(h_{1},h_{2})(t)\right \Vert _{L_{\xi }^{\infty
}\left( \left \langle \xi \right \rangle ^{p_{2}}e^{\hat{\varepsilon}\left
\langle \xi \right \rangle ^{p_{1}}}\right) L_{x}^{1}} &\leq &C_{1}(1+t)^{-%
\frac{3}{2}}b_{1}b_{2}\text{,}%
\vspace {3mm}
\\
\left \Vert \Gamma _{loss}(h_{1},h_{2})(t)\right \Vert _{L_{\xi }^{\infty
}\left( \left \langle \xi \right \rangle ^{p_{2}}e^{\hat{\varepsilon}\left
\langle \xi \right \rangle ^{p_{1}}}\right) L_{x}^{\infty }} &\leq
&C_{1}(1+t)^{-\frac{3}{2}}b_{1}b_{2}\text{,}%
\vspace {3mm}
\\
\left \Vert \Gamma _{gain}(h_{1},h_{2})(t)\right \Vert _{L_{\xi }^{\infty
}\left( \left \langle \xi \right \rangle ^{p_{2}-\gamma +1}e^{\hat{%
\varepsilon}\left \langle \xi \right \rangle ^{p_{1}}}\right) L_{x}^{1}}
&\leq &C_{2}(1+t)^{-\frac{3}{2}}b_{1}b_{2}\text{,}%
\vspace {3mm}
\\
\left \Vert \Gamma _{gain}(h_{1},h_{2})(t)\right \Vert _{L_{\xi }^{\infty
}\left( \left \langle \xi \right \rangle ^{p_{2}-\gamma +1}e^{\hat{%
\varepsilon}\left \langle \xi \right \rangle ^{p_{1}}}\right) L_{x}^{\infty
}} &\leq &C_{2}(1+t)^{-\frac{3}{2}}b_{1}b_{2}\text{.}
\end{eqnarray*}%
Combining this with Lemmas \ref{lemma: estimate of nonlinear term} and \ref%
{lemma: eps2^n}, we obtain
\begin{eqnarray*}
&&\left \Vert \Gamma (h_{1},h_{2})(t)\right \Vert _{L_{\xi }^{\infty }\left(
\left \langle \xi \right \rangle ^{p_{2}+j}e^{\hat{\varepsilon}\left \langle
\xi \right \rangle ^{p_{1}}}\right) L_{x}^{1}} \\
&\leq &2\left \Vert \Gamma (h_{1},h_{2})(t)\right \Vert _{L_{\xi }^{\infty
}\left( \left \langle \xi \right \rangle ^{p_{2}+2j}e^{\hat{\varepsilon}%
\left \langle \xi \right \rangle ^{p_{1}}}\right) L_{x}^{1}}^{\frac{1}{2}%
}\left \Vert \Gamma (h_{1},h_{2})(t)\right \Vert _{L_{\xi }^{\infty }\left(
\left \langle \xi \right \rangle ^{p_{2}}e^{\hat{\varepsilon}\left \langle
\xi \right \rangle ^{p_{1}}}\right) L_{x}^{1}}^{\frac{1}{2}} \\
&\leq &C_{\gamma ,\hat{\varepsilon},p_{1},p_{2},j}(1+t)^{-\frac{3}{4}%
}b_{1}b_{2}\text{,}
\end{eqnarray*}%
and%
\begin{eqnarray*}
&&\left \Vert \Gamma (h_{1},h_{2})(t)\right \Vert _{L_{\xi }^{\infty }\left(
\left \langle \xi \right \rangle ^{p_{2}+j}e^{\hat{\varepsilon}\left \langle
\xi \right \rangle ^{p_{1}}}\right) L_{x}^{\infty }} \\
&\leq &2\left \Vert \Gamma (h_{1},h_{2})(t)\right \Vert _{L_{\xi }^{\infty
}\left( \left \langle \xi \right \rangle ^{p_{2}+2j}e^{\hat{\varepsilon}%
\left \langle \xi \right \rangle ^{p_{1}}}\right) L_{x}^{\infty }}^{\frac{1}{%
2}}\left \Vert \Gamma (h_{1},h_{2})(t)\right \Vert _{L_{\xi }^{\infty
}\left( \left \langle \xi \right \rangle ^{p_{2}}e^{\hat{\varepsilon}\left
\langle \xi \right \rangle ^{p_{1}}}\right) L_{x}^{\infty }}^{\frac{1}{2}} \\
&\leq &C_{\gamma ,\hat{\varepsilon},p_{1},p_{2},j}(1+t)^{-\frac{3}{4}%
}b_{1}b_{2}\text{.}
\end{eqnarray*}%
$%
\hfill%
\square $

\begin{proposition}
\label{Prop: inhomo}Assume that $g_{0}$ satisfies (\ref{cond-f0}) and that $%
h_{1}$ and $h_{2}$ satisfy (\ref{Cond-hi}). Then there exists a number $%
C_{\gamma ,\hat{\varepsilon},p_{1},p_{2},j}>0$ such that the solution $g$ to
(\ref{bot.1.r}) satisfies%
\begin{eqnarray*}
&&\max \left \{ (1+t)^{\frac{3}{4}}\left \Vert g(t)\right \Vert _{L_{\xi
}^{\infty }\left( \left \langle \xi \right \rangle ^{p_{2}}e^{\hat{%
\varepsilon}\left \langle \xi \right \rangle ^{p_{1}}}\right)
L_{x}^{2}},(1+t)^{\frac{3}{4}}\left \Vert g(t)\right \Vert _{L_{\xi
}^{\infty }\left( \left \langle \xi \right \rangle ^{p_{2}}e^{\hat{%
\varepsilon}\left \langle \xi \right \rangle ^{p_{1}}}\right) L_{x}^{\infty
}},\right. \\
&&\left. \left \Vert g(t)\right \Vert _{L_{\xi }^{\infty }\left( \left
\langle \xi \right \rangle ^{p_{2}+2j}e^{\hat{\varepsilon}\left \langle \xi
\right \rangle ^{p_{1}}}\right) L_{x}^{2}},\left \Vert g(t)\right \Vert
_{L_{\xi }^{\infty }\left( \left \langle \xi \right \rangle ^{p_{2}+2j}e^{%
\hat{\varepsilon}\left \langle \xi \right \rangle ^{p_{1}}}\right)
L_{x}^{\infty }}\right \} \leq C_{\gamma ,\hat{\varepsilon}%
,p_{1},p_{2},j}(b_{0}+b_{1}b_{2})\text{.}
\end{eqnarray*}
\end{proposition}

\proof%
By Duhamel's principle, $g$ can be written as
\begin{equation*}
g(t)=\mathbb{G}^{t}g_{0}+\int_{0}^{t}\mathbb{G}^{t-s}\Gamma
(h_{1},h_{2})(s)ds\text{.}
\end{equation*}%
Hence, in view of Proposition \ref{thm: linear} and Lemma \ref{lemma: eps_n}%
,
\begin{eqnarray}\label{eq: N1}
&&\left \Vert g(t)\right \Vert _{L_{\xi }^{\infty }\left( \left \langle \xi
\right \rangle ^{p_{2}}e^{\hat{\varepsilon}\left \langle \xi \right \rangle
^{p_{1}}}\right) L_{x}^{\infty }}\notag \\
&\leq &\left \Vert \mathbb{G}^{t}g_{0}\right \Vert _{L_{\xi }^{\infty
}\left( \left \langle \xi \right \rangle ^{p_{2}}e^{\hat{\varepsilon}\left
\langle \xi \right \rangle ^{p_{1}}}\right) L_{x}^{\infty
}}+\int_{0}^{t}\left \Vert \mathbb{G}^{t-s}\Gamma (h_{1},h_{2})(s)\right
\Vert _{L_{\xi }^{\infty }\left( \left \langle \xi \right \rangle ^{p_{2}}e^{%
\hat{\varepsilon}\left \langle \xi \right \rangle ^{p_{1}}}\right)
L_{x}^{\infty }}ds \notag\\
&\lesssim &(1+t)^{-\frac{3}{2}}\left( \left \Vert g_{0}\right \Vert _{L_{\xi
}^{\infty }\left( \left \langle \xi \right \rangle ^{p_{2}+j}e^{\hat{%
\varepsilon}\left \langle \xi \right \rangle ^{p_{1}}}\right)
L_{x}^{1}}+\left \Vert g_{0}\right \Vert _{L_{\xi }^{\infty }\left( \left
\langle \xi \right \rangle ^{p_{2}+j}e^{\hat{\varepsilon}\left \langle \xi
\right \rangle ^{p_{1}}}\right) L_{x}^{\infty }}\right) \\
&&+\int_{0}^{t}(1+t-s)^{-\frac{3}{2}}\left \Vert \Gamma
(h_{1},h_{2})(s)\right \Vert _{L_{\xi }^{\infty }\left( \left \langle \xi
\right \rangle ^{p_{2}+j}e^{\hat{\varepsilon}\left \langle \xi \right
\rangle ^{p_{1}}}\right) L_{x}^{1}} ds\notag\\
&&+\int_{0}^{t}(1+t-s)^{-\frac{3}{2}}\left \Vert \Gamma
(h_{1},h_{2})(s)\right \Vert _{L_{\xi }^{\infty }\left( \left \langle \xi
\right \rangle ^{p_{2}+j}e^{\hat{\varepsilon}\left \langle \xi \right
\rangle ^{p_{1}}}\right) L_{x}^{\infty}} ds\notag  \\
&\lesssim &(1+t)^{-\frac{3}{2}}b_{0}+\int_{0}^{t}(1+t-s)^{-\frac{3}{2}%
}\left( (1+s)^{-\frac{3}{4}}+(1+s)^{-\frac{3}{4}}\right) b_{1}b_{2}ds \notag\\
&\lesssim &(1+t)^{-\frac{3}{4}}\left( b_{0}+b_{1}b_{2}\right) \text{,}\notag
\end{eqnarray}%
i.e.,
\begin{equation*}
\left \Vert g(t)\right \Vert _{L_{\xi }^{\infty }\left( \left \langle \xi
\right \rangle ^{p_{2}}e^{\hat{\varepsilon}\left \langle \xi \right \rangle
^{p_{1}}}\right) L_{x}^{\infty }}\lesssim (1+t)^{-\frac{3}{4}}\left(
b_{0}+b_{1}b_{2}\right) \text{.}
\end{equation*}%
This completes the estimate for $\left \Vert g(t)\right \Vert _{L_{\xi
}^{\infty }\left( \left \langle \xi \right \rangle ^{p_{2}}e^{\hat{%
\varepsilon}\left \langle \xi \right \rangle ^{p_{1}}}\right) L_{x}^{\infty
}}$, and the estimate for $(1+t)^{\frac{3}{4}}\left \Vert g(t)\right \Vert
_{L_{\xi }^{\infty }\left( \left \langle \xi \right \rangle ^{p_{2}}e^{\hat{%
\varepsilon}\left \langle \xi \right \rangle ^{p_{1}}}\right) L_{x}^{2}}$
can be obtained via the same argument.

On the other hand, regarding (\ref{bot.1.r}) as the damped transport
equation with the source term $Kg+\Gamma (h_{1},h_{2})$, $g$ can be
rewritten as
\begin{equation*}
g\left( t\right) =\mathbb{S}^{t}g_{0}+\int_{0}^{t}\mathbb{S}^{t-s}\left[
Kg(s)+\Gamma (h_{1},h_{2})(s)\right] ds\text{.}
\end{equation*}%
Let $T>0$. Hence, for any $0\leq t\leq T$,
\begin{eqnarray*}
\left \vert g(t)\right \vert _{L_{x}^{\infty }} &\leq &\left \vert \mathbb{S}%
^{t}g_{0}\right \vert _{L_{x}^{\infty }}+\int_{0}^{t}\left \vert \mathbb{S}%
^{t-s}\left( Kg(s)+\Gamma (h_{1},h_{2})(s)\right) \right \vert
_{L_{x}^{\infty }}ds \\
&\leq &e^{-\nu (\xi )t}\left \vert g_{0}\right \vert _{L_{x}^{\infty
}}+\int_{0}^{t}e^{-\nu (\xi )(t-s)}\left \vert Kg(s)+\Gamma
(h_{1},h_{2})(s)\right \vert _{L_{x}^{\infty }}ds\text{.}
\end{eqnarray*}%
By assumption (\ref{cond-f0}), it immediately follows that%
\begin{equation}
e^{\hat{\varepsilon}\left \langle \xi \right \rangle ^{p_{1}}}\left \langle
\xi \right \rangle ^{p_{2}+2j}e^{-\nu (\xi )t}\left \vert g_{0}\right \vert
_{L_{x}^{\infty }}\leq b_{0}\text{.}  \label{eq: N2}
\end{equation}%
Next, in view of (\ref{K-Lp}) and (\ref{Kw6}),
\begin{eqnarray*}
&&e^{\hat{\varepsilon}\left \langle \xi \right \rangle ^{p_{1}}}\left
\langle \xi \right \rangle ^{p_{2}+2j}e^{-\nu (\xi )(t-s)}\left \vert
Kg(s)\right \vert _{L_{x}^{\infty }} \\
&\leq &\sup_{\left \vert \xi \right \vert \leq \tau }\left[ e^{\hat{%
\varepsilon}\left \langle \xi \right \rangle ^{p_{1}}}e^{-\nu (\xi
)(t-s)}\left \langle \xi \right \rangle ^{\gamma -1}\left \langle \xi \right
\rangle ^{1-\gamma }\left \langle \xi \right \rangle ^{p_{2}+2j}\left \vert
Kg(s)\right \vert _{L_{x}^{\infty }}\right] \\
&&+\sup_{\left \vert \xi \right \vert >\tau }\left[ e^{-\nu (\xi
)(t-s)}\left \langle \xi \right \rangle ^{\gamma -1}\left \langle \xi \right
\rangle ^{-1}\left( \left \langle \xi \right \rangle ^{2-\gamma }\left
\langle \xi \right \rangle ^{p_{2}+2j}e^{\hat{\varepsilon}\left \langle \xi
\right \rangle ^{p_{1}}}\left \vert Kg(s)\right \vert _{L_{x}^{\infty
}}\right) \right] \\
&\leq &C_{1}(1+t-s)^{\frac{1-\gamma }{\gamma }}\left( \left \langle \tau
\right \rangle ^{2j}e^{\hat{\varepsilon}\left \langle \tau \right \rangle
^{p_{1}}}\left \Vert Kg(s)\right \Vert _{L_{\xi ,p_{2}+1-\gamma }^{\infty
}L_{x}^{\infty }}+(1+\tau )^{-1}\left \Vert Kg(s)\right \Vert _{L_{\xi
}^{\infty }\left( \left \langle \xi \right \rangle ^{p_{2}+2j+2-\gamma }e^{%
\hat{\varepsilon}\left \langle \xi \right \rangle ^{p_{1}}}\right)
L_{x}^{\infty }}\right) \\
&\leq &C_{1}^{\prime }(1+t-s)^{\frac{1-\gamma }{\gamma }}\left( \left
\langle \tau \right \rangle ^{2j}e^{\hat{\varepsilon}\left \langle \tau
\right \rangle ^{p_{1}}}\left \Vert g(s)\right \Vert _{L_{\xi
,p_{2}}^{\infty }L_{x}^{\infty }}+(1+\tau )^{-1}\left \Vert g(s)\right \Vert
_{L_{\xi }^{\infty }\left( \left \langle \xi \right \rangle ^{p_{2}+2j}e^{%
\hat{\varepsilon}\left \langle \xi \right \rangle ^{p_{1}}}\right)
L_{x}^{\infty }}\right)
\end{eqnarray*}%
for any $\tau >0$. Picking $\tau >0$ such that $(1+\tau )^{-1}C_{1}^{\prime
}<\frac{1}{4}$, we get
\begin{eqnarray}
&&\int_{0}^{t}e^{\hat{\varepsilon}\left \langle \xi \right \rangle
^{p_{1}}}\left \langle \xi \right \rangle ^{p_{2}+2j}e^{-\nu (\xi
)(t-s)}\left \vert Kg(s)\right \vert _{L_{x}^{\infty }}ds  \label{eq: N3} \\
&\leq &\int_{0}^{t}(1+t-s)^{\frac{1-\gamma }{\gamma }}\left( C_{2}\left
\Vert g(s)\right \Vert _{L_{\xi ,p_{2}}^{\infty }L_{x}^{\infty }}+\frac{1}{4}%
\sup_{0\leq s\leq T}\left \Vert g(s)\right \Vert _{L_{\xi }^{\infty }\left(
\left \langle \xi \right \rangle ^{p_{2}+2j}e^{\hat{\varepsilon}\left
\langle \xi \right \rangle ^{p_{1}}}\right) L_{x}^{\infty }}\right) ds
\notag \\
&\leq &C_{3}\left( b_{0}+b_{1}b_{2}\right) +\frac{1}{2}\sup_{0\leq s\leq
T}\left \Vert g(s)\right \Vert _{L_{\xi }^{\infty }\left( \left \langle \xi
\right \rangle ^{p_{2}+2j}e^{\hat{\varepsilon}\left \langle \xi \right
\rangle ^{p_{1}}}\right) L_{x}^{\infty }}\text{,}  \notag
\end{eqnarray}%
here the estimate for $\left \Vert g(s)\right \Vert _{L_{\xi ,p_{2}}^{\infty
}L_{x}^{\infty }}$ following the same argument as (\ref{eq: N1}). Finally,
we split $\Gamma $ into two parts $\Gamma _{loss}$ and $\Gamma _{gain}$.
Then, it follows from Lemma \ref{lemma: estimate of nonlinear term} that
\begin{eqnarray}
&&\int_{0}^{t}e^{\hat{\varepsilon}\left \langle \xi \right \rangle
^{p_{1}}}\left \langle \xi \right \rangle ^{p_{2}+2j}e^{-\nu (\xi
)(t-s)}\left \vert \Gamma _{gain}(h_{1},h_{2})(s)\right \vert
_{L_{x}^{\infty }}ds  \label{eq: N4} \\
&\leq &\int_{0}^{t}C_{4}(1+t-s)^{\frac{1-\gamma }{\gamma }}\left \Vert
\Gamma _{gain}(h_{1},h_{2})(s)\right \Vert _{L_{\xi }^{\infty }\left( \left
\langle \xi \right \rangle ^{p_{2}+2j+1-\gamma }e^{\hat{\varepsilon}\left
\langle \xi \right \rangle ^{p_{1}}}\right) L_{x}^{\infty }}ds  \notag \\
&\leq &\int_{0}^{t}C_{5}(1+t-s)^{\frac{1-\gamma }{\gamma }}b_{1}b_{2}ds\leq
C_{5}b_{1}b_{2}  \notag
\end{eqnarray}%
and%
\begin{eqnarray}
&&\int_{0}^{t}e^{\hat{\varepsilon}\left \langle \xi \right \rangle
^{p_{1}}}\left \langle \xi \right \rangle ^{p_{2}+2j}e^{-\nu (\xi
)(t-s)}\left \vert \Gamma _{loss}(h_{1},h_{2})(s)\right \vert
_{L_{x}^{\infty }}ds  \label{eq: N5} \\
&\leq &\int_{0}^{t}(1+t-s)^{-1}\left \Vert \Gamma _{loss}(h_{1},h_{2})(s)\right \Vert
_{L_{\xi }^{\infty }\left( \left \langle \xi \right \rangle ^{p_{2}+2j-\gamma}e^{%
\hat{\varepsilon}\left \langle \xi \right \rangle ^{p_{1}}}\right)
L_{x}^{\infty }}ds  \notag \\
&\leq &C_{6}b_{1}b_{2}\int_{0}^{t}(1+t-s)^{-1}(1+s)^{-\frac{3}{4}}ds\leq C_{6}b_{1}b_{2}%
\text{.}  \notag
\end{eqnarray}%
Combining (\ref{eq: N2})-(\ref{eq: N5}) gives%
\begin{equation*}
\sup_{0\leq t\leq T}\left \Vert g(t)\right \Vert _{L_{\xi }^{\infty }\left(
\left \langle \xi \right \rangle ^{p_{2}+2j}e^{\hat{\varepsilon}\left
\langle \xi \right \rangle ^{p_{1}}}\right) L_{x}^{\infty }}\leq C_{\gamma ,%
\hat{\varepsilon},p_{1},p_{2},j}(b_{0}+b_{1}b_{2})\text{,}
\end{equation*}%
which implies that%
\begin{equation*}
\left \Vert g(t)\right \Vert _{L_{\xi }^{\infty }\left( \left \langle \xi
\right \rangle ^{p_{2}+2j}e^{\hat{\varepsilon}\left \langle \xi \right
\rangle ^{p_{1}}}\right) L_{x}^{\infty }}\leq C_{\gamma ,\hat{\varepsilon}%
,p_{1},p_{2},j}(b_{0}+b_{1}b_{2})
\end{equation*}%
for $0\leq t<\infty $, since $T>0$ is arbitrary. The estimate for $%
\left
\Vert g(t)\right \Vert _{L_{\xi }^{\infty }\left( \left \langle \xi
\right
\rangle ^{p_{2}+2j}e^{\hat{\varepsilon}\left \langle \xi
\right
\rangle ^{p_{1}}}\right) L_{x}^{2}}$ can be obtained via the same
argument as well. The proof of this proposition is completed.$%
\hfill%
\square $

\subsection{Proof of Theorem \protect \ref{prop: nonlinear}.}

Define a norm as
\begin{align*}
|||h|||\equiv \sup_{t}& \left \{ (1+t)^{\frac{3}{4}}\left \Vert h(t)\right
\Vert _{L_{\xi }^{\infty }\left( \left \langle \xi \right \rangle ^{p_{2}}e^{%
\hat{\varepsilon}\left \langle \xi \right \rangle ^{p_{1}}}\right)
L_{x}^{2}},(1+t)^{\frac{3}{4}}\left \Vert h(t)\right \Vert _{L_{\xi
}^{\infty }\left( \left \langle \xi \right \rangle ^{p_{2}}e^{\hat{%
\varepsilon}\left \langle \xi \right \rangle ^{p_{1}}}\right) L_{x}^{\infty
}},\right. \text{ } \\
& \left. \left \Vert h(t)\right \Vert _{L_{\xi }^{\infty }\left( \left
\langle \xi \right \rangle ^{p_{2}+2j}e^{\hat{\varepsilon}\left \langle \xi
\right \rangle ^{p_{1}}}\right) L_{x}^{2}},\left \Vert h(t)\right \Vert
_{L_{\xi }^{\infty }\left( \left \langle \xi \right \rangle ^{p_{2}+2j}e^{%
\hat{\varepsilon}\left \langle \xi \right \rangle ^{p_{1}}}\right)
L_{x}^{\infty }}\right \} \text{.}
\end{align*}

Now, we consider the iteration $\left \{ f^{(i)}\right \} $ for which $%
f^{(0)}\left( t,x,\xi \right) \equiv 0$ and $f^{\left( i+1\right) }$, $i\in
\mathbb{N\cup \{}0\mathbb{\}}$, is a solution to the equation
\begin{equation}
\left \{
\begin{array}{l}
\displaystyle \partial _{t}f^{(i+1)}+\xi \cdot \nabla
_{x}f^{(i+1)}=Lf^{(i+1)}+\Gamma (f^{(i)},f^{(i)})\text{,} \\[4mm]
\displaystyle f^{(i+1)}(0,x,\xi )=\eta f_{0}(x,\xi )\text{,}%
\end{array}%
\right.  \label{inducion-1}
\end{equation}%
where $\eta >0$ is sufficiently small such that%
\begin{equation*}
\left( 1+C_{\gamma ,\hat{\varepsilon},p_{1},p_{2},j}\right) ^{2}\eta \left(
\left \Vert f_{0}\right \Vert _{L_{\xi }^{\infty }\left( \left \langle \xi
\right \rangle ^{p_{2}+2j}e^{\hat{\varepsilon}\left \langle \xi \right
\rangle ^{p_{1}}}\right) L_{x}^{1}}+\left \Vert f_{0}\right \Vert _{L_{\xi
}^{\infty }\left( \left \langle \xi \right \rangle ^{p_{2}+2j}e^{\hat{%
\varepsilon}\left \langle \xi \right \rangle ^{p_{1}}}\right) L_{x}^{\infty
}}\right) <1/8\text{.}
\end{equation*}%
Denote%
\begin{equation*}
b_{0}:=\eta \left( \left \Vert f_{0}\right \Vert _{L_{\xi }^{\infty }\left(
\left \langle \xi \right \rangle ^{p_{2}+2j}e^{\hat{\varepsilon}\left
\langle \xi \right \rangle ^{p_{1}}}\right) L_{x}^{1}}+\left \Vert
f_{0}\right \Vert _{L_{\xi }^{\infty }\left( \left \langle \xi \right
\rangle ^{p_{2}+2j}e^{\hat{\varepsilon}\left \langle \xi \right \rangle
^{p_{1}}}\right) L_{x}^{\infty }}\right) \text{,}
\end{equation*}%
According to Proposition \ref{Prop: inhomo},
\begin{equation*}
|||f^{(1)}|||\leq C_{\gamma ,\hat{\varepsilon},p_{1},p_{2},j}b_{0}\leq
2C_{\gamma ,\hat{\varepsilon},p_{1},p_{2},j}b_{0}\text{,}
\end{equation*}%
and then
\begin{eqnarray*}
|||f^{(2)}||| &\leq &C_{\gamma ,\hat{\varepsilon},p_{1},p_{2},j}\left[
b_{0}+\left( 2C_{\gamma ,\hat{\varepsilon},p_{1},p_{2},j}b_{0}\right) ^{2}%
\right] \\
&\leq &2C_{\gamma ,\hat{\varepsilon},p_{1},p_{2},j}b_{0}\left[ \frac{1}{2}%
+2C_{\gamma ,\hat{\varepsilon},p_{1},p_{2},j}b_{0}\right] \\
&\leq &2C_{\gamma ,\hat{\varepsilon},p_{1},p_{2},j}b_{0}\text{.}
\end{eqnarray*}%
By induction on $i$, we get
\begin{equation*}
|||f^{(i+1)}|||\leq 2C_{\gamma ,\hat{\varepsilon},p_{1},p_{2},j}b_{0}\text{,}
\end{equation*}%
through Proposition \ref{Prop: inhomo}. Hence, we get the boundedness of $%
\left \{ f^{(i)}\right \} $ in the norm $|||\cdot |||$.

Next, we demonstrate the convergence of $\left \{ f^{(i)}\right \} $ and
uniqueness of its limit. Set $h^{(i+1)}=f^{(i+1)}-f^{(i)}$ and then $%
h^{(i+1)}$ satisfies the equation
\begin{equation}
\left \{
\begin{array}{l}
\displaystyle \partial _{t}h^{(i+1)}+\xi \cdot \nabla
_{x}h^{(i+1)}=Lh^{(i+1)}+\Gamma (h^{(i)},f^{(i)})+\Gamma (f^{(i-1)},h^{(i)})%
\text{,}%
\vspace {3mm}
\\[4mm]
\displaystyle h^{(i+1)}(0,x,\xi )=0\text{.}%
\end{array}%
\right.  \label{inducion-2}
\end{equation}%
According to Proposition \ref{Prop: inhomo}, we get
\begin{equation*}
|||h^{(i+1)}|||\leq 4\left( C_{\gamma ,\hat{\varepsilon},p_{1},p_{2},j}%
\right) ^{2}b_{0}|||h^{(i)}|||
\end{equation*}%
for all $i\in \mathbb{N}$. Since $4\left( C_{\gamma ,\hat{\varepsilon}%
,p_{1},p_{2},j}\right) ^{2}b_{0}<1/2$, $\left \{ f^{(i)}\right \} $ is a
Cauchy sequence in the norm $|||\cdot |||$, so that it converges and its
limit $f$ will satisfy
\begin{eqnarray}
\left \Vert w_{3}f(t)\right \Vert _{L_{\xi ,p_{2}}^{\infty }L_{x}^{2}} &\leq
&\eta C_{1}(1+t)^{-\frac{3}{4}}\left( \left \Vert w_{3}f_{0}\right \Vert
_{L_{\xi ,p_{2}+2j}^{\infty }L_{x}^{1}}+\left \Vert w_{3}f_{0}\right \Vert
_{L_{\xi ,p_{2}+2j}^{\infty }L_{x}^{\infty }}\right) \text{,}%
\vspace {3mm}
\label{eq: N. 1. a-1} \\
\left \Vert w_{3}f(t)\right \Vert _{L_{\xi ,p_{2}}^{\infty }L_{x}^{\infty }}
&\leq &\eta C_{2}(1+t)^{-\frac{3}{4}}\left( \left \Vert w_{3}f_{0}\right
\Vert _{L_{\xi ,p_{2}+2j}^{\infty }L_{x}^{1}}+\left \Vert w_{3}f_{0}\right
\Vert _{L_{\xi ,p_{2}+2j}^{\infty }L_{x}^{\infty }}\right) \text{,}
\label{eq: N. 1. aa-1}
\end{eqnarray}
\begin{eqnarray}
\left \Vert w_{3}f(t)\right \Vert _{L_{\xi ,p_{2}+2j}^{\infty }L_{x}^{2}}
&\leq &\eta \bar{C}_{1}\left( \left \Vert w_{3}f_{0}\right \Vert _{L_{\xi
,p_{2}+2j}^{\infty }L_{x}^{1}}+\left \Vert w_{3}f_{0}\right \Vert _{L_{\xi
,p_{2}+2j}^{\infty }L_{x}^{\infty }}\right) \text{,}%
\vspace {3mm}
\label{eq: N. 1. b-1} \\
\left \Vert w_{3}f(t)\right \Vert _{L_{\xi ,p_{2}+2j}^{\infty }L_{x}^{\infty
}} &\leq &\eta \bar{C}_{2}\left( \left \Vert w_{3}f_{0}\right \Vert _{L_{\xi
,p_{2}+2j}^{\infty }L_{x}^{1}}+\left \Vert w_{3}f_{0}\right \Vert _{L_{\xi
,p_{2}+2j}^{\infty }L_{x}^{\infty }}\right) \text{.}  \label{eq: N. 1. bb-1}
\end{eqnarray}%
Therefore, (\ref{eq: N. 1. a}), (\ref{eq: N. 1. b}), (\ref{eq: N. 1. bb})
are obtained.

Finally, we will use a bootstrap argument to improve the estimate (\ref{eq:
N. 1. aa-1}). Write $f$ as
\begin{equation*}
f=\eta \mathbb{G}^{t}f_{0}+\int_{0}^{t}\mathbb{G}^{t-s}\Gamma (f,f)(s)ds%
\text{,}
\end{equation*}%
and then we have%
\begin{eqnarray*}
&&\left \Vert w_{3}f(t)\right \Vert _{L_{\xi ,p_{2}}^{\infty }L_{x}^{\infty
}} \\
&\leq &\eta \left \Vert w_{3}\mathbb{G}^{t}f_{0}\right \Vert _{L_{\xi
,p_{2}}^{\infty }L_{x}^{\infty }}+\int_{0}^{t}\left \Vert w_{3}\mathbb{G}%
^{t-s}\Gamma (f,f)(s)\right \Vert _{L_{\xi ,p_{2}}^{\infty }L_{x}^{\infty
}}ds \\
&\lesssim &\eta \left( 1+t\right) ^{-\frac{3}{2}}\left( \left \Vert
w_{3}f_{0}\right \Vert _{L_{\xi ,p_{2}+j}^{\infty }L_{x}^{1}}+\left \Vert
w_{3}f_{0}\right \Vert _{L_{\xi ,p_{2}+j}^{\infty }L_{x}^{\infty }}\right) \\
&&+\int_{0}^{t}\left( 1+t-s\right) ^{-\frac{3}{2}}\left( \left \Vert
w_{3}\Gamma (f,f)(s)\right \Vert _{L_{\xi ,p_{2}+j}^{\infty
}L_{x}^{1}}+\left \Vert w_{3}\Gamma (f,f)(s)\right \Vert _{L_{\xi
,p_{2}+j}^{\infty }L_{x}^{\infty }}\right) ds \\
&\lesssim &\eta \left( 1+t\right) ^{-\frac{3}{2}}\left( \left \Vert
w_{3}f_{0}\right \Vert _{L_{\xi ,p_{2}+j}^{\infty }L_{x}^{1}}+\left \Vert
w_{3}f_{0}\right \Vert _{L_{\xi ,p_{2}+j}^{\infty }L_{x}^{\infty }}\right) \\
&&+\int_{0}^{t}\left( 1+t-s\right) ^{-\frac{3}{2}}\left( \left \Vert
w_{3}f\right \Vert _{L_{\xi ,p_{2}+j}^{\infty }L_{x}^{2}}^{2}+\left \Vert
w_{3}f\right \Vert _{L_{\xi ,p_{2}+j}^{\infty }L_{x}^{\infty }}^{2}\right) ds
\\
&\lesssim &\eta \left( 1+t\right) ^{-\frac{3}{2}}\left( \left \Vert
w_{3}f_{0}\right \Vert _{L_{\xi ,p_{2}+j}^{\infty }L_{x}^{1}}+\left \Vert
w_{3}f_{0}\right \Vert _{L_{\xi ,p_{2}+j}^{\infty }L_{x}^{\infty }}\right) \\
&&+\eta \int_{0}^{t}\left( 1+t-s\right) ^{-\frac{3}{2}}\left( 1+s\right) ^{-%
\frac{3}{2}}ds\cdot \left( \left \Vert w_{3}f_{0}\right \Vert _{L_{\xi
,p_{2}+3j}^{\infty }L_{x}^{1}}+\left \Vert w_{3}f_{0}\right \Vert _{L_{\xi
,p_{2}+3j}^{\infty }L_{x}^{\infty }}\right) \\
&\lesssim &\eta \left( 1+t\right) ^{-\frac{3}{2}}\left( \left \Vert
w_{3}f_{0}\right \Vert _{L_{\xi ,p_{2}+3j}^{\infty }L_{x}^{1}}+\left \Vert
w_{3}f_{0}\right \Vert _{L_{\xi ,p_{2}+3j}^{\infty }L_{x}^{\infty }}\right)
\text{,}
\end{eqnarray*}%
by using (\ref{Gamma-w3-sup}), Proposition \ref{thm: linear}, and (\ref{eq:
N. 1. aa-1}). This completes the proof.%

\section{Appendix}

\subsection{\textbf{Proof of (\protect \ref{Gamma-product}) and (\protect \ref%
{Gamma-L2-1}). }}

We claim that
\begin{equation}
\left( \int \left \vert \nu ^{-1/2}(\xi )\Gamma (g,h)(\xi )\right \vert
^{2}d\xi \right) ^{1/2}\lesssim \left \vert g\right \vert _{L_{\xi }^{\infty
}}\left \vert h\right \vert _{L_{\sigma }^{2}}+\left \vert g\right \vert
_{L_{\sigma }^{2}}\left \vert h\right \vert _{L_{\xi }^{\infty }}\text{,}
\label{Gamma-product-1}
\end{equation}%
\begin{equation}
\left( \int \left \vert \nu ^{-1}(\xi )\Gamma (g,h)(\xi )\right \vert
^{2}d\xi \right) ^{1/2}\lesssim \left \vert g\right \vert _{L_{\xi }^{\infty
}}\left \vert h\right \vert _{L_{\xi }^{2}}+\left \vert g\right \vert
_{L_{\xi }^{2}}\left \vert h\right \vert _{L_{\xi }^{\infty }}\text{.}
\label{Gamma-L2-2}
\end{equation}%
Recall that we split $\Gamma $ into two parts $\Gamma _{gain}$ and $\Gamma
_{loss}$ as below:%
\begin{eqnarray*}
\Gamma (g,h) &\equiv &\Gamma _{gain}(g,h)-\Gamma _{loss}(g,h) \\
&=&\frac{1}{2}\int_{\mathbb{R}^{3}\times \mathbb{S}^{2}}B(\vartheta )|\xi
-\xi _{\ast }|^{\gamma }\mathcal{M}_{\ast }^{1/2}\left[ g_{\ast }^{\prime
}h^{\prime }+g^{\prime }h_{\ast }^{\prime }\right] d\xi _{\ast }d\omega \\
&&-\frac{1}{2}\int_{\mathbb{R}^{3}\times \mathbb{S}^{2}}B(\vartheta )|\xi
-\xi _{\ast }|^{\gamma }\mathcal{M}_{\ast }^{1/2}\left[ g_{\ast }h+gh_{\ast }%
\right] d\xi _{\ast }d\omega \text{.}
\end{eqnarray*}%
In the sequel, we shall estimate $\Gamma _{gain}$ and $\Gamma _{loss}$
individually. $%
\vspace {3mm}%
$\newline
\textbf{Estimate on }$\Gamma _{loss}(g,h)$\textbf{.} It readily follows from
Lemma \ref{[Guo]} that
\begin{equation}
\left \vert \Gamma _{loss}(g,h)\right \vert \lesssim \nu \left( \xi \right)
\left( \left \vert g\right \vert _{L_{\xi }^{\infty }}\left \vert h\right
\vert +\left \vert g\right \vert \left \vert h\right \vert _{L_{\xi
}^{\infty }}\right) \text{.}  \label{loss-sup}
\end{equation}%
Therefore, we have%
\begin{equation}
\left( \int \left \vert \nu ^{-1/2}(\xi )\Gamma _{loss}(g,h)(\xi )\right
\vert ^{2}d\xi \right) ^{1/2}\lesssim \left \vert g\right \vert _{L_{\xi
}^{\infty }}\left \vert h\right \vert _{L_{\sigma }^{2}}+\left \vert g\right
\vert _{L_{\sigma }^{2}}\left \vert h\right \vert _{L_{\xi }^{\infty }}\text{%
,}  \label{loss-product}
\end{equation}%
\begin{equation}
\left( \int \left \vert \nu ^{-1}(\xi )\Gamma _{loss}(g,h)(\xi )\right \vert
^{2}d\xi \right) ^{1/2}\lesssim \left \vert g\right \vert _{L_{\xi }^{\infty
}}\left \vert h\right \vert _{L_{\xi }^{2}}+\left \vert g\right \vert
_{L_{\xi }^{2}}\left \vert h\right \vert _{L_{\xi }^{\infty }}\text{.}
\label{loss-L2}
\end{equation}%
$%
\vspace {3mm}%
$\newline
\textbf{Estimate on }$\Gamma _{gain}(g,h)$\textbf{. }By the Cauchy-Schwartz
inequality and Lemma \ref{[Guo]},
\begin{eqnarray}
&&\left \vert \Gamma _{gain}(g,h)(\xi )\right \vert ^{2}  \notag \\
&\lesssim &\nu \left( \xi \right) \left( \int_{\mathbb{R}^{3}\times \mathbb{S%
}^{2}}\left \vert B(\vartheta )\right \vert |\xi -\xi _{\ast }|^{\gamma
}\exp \left( -\frac{\left \vert \xi _{\ast }\right \vert ^{2}}{4}\right) %
\left[ \left \vert g_{\ast }^{\prime }\right \vert ^{2}\left \vert h^{\prime
}\right \vert ^{2}+\left \vert g^{\prime }\right \vert ^{2}\left \vert
h_{\ast }^{\prime }\right \vert ^{2}\right] d\xi _{\ast }d\omega \right)
\text{,}  \label{gain^2}
\end{eqnarray}%
so that
\begin{eqnarray*}
&&\int_{\mathbb{R}^{3}}\left \vert \nu ^{-1}(\xi )\Gamma _{gain}(g,h)(\xi
)\right \vert ^{2}d\xi \\
&\lesssim &\int_{\mathbb{S}^{2}}B(\vartheta )\left( \int_{\mathbb{R}%
^{3}\times \mathbb{R}^{3}}\nu \left( \xi \right) ^{-1}|\xi -\xi _{\ast
}|^{\gamma }\exp (-\frac{\left \vert \xi _{\ast }\right \vert ^{2}}{4})\left[
\left \vert g_{\ast }^{\prime }\right \vert ^{2}\left \vert h^{\prime
}\right \vert ^{2}+\left \vert g^{\prime }\right \vert ^{2}\left \vert
h_{\ast }^{\prime }\right \vert ^{2}\right] d\xi _{\ast }d\xi \right)
d\omega \text{.}
\end{eqnarray*}%
We split the $(\xi _{\ast },\xi )$-space into three regions: $I_{1}=\left \{
\left \vert \xi _{\ast }\right \vert \geq \frac{\left \vert \xi \right \vert
}{2}\right \} $, $I_{2}=\left \{ \left \vert \xi _{\ast }\right \vert <\frac{%
\left \vert \xi \right \vert }{2},\left \vert \xi \right \vert \leq
1\right
\} $, and $I_{3}=\left \{ \left \vert \xi _{\ast }\right \vert <%
\frac{\left
\vert \xi \right \vert }{2},\left \vert \xi \right \vert
>1\right \} $.$%
\vspace {3mm}%
$\newline
\textbf{Case 1}: On $I_{1}\equiv \left \{ \left \vert \xi _{\ast
}\right
\vert \geq \left \vert \xi \right \vert /2\right \} $. Since $|\xi
-\xi _{\ast }|=|\xi ^{\prime }-\xi _{\ast }^{\prime }|$, $\left \vert \xi
\right
\vert ^{2}+\left \vert \xi _{\ast }\right \vert ^{2}=\left \vert \xi
^{\prime }\right \vert ^{2}+\left \vert \xi _{\ast }^{\prime }\right \vert
^{2}$, and $\nu ^{-1}\approx \left \langle \xi \right \rangle ^{-\gamma
}\lesssim \left \langle \xi ^{\prime }\right \rangle ^{-\gamma
}\left
\langle \xi _{\ast }^{\prime }\right \rangle ^{-\gamma }$, we have
\begin{eqnarray}
&&\int_{\mathbb{S}^{2}}B(\vartheta )\left( \int_{I_{1}}\nu \left( \xi
\right) ^{-1}|\xi -\xi _{\ast }|^{\gamma }\exp (-\frac{\left \vert \xi
_{\ast }\right \vert ^{2}}{4})\left[ \left \vert g_{\ast }^{\prime }\right
\vert ^{2}\left \vert h^{\prime }\right \vert ^{2}+\left \vert g^{\prime
}\right \vert ^{2}\left \vert h_{\ast }^{\prime }\right \vert ^{2}\right]
d\xi _{\ast }d\xi \right) d\omega  \label{gain-1} \\
&\lesssim &\int_{\mathbb{S}^{2}}B(\vartheta )\left( \int_{I_{1}}\nu \left(
\xi \right) ^{-1}|\xi -\xi _{\ast }|^{\gamma }\exp (-\frac{\left \vert \xi
_{\ast }\right \vert ^{2}}{8}-\frac{\left \vert \xi \right \vert ^{2}}{32})%
\left[ \left \vert g_{\ast }^{\prime }\right \vert ^{2}\left \vert h^{\prime
}\right \vert ^{2}+\left \vert g^{\prime }\right \vert ^{2}\left \vert
h_{\ast }^{\prime }\right \vert ^{2}\right] d\xi _{\ast }d\xi \right) d\omega
\notag \\
&\lesssim &\int_{\mathbb{S}^{2}}B(\vartheta )\left( \int_{\psi (I_{1})}|\xi
^{\prime }-\xi _{\ast }^{\prime }|^{\gamma }\exp \left( -\left( \frac{\left
\vert \xi _{\ast }^{\prime }\right \vert ^{2}}{32}+\frac{\left \vert \xi
^{\prime }\right \vert ^{2}}{32}\right) \right) \left[ \left \vert g_{\ast
}^{\prime }\right \vert ^{2}\left \vert h^{\prime }\right \vert ^{2}+\left
\vert g^{\prime }\right \vert ^{2}\left \vert h_{\ast }^{\prime }\right
\vert ^{2}\right] d\xi _{\ast }^{\prime }d\xi ^{\prime }\right) d\omega
\notag \\
&\lesssim &\int_{\mathbb{S}^{2}}B(\vartheta )\left( \int_{I_{1}}|\xi -\xi
_{\ast }|^{\gamma }\exp \left( -\left( \frac{\left \vert \xi _{\ast }\right
\vert ^{2}}{32}+\frac{\left \vert \xi \right \vert ^{2}}{32}\right) \right) %
\left[ \left \vert g_{\ast }\right \vert ^{2}\left \vert h\right \vert
^{2}+\left \vert g\right \vert ^{2}\left \vert h_{\ast }\right \vert ^{2}%
\right] d\xi _{\ast }d\xi \right) d\omega  \notag \\
&\lesssim &\left \vert g\right \vert _{L_{\xi }^{\infty }}^{2}\left \vert
h\right \vert _{L_{\sigma }^{2}}^{2}+\left \vert g\right \vert _{L_{\sigma
}^{2}}^{2}\left \vert h\right \vert _{L_{\xi }^{\infty }}^{2}\text{,}  \notag
\end{eqnarray}%
by change of the variables $(\xi _{\ast },\xi )\overset{\psi }{%
\longrightarrow }(\xi _{\ast }^{\prime },\xi ^{\prime })$ and Lemma \ref%
{[Guo]}.$%
\vspace {3mm}%
$ \newline
\noindent \textbf{Case 2}: On $I_{2}=\left \{ \left \vert \xi _{\ast
}\right
\vert <\left \vert \xi \right \vert /2,\left \vert \xi \right \vert
\leq 1\right \} $. In this region, $\left \vert \xi -\xi _{\ast
}\right
\vert >\left \vert \xi \right \vert /2$, $1\leq \nu ^{-1}(\xi )\leq
2$, and
\begin{equation*}
\left \vert \xi _{\ast }^{\prime }\right \vert ,\left \vert \xi ^{\prime
}\right \vert \leq \sqrt{2}\left( \left \vert \xi ^{\prime }\right \vert
^{2}+\left \vert \xi _{\ast }^{\prime }\right \vert ^{2}\right) ^{1/2}=\sqrt{%
2}\left( \left \vert \xi \right \vert ^{2}+\left \vert \xi _{\ast }\right
\vert ^{2}\right) ^{1/2}\leq 2\left \vert \xi \right \vert <2\text{.}
\end{equation*}%
Hence,
\begin{eqnarray}
&&\int_{\mathbb{S}^{2}}B(\vartheta )\left( \int_{I_{2}}\nu \left( \xi
\right) ^{-1}|\xi -\xi _{\ast }|^{\gamma }\exp (-\frac{\left \vert \xi
_{\ast }\right \vert ^{2}}{4})\left[ \left \vert g_{\ast }^{\prime }\right
\vert ^{2}\left \vert h^{\prime }\right \vert ^{2}+\left \vert g^{\prime
}\right \vert ^{2}\left \vert h_{\ast }^{\prime }\right \vert ^{2}\right]
d\xi _{\ast }d\xi \right) d\omega  \label{gain-2} \\
&\lesssim &\int_{\mathbb{S}^{2}}B(\vartheta )\left( \int_{I_{2}}|\xi
|^{\gamma }\exp (-\frac{\left \vert \xi _{\ast }\right \vert ^{2}}{4})\left[
\left \vert g_{\ast }^{\prime }\right \vert ^{2}\left \vert h^{\prime
}\right \vert ^{2}+\left \vert g^{\prime }\right \vert ^{2}\left \vert
h_{\ast }^{\prime }\right \vert ^{2}\right] d\xi _{\ast }d\xi \right) d\omega
\notag \\
&\lesssim &\int_{\mathbb{S}^{2}}B(\vartheta )\left( \int_{I_{2}}|\xi _{\ast
}^{\prime }|^{\gamma }\left[ \left \vert g_{\ast }^{\prime }\right \vert
^{2}\left \vert h^{\prime }\right \vert ^{2}+\left \vert g^{\prime }\right
\vert ^{2}\left \vert h_{\ast }^{\prime }\right \vert ^{2}\right] d\xi
_{\ast }d\xi \right) d\omega  \notag \\
&\lesssim &\int_{\mathbb{S}^{2}}B(\vartheta )\left( \int_{\psi \left(
I_{2}\right) }|\xi _{\ast }^{\prime }|^{\gamma }\left[ \left \vert g_{\ast
}^{\prime }\right \vert ^{2}\left \vert h^{\prime }\right \vert ^{2}+\left
\vert g^{\prime }\right \vert ^{2}\left \vert h_{\ast }^{\prime }\right
\vert ^{2}\right] d\xi _{\ast }^{\prime }d\xi ^{\prime }\right) d\omega
\notag \\
&\lesssim &\int_{\mathbb{S}^{2}}B(\vartheta )\left( \int_{I_{2}}|\xi _{\ast
}|^{\gamma }\left[ \left \vert g_{\ast }\right \vert ^{2}\left \vert h\right
\vert ^{2}+\left \vert g\right \vert ^{2}\left \vert h_{\ast }\right \vert
^{2}\right] d\xi _{\ast }d\xi \right) d\omega  \notag \\
&\lesssim &\left \vert g\right \vert _{L_{\xi }^{\infty }}^{2}\left \vert
h\right \vert _{L_{\sigma }^{2}}^{2}+\left \vert g\right \vert _{L_{\sigma
}^{2}}^{2}\left \vert h\right \vert _{L_{\xi }^{\infty }}^{2}\text{.}  \notag
\end{eqnarray}%
The last inequality is valid since $\int_{\left \vert \xi _{\ast
}\right
\vert <1/2}|\xi _{\ast }|^{\gamma }d\xi _{\ast }<\infty $,
\begin{equation*}
\int_{\left \vert \xi \right \vert \leq 1}\left \vert h\right \vert ^{2}d\xi
=\int_{\left \vert \xi \right \vert \leq 1}\nu ^{-1}\left( \xi \right) \nu
\left( \xi \right) \left \vert h\right \vert ^{2}d\xi \leq C_{\gamma
}\int_{\left \vert \xi \right \vert \leq 1}\nu \left( \xi \right) \left
\vert h\right \vert ^{2}d\xi \leq C_{r}\left \vert h\right \vert _{L_{\sigma
}^{2}}^{2}\text{,}
\end{equation*}%
\begin{equation*}
\int_{\left \vert \xi \right \vert \leq 1}\left \vert g\right \vert ^{2}d\xi
=\int_{\left \vert \xi \right \vert \leq 1}\nu ^{-1}\left( \xi \right) \nu
\left( \xi \right) \left \vert g\right \vert ^{2}d\xi \leq C_{\gamma
}\int_{\left \vert \xi \right \vert \leq 1}\nu \left( \xi \right) \left
\vert h\right \vert ^{2}d\xi \leq C_{r}\left \vert g\right \vert \text{,}
\end{equation*}%
for some $C_{\gamma }>0$.$%
\vspace {3mm}%
$\newline
\textbf{Case 3}: On $I_{3}=\left \{ \left \vert \xi _{\ast }\right \vert
<\left \vert \xi \right \vert /2,\left \vert \xi \right \vert >1\right \} $.
In this region, $\left \vert \xi -\xi _{\ast }\right \vert >\left \vert \xi
\right \vert /2>\left( 1+\left \vert \xi \right \vert \right) /4$; moreover,
\begin{equation*}
\frac{\left \langle \xi ^{\prime }\right \rangle }{\sqrt{2}}\text{, }\frac{%
\left \langle \xi _{\ast }^{\prime }\right \rangle }{\sqrt{2}}\leq \left(
\frac{1+\left \vert \xi ^{\prime }\right \vert ^{2}}{2}+\frac{1+\left \vert
\xi _{\ast }^{\prime }\right \vert ^{2}}{2}\right) ^{1/2}=\left( 1+\frac{%
\left \vert \xi \right \vert ^{2}+\left \vert \xi _{\ast }\right \vert ^{2}}{%
2}\right) ^{1/2}\leq \left( 1+\left \vert \xi \right \vert ^{2}\right) ^{1/2}%
\text{,}
\end{equation*}%
which implies that $\nu \left( \xi \right) \lesssim \left \langle \xi
^{\prime }\right \rangle ^{\gamma }$, $\left \langle \xi _{\ast }^{\prime
}\right \rangle ^{\gamma }$ (used in the proof of (\ref{gain-product}) ).
Hence,%
\begin{eqnarray}
&&\int_{\mathbb{S}^{2}}B(\vartheta )\left( \int_{I_{3}}\nu \left( \xi
\right) ^{-1}|\xi -\xi _{\ast }|^{\gamma }\exp (-\frac{\left \vert \xi
_{\ast }\right \vert ^{2}}{4})\left[ \left \vert g_{\ast }^{\prime }\right
\vert ^{2}\left \vert h^{\prime }\right \vert ^{2}+\left \vert g^{\prime
}\right \vert ^{2}\left \vert h_{\ast }^{\prime }\right \vert ^{2}\right]
d\xi _{\ast }d\xi \right) d\omega  \label{gain-3} \\
&\lesssim &\int_{\mathbb{S}^{2}}B(\vartheta )\left( \int_{I_{3}}\left[ \left
\vert g_{\ast }^{\prime }\right \vert ^{2}\left \vert h^{\prime }\right
\vert ^{2}+\left \vert g^{\prime }\right \vert ^{2}\left \vert h_{\ast
}^{\prime }\right \vert ^{2}\right] d\xi _{\ast }d\xi \right) d\omega  \notag
\\
&\lesssim &\int_{\mathbb{S}^{2}}B(\vartheta )\left( \int_{\psi \left(
I_{3}\right) }\left[ \left \vert g_{\ast }^{\prime }\right \vert ^{2}\left
\vert h^{\prime }\right \vert ^{2}+\left \vert g^{\prime }\right \vert
^{2}\left \vert h_{\ast }^{\prime }\right \vert ^{2}\right] d\xi _{\ast
}^{\prime }d\xi ^{\prime }\right) d\omega  \notag \\
&\lesssim &\int_{\mathbb{S}^{2}}B(\vartheta )\left( \int_{I_{3}}\left[ \left
\vert g_{\ast }\right \vert ^{2}\left \vert h\right \vert ^{2}+\left \vert
g\right \vert ^{2}\left \vert h_{\ast }\right \vert ^{2}\right] d\xi _{\ast
}d\xi \right) d\omega  \notag \\
&\lesssim &\left \vert g\right \vert _{L_{\xi }^{\infty }}^{2}\left \vert
h\right \vert _{L_{\xi }^{2}}^{2}+\left \vert g\right \vert _{L_{\xi
}^{2}}^{2}\left \vert h\right \vert _{L_{\xi }^{\infty }}^{2}\text{.}
\end{eqnarray}

Gathering (\ref{gain-1})-(\ref{gain-3}) yields%
\begin{equation}
\left( \int \left \vert \nu ^{-1}(\xi )\Gamma _{gain}(g,h)(\xi )\right \vert
^{2}d\xi \right) ^{1/2}\lesssim \left \vert g\right \vert _{L_{\xi }^{\infty
}}\left \vert h\right \vert _{L_{\xi }^{2}}+\left \vert g\right \vert
_{L_{\xi }^{2}}\left \vert h\right \vert _{L_{\xi }^{\infty }}\text{.}
\label{gain-L2}
\end{equation}%
Similarly, we have%
\begin{eqnarray}
&&\left( \int \left \vert \nu ^{-1/2}(\xi )\Gamma _{gain}(g,h)(\xi )\right
\vert ^{2}d\xi \right) ^{1/2}  \label{gain-product} \\
&\lesssim &\left( \int_{\mathbb{S}^{2}}B(\vartheta )\int_{I_{1}\cup
I_{2}\cup I_{3}}|\xi -\xi _{\ast }|^{\gamma }\exp (-\frac{\left \vert \xi
_{\ast }\right \vert ^{2}}{4})\left[ \left \vert g_{\ast }^{\prime }\right
\vert ^{2}\left \vert h^{\prime }\right \vert ^{2}+\left \vert g^{\prime
}\right \vert ^{2}\left \vert h_{\ast }^{\prime }\right \vert ^{2}\right]
d\xi _{\ast }d\xi d\omega \right) ^{1/2}  \notag \\
&\lesssim &\left \vert g\right \vert _{L_{\xi }^{\infty }}\left \vert
h\right \vert _{L_{\sigma }^{2}}+\left \vert g\right \vert _{L_{\sigma
}^{2}}\left \vert h\right \vert _{L_{\xi }^{\infty }}\text{,}  \notag
\end{eqnarray}%
by following the same argument.

As a consequence, combining (\ref{loss-L2}) and (\ref{gain-L2}), we obtain (%
\ref{Gamma-L2-2}). Combining (\ref{loss-product}) and (\ref{gain-product}),
we obtain (\ref{Gamma-product-1}) and thus
\begin{eqnarray*}
\left \vert \left \langle f,\Gamma (g,h)\right \rangle _{\xi }\right \vert
&\leq &\left( \int_{\mathbb{R}^{3}}\nu \left( \xi \right) \left \vert
f\right \vert ^{2}d\xi \right) ^{1/2}\left( \int_{\mathbb{R}^{3}}\nu
^{-1}\left( \xi \right) \left \vert \Gamma (g,h)\right \vert ^{2}d\xi
\right) ^{1/2} \\
&\lesssim &\left \vert f\right \vert _{L_{\sigma }^{2}}\left( \left \vert
g\right \vert _{L_{\xi }^{\infty }}\left \vert h\right \vert _{L_{\sigma
}^{2}}+\left \vert g\right \vert _{L_{\sigma }^{2}}\left \vert h\right \vert
_{L_{\xi }^{\infty }}\right) \text{.}
\end{eqnarray*}%
The proof is completed.$%
\hfill%
\square $

\textbf{Acknowledgment:}

Y.-C. Lin is supported by the Ministry of Science and Technology under the
grant MOST 110-2115-M-006-004-. H.T. Wang is supported by National Nature Science Foundation of China under Grant No. 11901386 and 12031013, the Strategic Priority Research Program of Chinese Academy of Sciences under Grant No. XDA25010403. K.-C. Wu is supported by the Ministry of
Science and Technology under the grant 110-2636-M-006-005- and National
Center for Theoretical Sciences.

\end{document}